\documentclass[10.9pt,a4paper]{amsart}
\usepackage{a4wide}

\usepackage[utf8]{inputenc}
\usepackage[english]{babel}
\usepackage{comment}

\usepackage{amsfonts,amssymb,amsmath,xcolor,pinlabel}
\usepackage[nobysame,alphabetic, initials]{amsrefs}
\usepackage{enumerate}

\usepackage{xcolor}
\usepackage[colorlinks,
    linkcolor={red!50!black},
    citecolor={blue!50!black},
    urlcolor={blue!80!black}]{hyperref}

\definecolor{darkgreen}{rgb}{0,.5,0}
\definecolor{darkpurple}{rgb}{.5,0,.5}

\usepackage{tikz-cd}
\usepackage[all]{xy}
\usetikzlibrary{calc, matrix, arrows, cd}

\newcommand{\bsm}{\left(\begin{smallmatrix}}
\newcommand{\esm}{\end{smallmatrix}\right)}
\newtheorem{innercustomthm}{Theorem}
\newenvironment{customthm}[1]
  {\renewcommand\theinnercustomthm{#1}\innercustomthm}
  {\endinnercustomthm}

\newtheorem{theorem}{Theorem}[section]

\newtheorem{lemma}[theorem]{Lemma}
\newtheorem{proposition}[theorem]{Proposition}

\newtheorem{question}[theorem]{Question}

\theoremstyle{definition}
\newtheorem{definition}[theorem]{Definition}

\newtheorem{example}[theorem]{Example}
\newtheorem{observation}[theorem]{Observation}

\newtheorem{remark}[theorem]{Remark}
\newtheorem{construction}[theorem]{Construction}
\newtheorem{convention}[theorem]{Convention}
\newtheorem{notation}[theorem]{Notation}
\newtheorem{claim}{Claim}
\newtheorem*{claim*}{Claim}

\newcommand{\N}{\mathbb{N}}
\newcommand{\Z}{\mathbb{Z}}
\newcommand{\Q}{\mathbb{Q}}

\newcommand{\C}{\mathbb{C}}

\newcommand{\im}{\operatorname{im}}
\newcommand{\PD}{\operatorname{PD}}
\newcommand{\rk}{\operatorname{rk}}
\newcommand{\ev}{\operatorname{ev}}
\newcommand{\BS}{\operatorname{BS}}
\newcommand{\aug}{\operatorname{aug}}
\newcommand{\Hom}{\operatorname{Hom}}

\newcommand{\id}{\operatorname{id}}

\newcommand{\Homeo}{\operatorname{Homeo}}
\newcommand{\Bl}{\operatorname{Bl}}
\newcommand{\core}{\operatorname{core}}
\newcommand{\cocore}{\operatorname{cocore}}
\newcommand{\Aut}{\operatorname{Aut}}

\newcommand{\unaryminus}{\scalebox{0.75}[1.0]{\( - \)}}

\newcommand{\obar}[1]{\mkern 1.5mu\overline{\mkern-1.5mu#1\mkern-1.5mu}\mkern 1.5mu}
\newcommand{\ubar}[1]{\mkern 1.5mu\underline{\mkern-1.5mu#1\mkern-1.5mu}\mkern 1.5mu}
\newcommand{\Vtu}{\smash{\obar{V}_0^{\tau}}}
\newcommand{\Vtl}{\smash{\ubar{V}_0^{\tau}}}
\newcommand{\Vitu}{\smash{\obar{V}_0^{\iota\tau}}}
\newcommand{\Vitl}{\smash{\ubar{V}_0^{\iota\tau}}}

\newcommand{\tmfrac}[2]{\mbox{\large$\frac{#1}{#2}$}}

\newcommand{\cpring}{(\mathbb{C}P^2)^\circ}

\usepackage{stackengine}

\begin{document}
\title{$\Z$-disks in $\C P^2$}

\author[A.~Conway]{Anthony Conway}
\address{The University of Texas at Austin, Austin TX 78712}
\email{anthony.conway@austin.utexas.edu}
\author[I.~Dai]{Irving Dai}
\address{The University of Texas at Austin, Austin TX 78712}
\email{irving.dai@math.utexas.edu}
\author[M.~Miller]{Maggie Miller}
\address{The University of Texas at Austin, Austin TX 78712}
\email{maggie.miller.math@gmail.com}

\begin{abstract}
We study 
locally flat disks in~$\cpring:=(\C P^2)\setminus \mathring{B}^4$ with boundary a fixed knot~$K$ and whose complement has fundamental group $\mathbb{Z}$.
We show that up to topological isotopy rel.\ boundary, such disks necessarily arise by performing a positive crossing change on $K$ to an Alexander polynomial one knot and capping off with a $\Z$-disk in $D^4.$ 
Such a crossing change determines a loop in $S^3 \setminus  K$ and we prove that the homology class of its lift to the infinite cyclic cover leads to a complete invariant of the disk. 
We prove that this determines a bijection between the set of rel.\ boundary topological isotopy classes of~$\Z$-disks with boundary $K$ and a quotient of the set of unitary units of the ring~$\Z[t^{\pm 1}]/(\Delta_K)$.
Number-theoretic considerations allow us to deduce that a knot~$K \subset S^3$ with quadratic Alexander polynomial bounds $0,1,2,4$, or infinitely many $\Z$-disks in $\cpring$.
This leads to the first examples of knots bounding infinitely many topologically distinct disks whose exteriors have the same fundamental group and equivariant intersection form.
Finally we give several examples where these disks are realized smoothly.
\end{abstract}
\maketitle

\section{Introduction}

Freedman proved that a knot $K \subset S^3$ bounds a locally flat disk $D \subset D^4$ with~$\pi_1(D^4 \setminus D) \cong \Z$ if and only if $K$ has Alexander polynomial one~\cite{Freedman}.
It is additionally known that if $K$ bounds such a disk, then it is unique up to 
isotopy rel.\  boundary~\cite{ConwayPowellDiscs}.
In other words, the set $\mathcal{D}_\Z(K,D^4)$ of  rel.\  boundary isotopy classes of~\emph{$\Z$-disks} in $D^4$ with boundary $K$ is nonempty if and only if $\Delta_K \doteq 1$ in which case it contains a single element.

This article describes the classification of~$\Z$-disks in $\cpring:=\C P^2 \setminus \mathring{B}^4$ with boundary a knot~$K$: we list necessary and sufficient conditions for the existence of such disks (Theorem~\ref{thm:ExistenceIntro}) and construct an explicit bijection from~$\mathcal{D}_\Z(K,\cpring)$ to a subset of the Alexander module (Theorem~\ref{thm:DiscsWithoutCPP22Intro}). 
The outcome is quite different than in $D^4$: when $\mathcal{D}_\Z(K,\cpring)$ is nonempty,  it rarely consists of a single element and we use some number theory to show that if $K$ has trivial or quadratic Alexander polynomial (e.g.\ if $K$ has genus one), then~$\mathcal{D}_\Z(K,\cpring)$ must have cardinality $0,1,2,4$, or be infinite~(Theorem~\ref{thm:UniquenessIntro}).

\medbreak

In what follows, a $4$-manifold is understood to mean a compact, connected, oriented, topological $4$-manifold and embeddings are understood to be locally flat.
Knots are assumed to be oriented.
Given a closed simply-connected $4$-manifold $X$,  we write $X^\circ:=X \setminus \mathring{B}^4$, and a \emph{$\Z$-disk in $X^\circ$} refers to a properly embedded disk~$D \subset X^\circ$ whose complement has fundamental group~$\Z$.
The knot~$K=\partial D$ is then called \emph{$\Z$-slice} in $X^\circ$.

\subsection{Existence of $\Z$-disks in $\cpring$.}
\label{sub:ExistenceIntro}

Necessary and sufficient criteria for deciding whether a knot $K$ is $\Z$-slice in $\cpring$ can be deduced from work of Borodzik and Friedl~\cite{BorodzikFriedlClassical1}.
To state this result,  use $E_K$ to denote the exterior of~$K$, and recall that the Blanchfield form of~$K$ is a nonsingular, sesquilinear, Hermitian form
$$ \Bl_K \colon H_1(E_K^\infty) \times H_1(E_K^\infty) \to \Q(t)/\Z[t^{\pm 1}]$$
which, roughly speaking,  measures equivariant linking in the infinite cyclic cover $E_K^\infty$ of $E_K$.
We say that $\Bl_K$ is \emph{represented} by a size $n$ Hermitian matrix~$A$ with $\det(A) \neq 0$ and coefficients in~$\Z[t^{\pm 1}]$ if~$\Bl_K$ is isometric to the following linking form:
\begin{align}
\label{eq:presents}
\Z[t^{\pm 1}]^n/A^T\Z[t^{\pm 1}]^n \times \Z[t^{\pm 1}]^n/A^T\Z[t^{\pm 1}]^n  \to \Q(t)/\Z[t^{\pm 1}] \\
([x],[y]) \mapsto -x^TA^{-1}\overline{y}. \nonumber
\end{align}
If such a matrix has size $1$, then it is a symmetric polynomial $p(t)$ and in this case, we say that~$\Bl_K$ is \emph{represented by $p(t)$}.
Here and in what follows,  we write $\Delta_K$ for the unique symmetric representative of the Alexander polynomial of $K$ that evaluates to $-1$ at $t=1$; the reason for this convention will be elucidated in Section~\ref{sub:Genus1}. Recalling that we are working in the topological category,
$\Z$-sliceness in~$\cpring$ can be characterised as follows; see Section~\ref{sec:Existence} for the proof.

\begin{theorem}
\label{thm:ExistenceIntro}
Given a knot~$K$, the following assertions are equivalent:
\begin{enumerate}
\item $K$ is~$\Z$-slice in~$\cpring$;
\item the Blanchfield form~$\Bl_K$ is presented by~$\unaryminus\Delta_K(t)$;
\item $K$ can be converted into an Alexander polynomial one knot by switching a single positive crossing to a negative crossing;
\item $K$ can be converted into an Alexander polynomial one knot by a single positive generalized crossing change.
\end{enumerate}
\end{theorem}

Here, given a curve $\gamma \subset S^3 \setminus K$ with $\ell k(K,\gamma)=0$ that is unknotted in~$S^3$,  we say that a knot~$K'$ in~$S^3$ is obtained from $K$ by a {\emph{positive generalized crossing change}} along~$\gamma$ if~$(S^3_{-1}(\gamma),K)\cong(S^3,K')$.
As  illustrated in Figure \ref{fig:generalizedcrossingchange}, when only two strands of $K$ are involved, we omit the word ``generalized".
We note that we are also considering~$K \subset S^3$ as a knot in~$S^3\setminus\nu(\gamma)$ and hence in~$S^3_{-1}(\gamma)$.
We emphasize that the last three items of Theorem~\ref{thm:ExistenceIntro} were already known to be equivalent by work of Borodzik and Friedl~\cite{BorodzikFriedlLinking}.

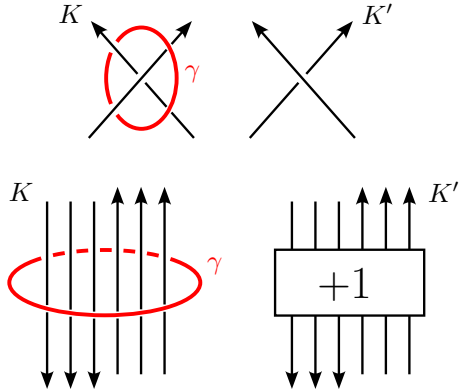
\begin{figure}
\labellist
\pinlabel{$K$} at 23 137
\pinlabel{$K'$} at 134 137
\pinlabel{\textcolor{red}{$\gamma$}} at 67 115
\pinlabel{$K$} at 5 72
\pinlabel{$K'$} at 158 72
\pinlabel{\textcolor{red}{$\gamma$}} at 75 45
\pinlabel{\huge{$+1$}} at 122 38
\endlabellist
\vspace{.1in}
    \includegraphics[width=55mm]{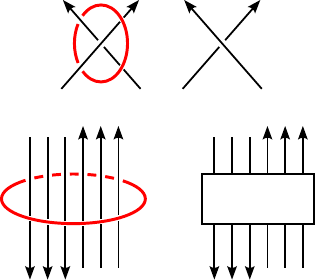}
    \caption{Top: a positive crossing change from $K$ to $K'$ realized as a (generalized) crossing change about a curve $\gamma$. Bottom: A generalized positive crossing change about $\gamma$ transforms $K$ into $K'$. In this example, $\gamma$ links $K$ geometrically more than two times (although as always links $K$ zero times algebraically).}\label{fig:generalizedcrossingchange}
\end{figure}

In summary,  Theorem~\ref{thm:ExistenceIntro} describes conditions for obstructing a knot from bounding a $\Z$-disk in $\cpring$ as well as ways to construct $\Z$-slice knots in~$\cpring$.
In Section~\ref{sub:Context} we list related work on the existence question for $\Z$-surfaces in $4$-manifolds with boundary $S^3$ such as~\cite{FellerLewarkOnClassical,FellerLewarkBalanced,
KjuchukovaMillerRaySakalli,ConwayPiccirilloPowell} but first we discuss the classification of $\Z$-disks in $\cpring$.

%
%

\subsection{Classification of $\Z$-disks in $\cpring$}

Assume now that a generalized crossing change along the \emph{surgery curve} $\gamma$ 
results in an Alexander polynomial one knot.
As we recall in more detail in Section~\ref{sec:CrossingChange},  one can associate to~$K$ and $\gamma$ a $\Z$-disk in~$\cpring$, that we refer to as a \emph{generalized-crossing-change~$\Z$-disk}.

From now on,  we fix a basepoint $z \in \partial E_K$ and a lift $\widetilde{z} \in \partial E_K^\infty$ of $z$ to the infinite cyclic cover.
When we refer to loops and lifts,  it is with respect to these basepoints.
We note in Proposition~\ref{prop:BelongstoG} below that the homology class of the lift to~$E_K^\infty$ of the surgery curve~$\gamma \subset E_K$ belongs to
\begin{equation}
\label{eq:GKIntro}
 \mathcal{G}_K:= \left\lbrace x \in H_1(E_K^\infty) \  \Big|  \ x \text{  is a generator and } \Bl_K(x,x)
=\frac{1}{\Delta_K} \right\rbrace.
\end{equation}
Observe that the set $\lbrace t^k \rbrace_{k \in \Z}$ acts on $ \mathcal{G}_K$ by multiplication.
The following theorem (which,  recall, takes place in the topological category) gives a complete characterization of $\Z$-disks in $\cpring$ as well as a complete invariant; see Section~\ref{sec:ClassifyCrossingChange} for the proof.

\newpage

\begin{theorem}
\label{thm:DiscsWithoutCPP22Intro}
Let $K$ be a knot.
\begin{enumerate}
\item  Every~$\Z$-disk in~$\cpring$ with boundary $K$ is isotopic rel.\ boundary to a crossing change~$\Z$-disk.
\item Two generalized crossing change~$\Z$-disks with boundary $K$ are isotopic rel.\  boundary if and only if the homology classes of the lifts of their surgery curves agree in the Alexander module $H_1(E_K^\infty)$ up to multiplication by $t^k$ for some~$k \in \Z$.
\item Mapping a generalized crossing change $\Z$-disk with boundary $K$ to the homology class of the lift of its surgery curve defines a bijection
$$\mathcal{D}_\Z(K,\cpring) \xrightarrow{\cong} \mathcal{G}_K/\lbrace t^k \rbrace_{k \in \Z}.$$
\end{enumerate}
\end{theorem}

In summary, Theorem~\ref{thm:DiscsWithoutCPP22Intro} shows that all $\Z$-disks in $\cpring$ arise as crossing change $\Z$-disks and that the lift of the surgery curve leads to a complete invariant.
In Section~\ref{sub:Context} we compare Theorem~\ref{thm:DiscsWithoutCPP22Intro} to the classification of~$\Z$-surfaces obtained in~\cite{ConwayPiccirilloPowell} but first we describe how the bijection with~$\mathcal{G}_K/\lbrace t^k \rbrace_{k \in \Z}$ allows us to explicitly enumerate $\Z$-disks when $K$ has trivial or quadratic Alexander polynomial.

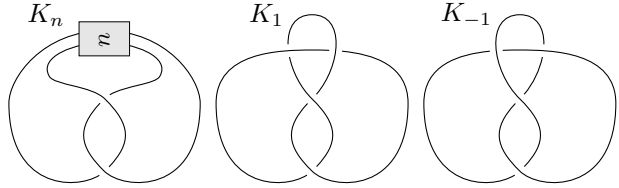
\begin{figure}[!htbp]
\centering
\labellist
\pinlabel{$K_n$} at 20 90
\pinlabel{$K_1$} at 137 90
\pinlabel{$K_{-1}$} at 243 90
\small
\pinlabel{\rotatebox{90}{$n$}} at 51 77
\endlabellist
\includegraphics[width=80mm]{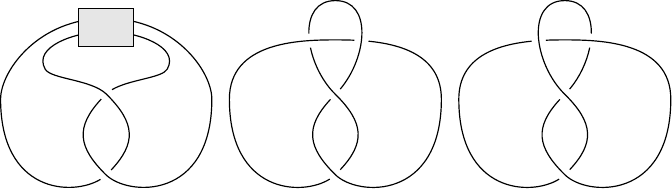}
\caption{{\bf{Left:}} The twist knot $K_n$ where the box denotes $n$ full twists. {\bf{Middle:}} The figure eight knot $K_1$. {\bf{Right:}} The right-handed trefoil $K_{-1}$.}\label{fig:twistknot}
\end{figure}

\subsection{Knots with quadratic Alexander polynomial: an explicit enumeration}
\label{sub:Genus1}
Alexander polynomial one knots bound a unique $\Z$-disk in $\cpring$ up to isotopy rel.\ boundary; this is a consequence of Theorem~\ref{thm:DiscsWithoutCPP22Intro} but also follows from the combination of~\cite[Theorem 1.2]{ConwayPowell} and~\cite[Theorem C]{OrsonPowell}.
When the polynomial is nontrivial, some number theory is needed to understand the cardinality of~$\mathcal{G}_K/\lbrace t^k \rbrace_{k \in \Z}$, as we now explain.
If the Alexander polynomial of a knot is trivial or quadratic (e.g. if $K$ is a genus one knot),  then up to multiplication by $\pm t^k$,  it is of the form
$$\Delta_n:=nt-(2n+1)+nt^{-1}$$
for some $n \in \Z$. 
Note that $\Delta_0=1, \Delta_{-1}=-t+1-t^{-1}$ is the Alexander polynomial of the trefoil,~$\Delta_1$ is the Alexander polynomial of the figure eight knot, and more generally, $\Delta_n$ is the Alexander polynomial of the twist knot $K_n$ with $n$ full twists illustrated in Figure \ref{fig:twistknot}.
Twists knots are genus one knots that are~$\Z$-slice in~$\cpring$ because they can be unknotted by switching a single positive crossing change; recall Theorem~\ref{thm:ExistenceIntro}.
We find it more convenient to index the~$\Delta_n$ using their leading coefficient which is why we work with a symmetric representative of the Alexander polynomial that satisfies~$\Delta_K(1)=-1$: this way $\Delta_{K_n}=\Delta_n$.

Our next result enumerates $\Z$-disks in $\cpring$ for knots with trivial or quadratic Alexander polynomial; its proof follows by combining Theorem~\ref{thm:DiscsWithoutCPP22Intro} with the number theoretic Theorem~\ref{thm:NumberTheoryIntro}.

\begin{theorem}
\label{thm:UniquenessIntro}
If a knot $K$ with trivial or quadratic Alexander polynomial is $\Z$-slice in~$\cpring$, then it bounds $1,2,4$ or infinitely many disks up to isotopy rel.\ boundary. More precisely, up to isotopy rel.\ boundary
$$
K \text{ bounds }
\begin{cases}
1 \text{ disk} & \quad \text{if $\Delta_K=\Delta_0,\Delta_{-1}$}, \\
2 \text{ disks} & \quad \text{if $\Delta_K=\Delta_1,\Delta_2$, or $\Delta_K=\Delta_n$ for $n=-p^k$ with $p$ prime and $k$ odd}, \\
4 \text{ disks} &  \quad \text{if $\Delta_K=\Delta_n$ for $n=-p^k$ with $p$ prime and $k$ even}, \\
\infty \text{ disks} & \quad \text{otherwise}.\\
\end{cases}
$$
\end{theorem}

We emphasize the difference with $\Z$-disks in $D^4$ where $|\mathcal{D}_\Z(K,D^4)| \in \lbrace 0,1\rbrace$: in contrast,  when~$\Delta_K$ is quadratic,  Theorem~\ref{thm:UniquenessIntro} shows that, generically,~$\mathcal{D}_\Z(K,\cpring)$ is infinite.

\begin{remark}
\label{rem:EquivalenceIntro}
Theorems~\ref{thm:DiscsWithoutCPP22Intro} and~\ref{thm:UniquenessIntro} admit analogues where the word ``isotopy" is replaced by ``equivalence".
To form the analogue of Theorem~\ref{thm:DiscsWithoutCPP22Intro}, one substitutes multiplication by $t^k$ for multiplication by $\pm t^k$, whereas for Theorem~\ref{thm:UniquenessIntro}, the outcome is that if $K$ is a knot with trivial or quadratic Alexander polynomial that is $\Z$-slice in $\cpring$, then, up to equivalence rel.\ boundary,
$$
K \text{ bounds }
\begin{cases}
1 \text{ disk} & \quad \text{if $\Delta_K=\Delta_0,\Delta_{-1},\Delta_1,\Delta_2$}, \\
& \quad \text{or $\Delta_K=\Delta_n$ for $n=-p^k$ with $p$ prime and $k$ odd}, \\
2 \text{ disks} &  \quad \text{if $\Delta_K=\Delta_n$ for $n=-p^k$ with $p$ prime and $k$ even}, \\
\infty \text{ disks} & \quad \text{otherwise}.\\
\end{cases}
$$
We refer to Remarks~\ref{rem:NumberTheoryEquivalence} and~\ref{rem:Equivalence} for further details.
\end{remark}

We describe some further context surrounding Theorem~\ref{thm:UniquenessIntro} in Section~\ref{sub:Context} but presently we discuss examples of Theorem~\ref{thm:UniquenessIntro} with an emphasis on realizing the disks smoothly.

As we explain in Section~\ref{sec:CrossingChange}, the construction of the $\Z$-disk $D_\gamma$ associated to a generalized crossing change curve $\gamma$ depends on a choice of orientation for $\gamma$. Reversing the orientation on $\gamma$ results in a disk $D_{-\gamma}$ which is equivalent, but not necessarily isotopic, to the original rel.\ boundary. 
Each generalized crossing change curve thus gives two $\Z$-disks which are equivalent rel.\ boundary, but usually distinct in the sense of Theorem~\ref{thm:UniquenessIntro}. 
For knots with trivial or quadratic Alexander polynomial, we note in Remark~\ref{rem:pmdistinctExampleSection} that $D_{\pm \gamma}$ are isotopic rel.\ boundary if and only if~$\Delta_K = \Delta_0$ or $\Delta_{-1}$.

\begin{example}
\label{ex:Genus1Enumeration}
We illustrate Theorem~\ref{thm:UniquenessIntro} using twist knots and include several examples of explicit smooth disks properly embedded in~$\cpring$.
The point here is that if $K$ can be unknotted by a single positive generalized crossing change, then the associated $\Z$-disk is smoothly embedded.
\begin{itemize}
\item The trefoil~$K_{-1}$ bounds a unique $\Z$-disk in $\cpring$ up to isotopy rel.\  boundary. 
As we discuss in Example~\ref{ex:Trefoil}, this disk can be smoothly realized as a crossing change $\Z$-disk $D_\gamma$.
We explicitly show that in this case, the disks $D_{\pm \gamma}$ are smoothly isotopic rel.\ boundary. 
\item The figure eight knot $K_1$ and the stevedore knot $K_2$ each bound two $\Z$-disks up to isotopy rel.\  boundary and a single disk up to equivalence rel.\  boundary. 
Example~\ref{ex:FigureEight} focuses on the figure eight knot and shows that  the disks are smoothly realized by crossing change~$\Z$-disks~$D_{\pm \gamma}$.
\item The twist knot $K_{-4}$ bounds four $\Z$-disks in $\cpring$ up to isotopy rel.\  boundary and two~$\Z$-disks up to equivalence rel.\  boundary. 
As we discuss in Example~\ref{ex:knsquared}, these disks are smoothly realized by generalized crossing change $\Z$-disks $D_{\pm \gamma_1},D_{\pm \gamma_2}$.
\item More generally, for $n = -k^2$, $k$ a prime power, the twist knot $K_n$ bounds four $\Z$-disks in~$\cpring$ up to isotopy rel.\  boundary and two $\Z$-disks up to equivalence rel.\  boundary. 
In Example~\ref{ex:knsquared}, we realize two isotopy classes smoothly as crossing change $\Z$-disks~$D_{\pm \gamma_1}$ with~$\gamma_1$ an unknotting curve, 
and describe the other two disks as $D_{\pm \gamma_2}$ with $\gamma_2$ a generalized unknotting curve to an Alexander polynomial knot.
We do not know if all four isotopy classes can be realized smoothly. The construction also holds for $k$ not a prime power, but in this case $K_n$ has infinitely many $\Z$-disks, out of which we have only attempted to describe four.
\item Finally, for $n = -k^2$, $k$ a prime power, the $n$-twisted negative Whitehead double of the torus knot $T(-k,k-1)$ bounds four $\Z$-disks in $\cpring$ up to isotopy rel.\  boundary and two $\Z$-disks up to equivalence rel.\  boundary. 
In Example~\ref{ex:toruswhitehead},  we realize all of these smoothly via a pair of distinct generalized unknotting curves (together with their orientation reversals). The construction also holds for $k$ not a prime power, but in this case the knot in question has infinitely many $\Z$-disks, out of which we have only attempted to describe four.
\end{itemize}
This shows that each case of Theorem~\ref{thm:UniquenessIntro} where $K$ bounds a finite number of $\Z$-disks can be illustrated by an example in which all isotopy classes can be realized smoothly.
\end{example}

We further illustrate Theorems~\ref{thm:ExistenceIntro} and~\ref{thm:UniquenessIntro} by calculating the number~$d_K$ of $\Z$-disks in $\cpring$ bounded by a knot $K$.
Thanks to Theorem~\ref{thm:ExistenceIntro},  we know that if a knot or its mirror is $\Z$-slice in $\cpring$,  then its Blanchfield form is presented by a size one nondegenerate Hermitian matrix.
Work of Borodzik-Friedl~\cite[Theorem 1.1]{BorodzikFriedlLinking} therefore implies that if a knot $K$ or its mirror is $\Z$-slice in $\cpring$,  then  $u_a(K)=1$, where $u_a(K)$ denotes the algebraic unknotting number of~$K$.
The algebraic unknotting number is listed on KnotInfo~\cite{Knotinfo}.
\begin{example}
We list the number $d_K$ of $\Z$-disks in $\cpring$ up to isotopy for knots $K$ up to~$5$ crossings.
As mentioned in Example~\ref{ex:Genus1Enumeration}, the unknot and the trefoil satisfy $d_K=1$, whereas the figure eight has $d_K=2$.
Since the cinquefoil $K=5_1$ satisfies $u_a(K)=2$, we have $d_K=0$.
Finally, the knot $K=5_2$ is $\Z$-slice in~$\cpring$ with $\Delta_{K}=-2t+3-2t^{-1}=\Delta_{-2}$ so $d_K=2$.
\end{example}

\subsection{The underlying number theory}
\label{sub:NumberTheoryIntro}

Next we state the number theoretic result that underlies Theorem~\ref{thm:UniquenessIntro}.
First we recast~$\mathcal{G}_K$ in a more algebraic way.
Assume that~$K$ is a knot that is~$\Z$-slice in $\cpring$ so that by Theorem~\ref{thm:ExistenceIntro} its Alexander module is cyclic (i.e. ~$H_1(E_K^\infty)\cong \Z[t^{\pm 1}]/(\Delta_K)$) and $\Bl_K(x,y)=\frac{x\overline{y}}{\Delta_K}$ for every $x,y \in H_1(E_K^\infty)$.
Here $x \mapsto \overline{x}$ denotes the involution $p(t) \mapsto p(t^{-1}).$
This implies that polynomial multiplication induces a free and transitive action of
$$
U\left(\frac{\Z[t^{\pm 1}]}{(\Delta_K)}\right):=
\left\lbrace u \in \frac{\Z[t^{\pm 1}]}{(\Delta_K)}  \  \Big|  \  u\overline{u}=1 \right\rbrace$$
on $\mathcal{G}_K$.
In particular there is a noncanonical bijection
$$ \mathcal{G}_K \xrightarrow{\approx} U\left(\frac{\Z[t^{\pm 1}]}{(\Delta_K)}\right).$$
The general setting is as follows.
Given a ring~$R$ with involution~$x \mapsto \overline{x}$,  the \emph{group of unitary units}~$U(R)$ refers to those~$x \in R$ such that~$x \overline{x}=1$.
For example, when~$R=\Z[t^{\pm 1}]$ with the involution $p(t) \mapsto p(t^{-1})$, all units are unitary and are of the form~$\pm t^{k}$ with~$k \in \Z$.
More generally, if $R=\Z[t^{\pm 1}]/p(t)$ for some $p(t) \in \Z[t^{\pm 1}]$, then the group  $U(R)$ always contains the image of the set~$\lbrace \pm  t^k \rbrace_{k \in \Z} \subset \Z[t^{\pm 1}]$ under the projection map~$\Z[t^{\pm 1}] \to \Z[t^{\pm 1}]/p(t).$

Theorem~\ref{thm:UniquenessIntro} therefore follows from Theorem~\ref{thm:DiscsWithoutCPP22Intro} and the study of 
\color{black}
 $U(\Z[t^{\pm 1}]/p(t))/\lbrace  t^k\rbrace_{k \in \Z}$ for quadratic Laurent polynomials $p(t)$ that arise as Alexander polynomials,
i.e.  for the polynomials~$\Delta_n=nt-(2n+1)+nt^{-1}$.
We write 
$$U_n:=U\left(\frac{\Z[t^{\pm 1}]}{(\Delta_n)}\right)$$
as a shorthand for the set of unitary units of $\Z[t^{\pm 1}]/(\Delta_n)$.
In order to prove Theorem~\ref{thm:UniquenessIntro}, it suffices to determine for what values of $n$ the set $U_n/\lbrace  t^k \rbrace_{k \in \Z}$ is infinite and to determine its cardinality when it is finite. 
The next theorem (which is proved in Section~\ref{sec:UnitaryUnits}) achieves slightly more as it also determines the rank of~$U_n/\lbrace  t^k \rbrace_{k \in \Z}$ in the infinite case in terms of the number $\Omega(n)$ of positive primes dividing $n$.

\begin{theorem}
\label{thm:NumberTheoryIntro}
For every~$n \in \Z$, the group~$U_n/\{t^k\}_{k \in \Z}$ can be described as follows.
\begin{enumerate}
\item The group~$U_n/\{t^k\}_{k \in \Z}$ is finite precisely when~$n = 2, 1, 0, -1$, or~$-p^k$ for a prime~$p$, and 
$$
 U_n/\{t^k\}_{k \in \Z} 
\cong \begin{cases}
\lbrace 1 \rbrace & \quad \text{ if $n=-1,0$},  \\
\Z/2\Z & \quad \text{ if $n=1,2$ or $n=-p^k$ with $k$ odd},  \\
\Z/4\Z & \quad \text{ if $n=-p^k$ with $k$ even}. \\
\end{cases}
$$
\item In all other cases, ~$U_n/\{t^k\}_{k \in \Z}$ is infinite and 
$$
\rk \left( U_n/\{t^k\}_{k \in \Z} \right)
=\begin{cases}
\Omega(n) & \quad \text{ if $\Delta_n$ is irreducible and $n>0$},  \\
\Omega(n)-1 & \quad \text{ if $\Delta_n$ is irreducible and $n<0$ or if $\Delta_n$ is reducible.}
\end{cases}
$$
\end{enumerate}
\end{theorem}

We use the convention that strictly negative rank should be interpreted as zero rank; we write this for simplicity of the theorem statement. 
The cases in which~$(2)$ give a rank less than or equal to zero correspond precisely to the cases in~$(1)$; see Remark~\ref{rem:Reducibility}.

\begin{remark}
\label{rem:NumberTheoryEquivalence}
As explained in Remark~\ref{rem:pmextension}, Theorem~\ref{thm:NumberTheoryIntro} can be modified to obtain the following description of $U_n/\{ \pm t^k\}_{k \in \Z}$.
\begin{enumerate}
\item The group~$U_n/\{\pm t^k\}_{k \in \Z}$ is finite precisely when~$n = 2, 1, 0, -1$, or~$-p^k$ for a prime~$p$, and
$$
U_n/\{\pm t^k\}_{k \in \Z} 
\cong\begin{cases}
\lbrace 1 \rbrace & \quad \text{ if $n=-1,0,1,2$ or $n=-p^k$ with $k$ odd},  \\
\Z/2\Z & \quad \text{ if $n=-p^k$ with $k$ even}. \\
\end{cases}
$$
\item In all other cases, ~$U_n/\{ \pm t^k\}_{k \in \Z}$ is infinite and 
$$
\rk \left( U_n/\{ \pm t^k\}_{k \in \Z} \right)
=\begin{cases}
\Omega(n) & \quad \text{ if $\Delta_n$ is irreducible and $n>0$},  \\
\Omega(n)-1 & \quad \text{ if $\Delta_n$ is irreducible and $n<0$ or if $\Delta_n$ is reducible.}
\end{cases}
$$
\end{enumerate}
As noted in Remark~\ref{rem:EquivalenceIntro},  for $K$ a knot with trivial or quadratic Alexander polynomial,  this result leads to the classification of $\Z$-disks in $\cpring$  with boundary $K$ up to equivalence rel.\ boundary (instead of isotopy rel.\ boundary).
\end{remark}

\subsection{Challenges in the smooth category} 

We now briefly discuss the results of Theorem~\ref{thm:UniquenessIntro} in the context of the smooth category. 
The most immediate question in this regard is to determine which of the isotopy classes in Theorem~\ref{thm:UniquenessIntro} are realized by smooth disks.
As we have seen in Example~\ref{ex:Genus1Enumeration}, there are knots illustrating the first through third cases of Theorem~\ref{thm:UniquenessIntro} for which all isotopy classes are realized smoothly. 
 
 \begin{question}
 Is there a genus one knot with an infinite number of $\Z$-disks in $\cpring$ such that each class is realized smoothly?
 \end{question}
 
\begin{question}
Is there a genus one knot such that only a strict, non-empty subset of the isotopy classes from Theorem~\ref{thm:UniquenessIntro} have a smooth representative?
 \end{question}


Recall that two smoothly embedded disks in a $4$-manifold with boundary~$S^3$ are \emph{exotic} if they are topologically but not smoothly isotopic rel.\ boundary. 
Exotic disks are known to exist in~$D^4$~\cite{AkbulutZeeman,Hayden} (see also~\cite{HaydenSundberg,DaiMallickStoffregen}) as well as in many $2$-handlebodies with~$S^3$ boundary,  including~$\cpring$~\cite[Theorem 1.13]{ConwayPiccirilloPowell}.
On the other hand,  given a knot~$K$ that is smoothly slice in a $4$-manifold~$N$, it remains challenging to decide whether or not $K$ bounds any exotic slice disks in $N$.
Theorem~\ref{thm:DiscsWithoutCPP22Intro} implies that two smooth crossing change disks $D_{\gamma_1}$ and $D_{\gamma_2}$ with boundary a knot~$K$ are exotic if they are smoothly distinct and the classes~$[\widetilde{\gamma}_1]$ and~$[\widetilde{\gamma}_2]$ agree in~$\mathcal{G}_K/\lbrace t^k \rbrace_{k \in \Z}$. 

One broad approach towards generating pairs of exotic disks would thus be to choose a knot~$K$ falling into the first three cases of Theorem~\ref{thm:UniquenessIntro} and find several distinct unknotting curves that turn $K$ into a smoothly $\Z$-slice knot. If sufficiently many curves are found, then certainly~$[\widetilde{\gamma}_1]=[\widetilde{\gamma}_2] \in \mathcal{G}_K/\lbrace t^k \rbrace_{k \in \Z}$ for some pair $(\gamma_1, \gamma_2)$. 
As we describe in Section~\ref{sec:exotic}, one can then attempt to leverage smooth invariants such as knot Floer homology to distinguish the corresponding pair of disks (see for example \cite{DaiMallickStoffregen}).

For instance, the trefoil bounds a unique crossing change $\Z$-disk in $\cpring$ up to topological isotopy rel.\ boundary. Thus, any pair of smooth $\Z$-disks for the trefoil which can be distinguished in the smooth category would constitute an exotic pair. Unfortunately, the authors have not been able to find any such candidates for exotic pairs of disks, although a systematic investigation of this strategy is beyond the scope of this paper. See Section~\ref{sec:exotic} for further discussion. 

 \begin{question}
Does the right-handed trefoil bound an exotic pair of $\Z$-disks in $\cpring$?
 \end{question}

\subsection{Further context and related work}
\label{sub:Context}

Let $N$ be a  simply-connected $4$-manifold with boundary $S^3$, and let~$K \subset S^3$ be a knot.
$\Z$-surfaces with boundary $K$ considered up to equivalence rel.\  boundary were classified in~\cite{ConwayPowell,ConwayPiccirilloPowell}.
Complete invariants are given by the equivariant intersection form and automorphism invariant of their exterior.
Prior work on the existence question includes~\cite{FellerLewarkOnClassical,FellerLewarkBalanced} (for $\Z$-sliceness in $N=D^4 \#_m S^2 \times S^2 \#_n \C P^2 \# \overline{\C P}^2)$ and~\cite{KjuchukovaMillerRaySakalli} (for $\Z$-sliceness in~$N=(\#_\ell \C P^2)^\circ$).
Concerning this latter result, to the best of our knowledge~\cite[Theorem~1.3]{KjuchukovaMillerRaySakalli} does not recover Theorem~\ref{thm:ExistenceIntro}.
We emphasize again that Theorem~\ref{thm:ExistenceIntro} follows fairly promptly from~\cite{BorodzikFriedlLinking}.
We also note that since $\Z$-disks are nullhomologous (see e.g.~\cite[Lemma~5.1]{ConwayPowell}), Theorem~\ref{thm:ExistenceIntro} also yields criteria for a knot to be $H$-slice, a topic that has attracted some attention recently.

We now focus on $\Z$-disks.
In what follows, the exterior of a disk $D \subset N$ is denoted $N_D$.
Given a Hermitian form~$\lambda$ over $\Z[t^{\pm 1}]$, 
we write~$\mathcal{D}_\lambda(K,N)^{\operatorname{equiv}} $ for the set of rel.\ boundary equivalence classes of $\Z$-disks in $N$ with boundary~$K$ and equivariant intersection form~$\lambda_{N_D}\cong \lambda$; we refer to Section~\ref{sub:Equivariant} for a brief review of equivariant intersection forms.
The automorphism invariant from~\cite{ConwayPiccirilloPowell} gives rise to a (noncanonical) bijection
$$b \colon \mathcal{D}_\lambda(K,N)^{\operatorname{equiv}} \xrightarrow{\approx} \Aut(\Bl_K)/\Aut(\lambda). $$
The target of $b$ is the orbit set of the left action of the automorphism group $\Aut(\lambda)$~on the automorphism group~$\Aut(\Bl_K)$ by~$F \cdot h=h \circ\partial F^{-1}$; we refer to the introductions of~\cite{ConwayPowell,ConwayPiccirilloPowell} for more details. 

A statement involving isotopy classes of disks instead of equivalence classes of disks has not yet appeared in the literature
but can be obtained by tracing through the proofs of~\cite{ConwayPiccirilloPowell} and applying~\cite[Theorem C]{OrsonPowell} to upgrade equivalences to isotopies.

We now restrict to $N=\cpring$.
We observe in Proposition~\ref{prop:barCP2} that the exterior of a~$\Z$-disk  $D \subset \cpring$ with boundary $K$ necessarily has equivariant intersection form $\lambda_{N_D}(x,y)=x\Delta_K\overline{y}$.
This implies that~$\Aut(\Bl_K)/\Aut(\lambda) \approx U(\Z[t^{\pm 1}]/(\Delta_K))/\lbrace \pm t^k \rbrace_{k \in \Z}$ and so $b$ induces a bijection
$$b \colon \mathcal{D}(K,\cpring)^{\operatorname{equiv}} \xrightarrow{\approx}  \Aut(\Bl_K)/\Aut(\lambda) \xrightarrow{\approx} U(\Z[t^{\pm 1}]/(\Delta_K))/\lbrace \pm t^k \rbrace_{k \in \Z}. $$
Although this has not appeared in print,  if one traces through~\cite{ConwayPiccirilloPowell} and applies~\cite[Theorem~C]{OrsonPowell} when necessary, one can verify that $b$ induces a bijection 
\begin{equation}
\label{eq:CPPBijection}
b \colon \mathcal{D}(K,\cpring) \xrightarrow{\approx} U(\Z[t^{\pm 1}]/(\Delta_K))/\lbrace  t^k \rbrace_{k \in \Z}. 
\end{equation}
In practice,  the nature of the automorphism invariant $b_D$ can make it difficult to distinguish concrete $\Z$-disks in $\cpring$ with boundary $K$.

Thus the main novelty of Theorem~\ref{thm:DiscsWithoutCPP22Intro} resides less in the existence of a bijection as in~\eqref{eq:CPPBijection} than in the explicit nature of the invariant
as well as in the fact 
that every $\Z$-disk in $\cpring$ is isotopic to a crossing change~$\Z$-disk.
Another feature of Theorem~\ref{thm:DiscsWithoutCPP22Intro} is that while the bijection in~\eqref{eq:CPPBijection} can be derived by modifying the lengthy argument of~\cite[Section 5]{ConwayPiccirilloPowell}, our argument is more self-contained and instead relies on~\cite{BorodzikFriedlLinking}.
In summary, our main insight is that for generalized crossing change~$\Z$-disks,  the automorphism invariant boils down to the homology class of the lift of the surgery curve in the Alexander module.

\medbreak

Finally we describe some additional context surrounding Theorem~\ref{thm:UniquenessIntro}.
In~\cite{BoyerUniqueness,StongRealization,BoyerRealization,CrowleySixt,CCPSLong,ConwayPiccirilloPowell,ConwayCrowleyPowell}, isometries of linking forms have been used to study surfaces in $4$-manifolds as well as the (stable) classification of $4k$-manifolds.
A subset of the applications in these papers rely on being able to find examples of Hermitian forms~$(H,\lambda)$ such that the automorphism set~$\Aut(\partial \lambda)/\Aut(\lambda)$ is large; here $\partial \lambda$ denotes the boundary linking form of~$\lambda$, see e.g.~\cite[Section 2]{ConwayPowell}.

In~\cite{ConwayPiccirilloPowell},  examples were given of $3$-manifolds $Y$ that have arbitrarily large $\Aut(\Bl_Y)/\Aut(\lambda)$; this was later improved to examples with infinite~$\Aut(\Bl_Y)/\Aut(\lambda)$~\cite{ConwayCrowleyPowell}; we refer to~\cite{CCPSLong} for applications of this sort of algebra to the stable classification of $4k$-manifolds with $k>1$.
Prior to this article however,  few examples of knots $K$ had been produced for which~$\Aut(\Bl_K)/\Aut(\lambda)$ is nontrivial (i.e. for which the equivariant intersection form does not determine the isotopy type of the disk) and no example for which it is infinite.
Surprisingly, Theorem~\ref{thm:UniquenessIntro} suggests that the orbit set~$\Aut(\Bl_K)/\Aut(\lambda)$ might be in fact, generically, be infinite.

\subsection*{Organization}

Section~\ref{sec:Background} collects some background homological material.
In Section~\ref{sec:Existence} we prove Theorem~\ref{thm:ExistenceIntro} which concerns criteria for $\Z$-sliceness in~$\cpring$.
In Section~\ref{sec:CrossingChange} we study the properties of generalized crossing change $\Z$-disks.
In Section~\ref{sec:ClassifyCrossingChange} we prove the classification stated in Theorem~\ref{thm:DiscsWithoutCPP22Intro}.
Section~\ref{sec:Examples} is concerned with many examples and the smooth realization of $\Z$-disks.
In Section~\ref{sec:NumberTheory},  we review number theoretic background.
In Section~\ref{sec:UnitaryUnits},  we prove Theorem~\ref{thm:MainNumberTheory} which concerns the unitary units of~$\Z[t^{\pm 1}]/(\Delta_n)$.

\subsection*{Acknowledgments}

AC was partially supported by the NSF grant DMS\unaryminus 2303674. ID was partially supported by the NSF grant DMS\unaryminus 2303823. MM was partially supported by a Stanford Science Fellowship and a Clay Research Fellowship. The authors would like to thank Levent Alp\"oge for his significant help regarding the number-theoretic aspects of this paper.

\subsection*{Conventions}
 
We work in the topological category with locally flat embeddings unless otherwise stated.
From now on, all manifolds are assumed to be compact, connected, based and oriented; if a manifold has a nonempty, connected boundary, then the basepoint is assumed to be in the boundary.
The boundary of a manifold is oriented according to the ``outwards normal vector first" rule.
We write $\Delta_K$ for the unique symmetric representative of the Alexander polynomial of a knot~$K$ that evaluates to $-1$ at $t=1$.

\section{Background material}
\label{sec:Background}

We fix some notation concerning the (twisted) homology of infinite cyclic covers (Section~\ref{sub:Twisted}),  review the definition of the Blanchfield form (Section~\ref{sub:Blanchfield}) and of the equivariant intersection form (Section~\ref{sub:Equivariant}).
In the interest of expediency, we give the homological definitions of these pairings, referring to~\cite[Sections 2.1 and 2.2]{ConwayPiccirilloPowell} and to the references within for further details and geometric intuition.
We also refer to~\cite[Section 2]{FriedlPowell} for a treatment of twisted homology and the Blanchfield form that is similar to ours but contains further details.

\subsection{Twisted homology}
\label{sub:Twisted}

In what follows,  spaces are assumed to have the homotopy type of a finite CW complex.
Given a space~$X$ together with an epimorphism~$\varphi \colon \pi_1(X) \twoheadrightarrow \Z$, we write~$p\colon X^\infty \to X$ for the infinite cyclic cover corresponding to~$\ker(\varphi)$.
If~$A \subset X$ is a subspace, then we set~$A^\infty :=p^{-1}(A)$ and often write~$H_*(X,A;\Z[t^{\pm 1}])$ instead of~$H_*(X^\infty,A^\infty)$.
Note that these are finitely generated $\Z[t^{\pm 1}]$-modules because $X$ and $A$ have the homotopy type of finite CW-complexes and $\Z[t^{\pm 1}]$ is Noetherian, see e.g.~\cite[Proposition A.9]{FriedlNagelOrsonPowell}.

\begin{remark}
\label{rem:AlexanderPolynomial}
The \emph{Alexander polynomial} of~$X$ is the order of the \emph{Alexander module}~$H_1(X;\Z[t^{\pm 1}])$.
Note that~$\Delta_X$ is a Laurent polynomial that is well defined up to multiplication by~$\pm t^k$ with~$k \in \Z$.
If~$K$ is a knot with exterior~$E_K$ and~$0$-framed surgery~$M_K$,  and we take~$\varphi$ to be the abelianization homomorphism, then a Mayer-Vietoris calculation shows that the inclusion~$E_K \subset M_K$ induces a~$\Z[t^{\pm 1}]$-isomorphism~$H_1(E_K;\Z[t^{\pm 1}]) \to H_1(M_K;\Z[t^{\pm 1}])$.
The Alexander polynomial of~$K$, denoted~$\Delta_K$, can therefore equivalently be defined as the order of~$E_K$ or as the order of~$M_K$.
\end{remark}

We write $H^*(X,A;\Z[t^{\pm 1}])$ for the homology of $\Hom_{\Z[t^{\pm 1}]}(\overline{C_*(X^\infty,A^\infty)},\Z[t^{\pm 1}])$ and remind the reader that~$H^*(X,A;\Z[t^{\pm 1}])$ is isomorphic to the cohomology with compact support of the pair~$(X^\infty,A^\infty)$,  not to $H^*(X^\infty,A^\infty).$
Here,  given a $\Z[t^{\pm 1}]$-module $V$,  we write $\overline{V}$ for the $\Z[t^{\pm 1}]$-module whose underlying group agrees with that of $V$ but with the $\Z[t^{\pm 1}]$-module structure induced by $t \cdot v=t^{-1}v$ for $v \in V$.

More generally,
if $M$ is a $\Z[t^{\pm 1}]$-module, then we write $H_*(X,A;M)$ and $H^*(X,A;M)$ for the homology of the $\Z[t^{\pm 1}]$-chain complexes $M \otimes_{\Z[t^{\pm 1}]} C_*(X^\infty,A^\infty)$ and $\Hom_{\Z[t^{\pm 1}]}(\overline{C_*(X^\infty,A^\infty)},M)$, respectively.
Apart from the case~$M=\Z[t^{\pm 1 }]$ which we have already considered, we will occasionally consider the case where $M=\Q(t)$ is the field of fractions of $\Z[t^{\pm 1}]$ as well as~$M=\Q(t)/\Z[t^{\pm 1}].$ 

For the $\Z[t^{\pm 1}]$-modules $M=\Z[t^{\pm 1}],\Q(t)$ and $\Q(t)/\Z[t^{\pm 1}]$, there is an evaluation homomorphism
$$\ev \colon H^*(X,A;M) \to \overline{\Hom_{\Z[t^{\pm 1}]}(H_*(X,A;\Z[t^{\pm 1}]),M)}.$$
A more thorough discussion of this evaluation map can be found for example in~\cite[Section~2.3]{FriedlPowell}.

When $X$ is an $n$-manifold, there are Poincar\'e duality isomorphisms
$$\PD \colon H^{n-*}(X;M) \xrightarrow{\cong} H_*(X,\partial X;M) \ \ \  \text{ and}  \ \ \  \PD \colon H^{n-*}(X,\partial X;M) \xrightarrow{\cong} H_*(X;M).$$

\subsection{The Blanchfield form}
\label{sub:Blanchfield}

We very briefly review the homological definition of the Blanchfield form.
Given a $3$-manifold $Y$ and an epimorphism $\varphi \colon \pi_1(Y) \twoheadrightarrow \Z$ such that the Alexander module~$H_1(Y;\Z[t^{\pm 1}])$ is torsion, the \emph{Blanchfield form} is the pairing
\begin{align*}
\Bl_Y \colon H_1(Y;\Z[t^{\pm 1}]) \times H_1(Y;\Z[t^{\pm 1}]) &\to \Q(t)/\Z[t^{\pm 1}] \\
(x,y) &\mapsto \ev( \operatorname{BS}^{-1} \circ \PD^{-1}(i_*(y)))(x).
\end{align*}
Here $i_* \colon H_1(Y;\Z[t^{\pm 1}]) \to H_1(Y,\partial Y;\Z[t^{\pm 1}])$ is the homomorphism induced by the inclusion,  $\PD^{-1} \colon H_1(Y,\partial Y;\Z[t^{\pm 1}]) \to H^2(Y;\Z[t^{\pm 1}])$ denotes the inverse of the Poincar\'e duality isomorphism,~$\BS^{-1}$ denotes the inverse of the Bockstein map~$\BS \colon H^2(Y;\Z[t^{\pm 1}]) \to H^1(Y;\Q(t)/\Z[t^{\pm 1}])$ associated to the short exact sequence $0 \to \Z[t^{\pm 1}] \to \Q(t) \to \Q(t)/\Z[t^{\pm 1}] \to 0$ of coefficients (here the homomorphism~$\BS$ is an isomorphism because we assumed that $H_1(Y;\Z[t^{\pm 1}])$ is torsion), and $\ev \colon H^1(Y;\Q(t)/\Z[t^{\pm 1}]) \to \overline{\Hom(H_1(Y;\Z[t^{\pm 1}]),\Q(t)/\Z[t^{\pm 1}])}$ is the evaluation homomorphism.
The Blanchfield form is sesquilinear
and Hermitian; see e.g.~\cite{PowellBlanchfield}.
If $Y$ is closed, then $\Bl_Y$ is nonsingular.

\begin{remark}
\label{rem:AlexanderModule}
Given a knot $K$, the inclusion $E_K \subset M_K$ induces an isometry 
$$(H_1(E_K^\infty),\Bl_{E_K}) \xrightarrow{\cong} (H_1(M_K^\infty),\Bl_{M_K}).$$
In what follows, we will refer to either pairing as \emph{the Blanchfield form} of $K$ and write~$\Bl_K$. 
\end{remark}

\subsection{The equivariant intersection form}
\label{sub:Equivariant}

We briefly review the homological definition of the equivariant intersection form.
Given a $4$-manifold $X$ with (possibly empty) boundary, and an epimorphism $\varphi \colon \pi_1(X) \twoheadrightarrow \Z$, the \emph{equivariant intersection form} is the pairing
\begin{align*}
\lambda_X \colon H_2(X;\Z[t^{\pm 1}]) \times H_2(X;\Z[t^{\pm 1}]) &\to \Z[t^{\pm 1}] \\
(x,y) &\mapsto \ev(\PD^{-1} \circ i_*(y))(x),
\end{align*}
where $i_*\colon H_2(X;\Z[t^{\pm 1}]) \to  H_2(X,\partial X;\Z[t^{\pm 1}])$ is the homomorphism induced by the inclusion, $\PD^{-1} \colon H_2(X,\partial X;\Z[t^{\pm 1}]) \to H^2(X;\Z[t^{\pm 1}])$ is the inverse of the Poincar\'e duality isomorphism and~$\ev \colon H^2(X;\Z[t^{\pm 1}]) \to \overline{\Hom(H_2(X;\Z[t^{\pm 1}]),\Z[t^{\pm 1}])}$ is the evaluation homomorphism.
The equivariant intersection form is sesquilinear and Hermitian.

We will occasionally also consider the relative pairing
\begin{align*}
\lambda_X^\partial \colon H_2(X;\Z[t^{\pm 1}]) \times H_2(X,\partial X;\Z[t^{\pm 1}]) &\to \Z[t^{\pm 1}] \\
(x,y) &\mapsto \ev(\PD^{-1}(y))(x).
\end{align*}
Finally, we record the relation between the equivariant intersection form and the Blanchfield form.
This is a well known statement, see e.g.~\cite[Remark 2.4 and Proposition~3.5]{ConwayPowell} for a reference involving conventions that match ours,  and so we omit the proof.
\begin{proposition}
\label{prop:presents}
Let~$W$ be a~$4$-manifold with~$\pi_1(W) \cong \Z$, connected boundary,  $\pi_1(\partial W ) \to \pi_1(W)$ surjective and~$H_1(\partial W^\infty)$ a $\Z[t^{\pm 1}]$-torsion module.
Any matrix $A$ representing $\lambda_W$ presents the Blanchfield form $\Bl_{\partial W}$, meaning that $\Bl_{\partial W}$ is isometric to the linking form
\begin{align*}
\Z[t^{\pm 1}]^n/A^T\Z[t^{\pm 1}]^n \times \Z[t^{\pm 1}]^n/A^T\Z[t^{\pm 1}]^n  \to \Q(t)/\Z[t^{\pm 1}] \\
([x],[y]) \mapsto -x^TA^{-1}\overline{y},\nonumber
\end{align*}
where $n$ denotes the size of $A$, and the inverse is taken over $\Q(t).$
\end{proposition}

It is worth noting that the sign in this proposition depends on various conventions that might differ depending on the article.
For example no such sign appears in~\cite[Section 1.3 and Theorem~2.6]{BorodzikFriedlClassical1}, whereas~\cite[Lemma 10.2]{ChaOrrPowell} contains the sign but has all instances of $A$ replaced by $A^T$.

\section{Existence of $\Z$-disks in $\cpring$}
\label{sec:Existence}

The goal of this section is to prove Theorem~\ref{thm:ExistenceIntro} from the introduction which lists criteria for a knot $K$ that are equivalent to it being $\Z$-slice in $\cpring$.
\medbreak

Let $N$ be a simply-connected~$4$-manifold with boundary~$S^3$ and let $D \subset N$ be a $\Z$-disk with boundary a knot~$K \subset S^3$.
We use $N_D:=N \setminus \nu (D)$ to denote the exterior of $D$ and recall that~$\partial N_D=M_K$ is the result of~$0$-framed surgery along~$K$.
The next lemma lists some algebro-topological  properties of $\Z$-disk exteriors.

\begin{lemma}
\label{lem:DiscsGeneral}
Let $N$ be a simply-connected~$4$-manifold with boundary~$S^3$.
Given a~$\Z$-disk~$D \subset N$ with boundary~$K$,  the following assertions hold:
\begin{enumerate}
\item the~$\Z[t^{\pm 1}]$-module~$H_2(N_D;\Z[t^{\pm 1}])$ is free of rank~$b_2(N)$;
\item  if~$A(t)$ is a matrix representing~$\lambda_{N_D}$, then~$A(1)$ represents the intersection form of~$N$;
\item if~$A(t)$ is a matrix representing~$\lambda_{N_D}$, then the transpose $A^T(t)$ presents the Alexander module $H_1(M_K;\Z[t^{\pm 1}])$.
\end{enumerate}
\end{lemma}
\begin{proof}
The first assertion is proved in~\cite[Claim 4 on page 43]{ConwayPowell} for~$N=D^4$ but the proof generalises to other simply-connected $4$-manifolds with boundary $S^3$.
The second assertion follows from~\cite[Lemma 5.10]{ConwayPowell}.
The third assertion is proved in~\cite[Lemma 3.2~(4)]{ConwayPowell}.
\end{proof}

Next, we restrict to~$\Z$-disks in~$N=\cpring$ with boundary a knot $K$ and recall that $\Delta_K$ denotes the unique symmetric representative of the Alexander polynomial of a knot~$K$ that evaluates to~$-1$ at $t=1$.
The next lemma shows that only one Hermitian form~$(H,\lambda)$ over $\Z[t^{\pm 1}]$ can arise as the equivariant intersection form of such a~$\Z$-disk exterior.

\begin{proposition}
\label{prop:barCP2}
If~$D \subset \cpring$ is a~$\Z$-disk with~$\partial D=K$, then the equivariant intersection form~$\lambda_{N_D}$ is represented by the size one matrix~$(\unaryminus\Delta_K)$.
\end{proposition}
\begin{proof}
The first item of Lemma~\ref{lem:DiscsGeneral} implies that $H_2(N_D;\Z[t^{\pm 1}]) \cong \Z[t^{\pm 1}]$ and the third item of Lemma~\ref{lem:DiscsGeneral} implies that~$\lambda_{N_D}$ is represented by a size one Hermitian matrix whose determinant is equal (up to multiplication by units) to the Alexander polynomial of~$K$.
Since~$\lambda_{N_D}$ is Hermitian, such a matrix must therefore be of the form~$(f(t))$ where~$f(t)$ is a symmetric representative of this Alexander polynomial.
There are two such representatives,  namely $\pm \Delta_K$.
Since the second item of Lemma~\ref{lem:DiscsGeneral} ensures that~$A(1)$ is a matrix for the intersection form of~$\cpring$, i.e.~$A(1)=(1)$, we deduce that $A(t)=(\unaryminus\Delta_K)$.
This concludes the proof of the proposition.
\end{proof}

We can now prove Theorem~\ref{thm:ExistenceIntro} from the introduction. 
\begin{customthm}{\ref{thm:ExistenceIntro}}
\label{thm:Existence}
Given a knot~$K$, the following assertions are equivalent: 
\begin{enumerate}
\item $K$ is~$\Z$-slice in~$\cpring$;
\item the Blanchfield form~$\Bl_K$ is presented by $\unaryminus\Delta_K(t)$;
\item $K$ can be converted into an Alexander polynomial one knot by switching a single positive crossing to a negative crossing;
\item $K$ can be converted into an Alexander polynomial one knot by a single positive generalized crossing change.
\end{enumerate}
\end{customthm}
\begin{proof}
The implication~$(1) \Rightarrow (2)$ follows from the combination of Propositions~\ref{prop:presents} and~\ref{prop:barCP2}.
The implication~$(2) \Rightarrow (3)$ is due Borodzik-Friedl~\cite[Theorem 5.1]{BorodzikFriedlLinking} and~$(3) \Rightarrow (4)$ is immediate.
The fact that~$(4) \Rightarrow (1)$ is well known and will be discussed in detail in the next section, but we outline the proof briefly.
Assume that $K$ can be transformed into an Alexander polynomial one knot $J$ by a generalized crossing change.
The generalized crossing change is performed by surgering $(S^3,K)$ along a $(-1)$-framed unknotted curve~$\gamma \subset S^3$ that is disjoint from $K$ and nullhomologous in $E_K$.
It follows that $K \subset S^3$ and $J \subset S^3$ are concordant in $W:=(S^3 \times I)\# \C P^2$  via a concordance~$C$.
Since $J$ has Alexander polynomial one, work of Freedman ensures that~$J$ bounds a $\Z$-disk $\Delta \subset D^4$~\cite{Freedman}.
Delaying a thorough discussion of orientations to Section~\ref{sec:CrossingChange}, the required disk $D \subset \cpring$ is obtained by using~$\Delta$ to cap off the concordance: 
$$(\cpring,D) := (\unaryminus((S^3 \times I)\# \C P^2),C) \cup_{(S^3,J)} (D^4,\Delta).$$
It remains to show that the complement of $D$, denoted $N_D$, has fundamental group~$\Z$.
 By construction of the disk~$D$, we know that~$\cpring_D=\unaryminus W_C \cup_{M_J} N_\Delta$ where $W_C$ (resp.  $N_\Delta$) denotes exterior of $C \subset W$ (resp. of $\Delta \subset D^4$).
 The claim now follows from a van Kampen argument using that $\gamma$ is nullhomologous in $E_K \subset M_K$ and that $\Delta$ is a $\Z$-disk.
\end{proof}

\section{Crossing change $\Z$-disks}
\label{sec:CrossingChange}

The proof of Theorem~\ref{thm:Existence} shows that if a single generalized positive crossing change turns~$K$ into a knot that bounds a $\Z$-disk in~$D^4$, then $K$ bounds a $\Z$-disk in~$\cpring$.
We consider this construction in more detail (Section~\ref{sub:Definition}), study the exterior of such disks (Section~\ref{sub:Exterior}) and their further properties (Section~\ref{sub:Basis}).

\subsection{The definition of generalized crossing change $\Z$-disks}
\label{sub:Definition}

We describe in more detail how a generalized positive crossing change leads to a $\Z$-disk in~$\cpring$, provided the knot resulting from the crossing change has Alexander polynomial one.
The set-up that follows is used to ensure that we are comparing discs in a fixed copy of $\cpring$.

\begin{convention}
\label{conv:Fixed}
We fix the following data once and for all.
\begin{itemize}
\item A copy of $D^4$ with boundary $S^3$.
\item A $(-1)$-framed oriented unknot $c \subset S^3$.
\item An orientation-reversing homeomorphism $h \colon S_{-1}^3(c) \xrightarrow{\cong} S^3.$
\item  A copy of $\cpring$ which we think of as
$$\cpring=\left( -\left((S^3 \times [0,1]) \cup_{c \times \lbrace 1 \rbrace} h^2 \right)  \right) \cup_h D^4.$$
Here $h^2$ denotes a $2$-handle whose core we denote $\core(h^2)$ and whose cocore we denote~$\cocore(h^2).$
We think of~$\mu_c:=\partial \cocore(h^2)$ both as a curve in $S^3 \times \lbrace 1 \rbrace$ and in $S^3_{-1}(c).$ 
\item An oriented knot $K \subset S^3$.
\end{itemize}
\end{convention}

We reverse the orientation of $(S^3 \times [0,1]) \cup h^2$ so that $S^3 \times  \lbrace 0 \rbrace \subset -\left((S^3 \times [0,1]) \cup h^2\right)$ has the same orientation as the usual of $S^3$, thought of as the boundary of $D^4$.
Recall that we orient the boundary of a manifold according to the ``outwards normal vector first" rule so that
$$\partial (S^3 \times [0,1]) =-S^3\times \{ 0\} \sqcup S^3 \times \{ 1 \}$$
For this reason, we frequently think of $K$ as a subset of $S^3\times \{ 0\} \subset -((S^3 \times [0,1]) \cup h^2). $
\begin{remark}
Given a simple closed curve $\eta \subset S^3$ and an ambient isotopy $(f_t \colon S^3 \to S^3)_{t \in  [0,1]}$ with~$f_0=\id_{S^3}$,  we think of the trace $C_{\eta,f}$ of $(f_t|_{\eta})_{t \in [0,1]}$ as a subset of~$-(S^3 \times [0,1])$ with boundary
$$\partial C_{\eta,f} =\eta \sqcup -f(\eta) \subset S^3\times \{ 0\} \sqcup - S^3 \times \{ 1 \}.$$
\end{remark}

\begin{notation}
\label{not:Data}
We also fix the following data which we will use frequently when working with~$\Z$-disks in $\cpring$ that arise from a single positive generalized crossing change. 
\begin{itemize}
\item  
An oriented curve $\gamma \subset E_K$ that is unknotted in $S^3$ with $\ell k(\gamma,K)=0$ such that a generalized positive crossing change along $\gamma$ results in an Alexander polynomial one knot.
\item 
An ambient isotopy $(f_t \colon S^3 \to S^3)_{t \in [0,1]}$ with $f_0=\id_{S^3}$ and $f_1(\gamma)=c$ as oriented curves; write
\begin{itemize}
\item[$\circ$] $C_{\gamma,f} \subset -\left((S^3 \times [0,1]) \cup_{c \times \lbrace 1 \rbrace} h^2\right)$ for the trace of the isotopy $(f_t|_{\gamma})_{t \in [0,1]}$.
\item[$\circ$] $C_{K,f} \subset -\left((S^3 \times [0,1]) \cup_{c \times \lbrace 1 \rbrace} h^2\right)$ for the trace of the isotopy $(f_t|_K)_{t \in [0,1]}$.
\item[$\circ$]  $K' \subset S^3_{-1}(c)$ for the knot $f_1(K) \subset S^3$ viewed as a subset of $S^3_{-1}(c)$.
\item[$\circ$] $J:=h(-K') \subset S^3$ for the image of  
$K' \subset S^3_{-1}(c)$
  under the orientation-reversing homeomorphism $h \colon S^3_{-1}(c) \xrightarrow{\cong} S^3.$
\end{itemize}
\item[$\circ$] A $\Z$-disk $\Delta \subset D^4$ with boundary the knot $J \subset S^3$.
\end{itemize}
\end{notation}

Note that in the topological category,  a knot bounds a $\Z$-disk in $D^4$ if and only if it has Alexander polynomial one~\cite{Freedman}; this disk is unique up to topological isotopy rel.\ boundary~\cite{ConwayPowellDiscs}.
Consequently,  in the topological category,  there is no ambiguity in choosing a disk $\Delta$ with boundary~$J$.
On the other hand,  the existence of a smoothly embedded disk~$\Delta$ with boundary $J$ is not a given and, if such a disk exists, it may not be unique~\cite{AkbulutZeeman,Hayden}.

\begin{definition}
\label{def:CrossingChangeZdisc}
Given a curve $\gamma \subset E_K$, an isotopy $(f_t \colon S^3 \to S^3)_{t \in [0,1]}$ and a $\Z$-disk $\Delta \subset D^4$ as in Notation~\ref{not:Data}, the \emph{generalized crossing change $\Z$-disk} $D_{\gamma}$ is
$$ (\cpring,D_{\gamma})=\left( - ((S^3 \times [0,1]) \cup_{c \times \lbrace 1 \rbrace} h^2),C_{K,f} \right) \cup_h (D^4,\Delta).$$ 
\end{definition}

The following proposition shows that up to isotopy rel.\ boundary, the disk $D_\gamma$ only depends on the knot~$K$ and on the surgery curve~$\gamma$.

\begin{proposition}\label{prop:diskdoesntdepend}
Let $K$ be a knot and let $\gamma \subset E_K$ be a sugery curve as in Notation~\ref{not:Data}.
Up to isotopy rel boundary,  a generalized crossing change $\Z$-disk $D_\gamma$ with boundary $K$ depends neither on the choice of the ambient isotopy $(f_t)_{t \in [0,1]}$ nor on the $\Z$-disk $\Delta$.
\end{proposition}
\begin{proof}
Fix two ambient isotopies~$(f_t)_{t \in [0,1]},(g_t)_{t \in [0,1]}$ of~$S^3$ with $f_0=\id_{S^3}=g_0$ and $f_1,g_1$ taking~$\gamma$ to~$c$.
During this proof, we write~$D_{\gamma,f}$ and~$D_{\gamma,g}$ for the the resulting generalized crossing change $\Z$-disks.
Although not included in the notation,  we note that~$D_{\gamma,f}$ and~$D_{\gamma,g}$ a priori also depend on the disks used to cap off the traces $C_{K,f}$ and $C_{K,g}$.
In a nutshell, the idea of the proof is to use the Smale conjecture~\cite{HatcherSmale} to show that the isotopies~$(f_t)_{t \in [0,1]}$ and~$(g_t)_{t \in [0,1]}$ are isotopic and to use this isotopy of isotopies, together with the fact that $\Z$-discs in $D^4$ are unique up to isotopy rel. boundary~\cite{ConwayPowell}, to prove that~$D_{\gamma,f}$ and~$D_{\gamma,g}$ are isotopic.
We now expand on this outline.

Our goal is to define a $1$-parameter family of rel.\  boundary homeomorphisms $\Phi_s \colon \cpring \to \cpring$ with $\Phi_0=\id_{\cpring}$ and $\Phi_1$ taking one disk to another.

    \begin{claim}\label{claim:sameend}
If~$(f_t)$ and~$(g_t)$ are isotopic through a 2-parameter family~$(h_{t,s} \colon S^3\times \{ t \}\times \{ s \}\to S^3)_{t,s \in I}$ with 
$$h_{t,0}=f_t, \quad  h_{t,1}=g_t,  \quad  h_{0,s}=\id_{S^3},  \quad h_{1,s}(\gamma)=c,$$
then~$D_{\gamma,f}$ and~$D_{\gamma,g}$ are isotopic rel.\ boundary in~$\cpring$.
\end{claim}

In what follows we will sometimes think of $\mathbf{f}:=(f_t)$ and~$\mathbf{g}:=(g_t)$ as paths in the space of self-homeomorphisms~$\Homeo(S^3)$.
The assumptions of the claim are then that there is a homotopy $\mathbf{h} \colon I \times I \to \Homeo(S^3)$ with $\mathbf{h}_0=\mathbf{f}, \mathbf{h}_1=\mathbf{g}$ as well as~$\mathbf{h}_s(0)=\id_{S^3}$ and $\mathbf{h}_s(1)(\gamma)=c$ for every $s \in I$.

\begin{proof}[Proof of Claim \ref{claim:sameend}]
We have written $\cpring$ as the union of the three pieces $S^3 \times I$, $h^2$, and $D^4$:
\[
\cpring = \left( -\left((S^3 \times I) \cup_{c \times \lbrace 1 \rbrace} h^2 \right)  \right) \cup_{h} D^4.
\]
We construct a one-parameter family of self-homeomorphisms
\[
\Phi_s \colon \cpring \rightarrow \cpring
\]
with the following properties:
\begin{enumerate}
\item $\Phi_0$ is the identity;
\item $\Phi_s$ is the identity on $\partial \cpring$ for all $s$; 
\item $\Phi_s$ preserves each of the pieces $S^3 \times I$, $h^2$, and $D^4$ setwise for all $s$; and,
\item $\Phi_1(C_{K,f}) = C_{K,g}$.
\end{enumerate}
This suffices to prove the claim. Indeed,~$(\Phi_s)_{s \in [0,1]}$ tautologically constitutes an ambient isotopy of~$\cpring$ sending the disk~$D_{f, \gamma}$ to the disk~$\Phi_1(D_{f, \gamma})$. 
Due to the last two conditions above,~$\Phi_1(D_{\gamma,f})$ and~$D_{\gamma,g}$ coincide in~$S^3 \times I$. 
Hence~$\Phi_1(D_{\gamma,f})$ and~$D_{\gamma,g}$ differ only by a choice of capping~$\Z$-disk in~$D^4$; any two such~$\Z$-disks are topologically isotopic rel boundary~\cite{ConwayPowellDiscs}.

We build~$\Phi_s$ by first defining a one-parameter family of self-homeomorphisms of~$S^3 \times I$ satisfying the desired conditions. We then extend this family to a one-parameter family of self-homeomorphisms of~$(S^3 \times I) \cup_{c \times \lbrace 1 \rbrace} h^2$, and then successively to a one-parameter family of self-homeomorphisms of all of~$\cpring$. The construction of~$\Phi_s$ on~$S^3 \times I$ is immediate from the hypotheses of the claim by defining
\[
(\Phi_s)|_{S^3 \times I} = \mathbf{h}_s \circ \mathbf{h}_0^{-1} = \mathbf{h}_s \circ \mathbf{f}^{-1}.
\]
Here, we briefly abuse notation by considering $\mathbf{f}$ as a map from $S^3 \times I$ to itself sending $(x, t)$ to~$(f_t(x), t)$, and likewise for $\mathbf{h}_s$.
This is clearly consistent with the desired conditions.

We now extend~$\Phi_s$ over the~$2$-handle attachment. 
Note~$\Phi_s$ maps~$c \times \{1\}$ to itself for all~$s$; without loss of generality, we may assume~$\Phi_s$ preserves a tubular neighborhood~$\overline{\nu}(c)$ of~$c \times \{1\}$ in~$S^3 \times \{1\}$. 
Denote the induced one-parameter family of self-diffeomorphisms of~$\overline{\nu}(c)$ by~$\phi_s \colon \overline{\nu}(c) \rightarrow \overline{\nu}(c)$. 
Note that~$\phi_0=\id_{\overline{\nu}(c)}$ because~$\Phi_0=\id_{S^3}.$
For convenience of notation, we insert a collar~$\overline{\nu}(c) \times I$ between~$S^3 \times \{1\}$ and the~$2$-handle~$h^2$. 
We thus consider the~$2$-handle attachment as being $(S^3 \times I) \cup (\overline{\nu}(c) \times I) \cup h^2$, where~$\overline{\nu}(c) \subseteq S^3 \times \{1\}$ is glued to~$\overline{\nu}(c) \times \{0\}$ via the identity and~$h^2$ is attached along~$\overline{\nu}(c) \times \{1\}$. 
Extend~$\Phi_s$ over~$\overline{\nu}(c) \times I$ by defining
\[
(\Phi_s)|_{\overline{\nu}(c) \times I} (x, t) = (\phi_{(1-t)s}(x), t).
\]
Note that~$\Phi_s$ is the identity on~$\overline{\nu}(c) \times \{1\}$ for all~$s$. 
Hence we may identically extend
\[
(\Phi_s)|_{h^2} = \id.
\]
This produces the desired extension of~$\Phi_s$ over the~$2$-handle attachment.

Finally, we observe that~$\Phi_s$ may be extended over the capping~$4$-ball~$D^4$. 
Indeed, our definition of~$\Phi_s$ so far gives a one-parameter family~$\psi_s$ of self-diffeomorphisms of~$S^3 = \partial D^4$. 
This extends to a one-parameter family of self-homeomorphisms of~$D^4$ via the usual radial extension/Alexander trick:
\[
(\Phi_s)|_{D^4}(x, r) = (\psi_s(x), r)
\]
where~$x \in S^3$ and~$r$ is the radial coordinate of~$D^4$. This completes the proof of the claim.
\end{proof}
\color{black}

We now prove that the assumption of Claim~\ref{claim:sameend} always holds.

\begin{claim}\label{claim:makesame}
There is a homotopy $\mathbf{h} \colon I \times I \to \Homeo(S^3)$ with $\mathbf{h}_0=\mathbf{f}, \mathbf{h}_1=\mathbf{g}$ as well as~$\mathbf{h}_s(0)=\id_{S^3}$ and $\mathbf{h}_s(1)(\gamma)=c$ for every $s \in I$.
\end{claim}

Note that $h$ is a free homotopy, not a homotopy rel.\ endpoints.

\begin{proof}[Proof of Claim \ref{claim:makesame}]
We start with an observation whose proof is left to the reader.
Fix a path of homeomorphisms $\boldsymbol{\eta} \colon I \to \Homeo(S^3)$ with $\eta_0=\id_{S^3}$.
If $\mathbf{r} \colon I \to \Homeo(S^3)$ is a path of homeomorphisms with~$r_0=\id_{S^3}$ and $r_t(c)=c$ for every $t$,  then the path $\boldsymbol{\eta}$ is freely homotopic to the path $\boldsymbol{\eta * r}$ given by
$$
(\eta  * r)_t:=
\begin{cases}
\eta_{2t} &\quad t \in [0,\frac{1}{2}]  \\ 
r_{2t-1} \circ \eta_1 &\quad  t \in [\frac{1}{2},1]
\end{cases}
$$
through a homotopy $\boldsymbol{\varphi} \colon I \times I \to \Homeo(S^3)$ with $\boldsymbol{\varphi}_0=\boldsymbol{\eta},\boldsymbol{\varphi}_1=\boldsymbol{r * \eta}$ as well as $\varphi_s(0)=\eta_0=\id_{S^3}$ and $\varphi_s(1)(c)=r_s(c)=c$ for every $s \in I$.
%

We now begin the proof of the claim proper and assert that,  without loss of generality,  we can assume that~$f_1$ and~$g_1$ agree pointwise on a tubular neigbhorhood~$\overline{\nu}(\gamma)$ of~$\gamma$.
Let~$\mathbf{r}  \colon I \to \Homeo(S^3)$ be a path of homeomorphisms with~$r_0=\id_{S^3}$ and~$r_1 \circ f_1|_{\overline{\nu}(\gamma)} =g_1|_{\overline{\nu}(\gamma)}$.
The observation implies that~$\mathbf{f}$ is homotopic to~$\boldsymbol{f * r}$ through a homotopy~$\boldsymbol{\varphi}$ with~$\varphi_0(s)=\id_{S^3}$ and~$\varphi_1(s)(c)=c$ for every~$s \in I$.
Since~$(f*r)_1|_{\overline{\nu}(\gamma)}=(r_1 \circ f_1)|_{\overline{\nu}(\gamma)}=g_1|_{\overline{\nu}(\gamma)}$,  this concludes the proof of the assertion.


Next we assert that without loss of generality we can assume that~$f_1$ and ~$g_1$ agree on the union of~$\overline{\nu}(\gamma)$ with a closed tubular neigbhorhood of a disk bounded by~$\gamma$.
Let~$D \subset S^3$ be a smoothly embedded disk bounded by~$\gamma$ intersecting~$\overline{\nu}(\gamma)$ only in a collar of $\partial D$. 
Initially~$f_1(D)$ and~$g_1(D)$ are disks in~$S^3$ with boundary~$c$ that may be distinct.
An application of Schoenflies's theorem shows that these disks are ambiently isotopic rel.\  boundary via an ambient isotopy~$(r_t \colon S^3 \to S^3)_{t \in [0,1]}$ with~$r_0=\id_{S^3}$ and~$r_1 \circ f_1|_D=g_1|_D$.
This can be extended so that~$r_1 \circ f_1|_{\overline{\nu}(D)}=g_1|_{\overline{\nu}(D)}$.
The assertion again follows from the observation at the beginning of the proof.


We now assert that without loss of generality we can assume that~$f_1$ and~$g_1$ agree pointwise on the whole of~$S^3$.
Presently, the homeomorphisms~$f_1$ and~$g_1$ agree pointwise on~$\overline{\nu}(D)$ but might not agree on the closed ball~$B=S^3 \setminus f_1(\nu(D))$.
They do however agree on~$\partial B$,  so Alexander's trick ensures the existence of a family~$(\widetilde{r}_t \colon B \to B)_{t \in [0,1]}$ of rel.\ boundary homeomorphisms with~$\widetilde{r}_0=\id_B$ and~$\widetilde{r}_1 \circ f_1|_B=g_1|_B.$
 Since~$\widetilde{r}_t|_{\partial B}=\id_{\partial B}$ for every~$t \in [0,1]$, we can extend~$\widetilde{r}_t$ by the identity on~$S^3 \setminus \operatorname{Int}(B)$ resulting in a path~$\mathbf{r} \colon I \to \Homeo(S^3)$ with~$r_0=\id_{S^3},r_1 \circ f_1=g_1$ and~$r_1(s)(c)=c$ for every~$s \in I$.
The assertion again follows from the observation.

We now prove Claim~\ref{claim:makesame} for paths of homeomorphisms~$\mathbf{f}=(f_t)$ and~$\mathbf{g}=(g_t)$ with~$f_1=g_1$.
Consider the loop in~$[\boldsymbol{f \cdot \overline{g}}] \in \pi_1(\Homeo(S^3),\id_{S^3})$:
     \[
(f \cdot \overline{g})_t
     =\begin{cases}
     f_{2t}& t \in [0,\frac{1}{2}] \\
     g_{2t-1}& t \in [\frac{1}{2},1].
     \end{cases}
     \] 
     If~$[\boldsymbol{f \cdot \overline{g}}]=1$, then~$f,g$ are homotopic rel.~$f_0=g_0=\id_{S^3}$ and~$f_1=g_1$ through a homotopy~$\mathbf{h}$.
This homotopy satisfies the conditions of the claim because $h_s(0)=f_0=\id_{S^3}$ and $h_s(1)(\gamma)=f_1(\gamma)=c.$

Assume now that~$[\boldsymbol{f \cdot \overline{g}}] \neq 1 \in \pi_1(\Homeo(S^3),\id_{S^3})$.
  The Smale conjecture~\cite{HatcherSmale} gives the isomorphism~$\pi_1(\Homeo(S^3),\id) \cong  \pi_1(O(4)) \cong \Z_2$, where the generator is represented by the loop of homeomorphisms $\mathbf{r}=(r_t)_{t \in [0,1]}$ with $r_t$ the rotation of angle~$2\pi t $ about an axis fixing~$c$ setwise.
Since~$r_0=\id_{S^3}$ and~$r_t(c)=c$ for every~$t$,  we can use the observation at the begining of the proof of this claim to reduce ourselves to proving that $[\boldsymbol{(f * r) \cdot \overline{g}}] = 1\in \pi_1(\Homeo(S^3),\id_{S^3})$: thanks to the previous paragraph, replacing~$\mathbf{f}$ by~$\mathbf{f *r}$ would then give the desired outcome.

We now prove that $[\boldsymbol{(f * r) \cdot \overline{g}}] = 1\in \pi_1(\Homeo(S^3),\id_{S^3})$. This is equivalent to proving that~$[\boldsymbol{\overline{g} \cdot (f*r)}] = 1 \in  \pi_1(\Homeo(S^3),f_1)$.
Consider the nontrivial loop $\boldsymbol{r \circ f_1}$ of homeomorphisms  based at $f_1$ and defined by~$(r \circ f_1)_t:=r_t \circ f_1$.
Note that
$\boldsymbol{\overline{g} \cdot (f*r)}
 \simeq 
 \boldsymbol{(\overline{g} \cdot f) \cdot (r \circ f_1)}.$
 Since neither loop is nullhomotopic and $\pi_1(\Homeo(S^3),f_1)\cong \Z_2$, the conclusion follows. 
    \end{proof}
    Claim \ref{claim:makesame} implies that the paths of homeomorphisms $\mathbf{f}$ and $\mathbf{g}$ are homotopic via a homotopy satisfying the hypotheses of Claim \ref{claim:sameend} and thus that $D_{\gamma,f},D_{\gamma,g}$ are isotopic rel.\ boundary.
\end{proof}

The orientations of $c$ and $\gamma$ are essential in the definition of $D_{\gamma}$.
The two disks $D_{\gamma}$ and $D_{-\gamma}$ are related by a homeomorphism of $\cpring$ inducing the map $-1$ on $H_2(\cpring)=\mathbb{Z}$, but in general need not be isotopic rel.\ boundary. 

\subsection{The exterior of a generalized crossing change $\Z$-disk.}
\label{sub:Exterior}

The goal of this section is to describe a decomposition of the exterior of a generalized crossing change $\Z$-disk and its infinite cyclic cover.

\begin{notation}
Continuing with Notation~\ref{not:Data}, we write $M_{K'}$ for the~$0$-surgery on $K' \subset S^3_{-1}(c)$ and $E_{K'} \subset S^3_{-1}(c)$ for its exterior.
Since $J:=h(K')$, observe that our fixed homeomorphism $h \colon S^3_{-1}(c) \xrightarrow{\cong} S^3$ induces homeomorphisms $E_{K'} \xrightarrow{\cong} E_J$ and $M_{K'} \xrightarrow{\cong} M_J$ that we also denote by~$h$. 
\end{notation}

For convenience,  we write the exterior of the trace $C_{K,f}$ of the isotopy $(f_t|_K)_{t \in [0,1]}$ as
\[
W := (S^3 \times I) \setminus \nu(C_{K, f}).
\]
This way, the exterior~$N_{D_\gamma}$ of a generalized crossing change $\Z$-disk decomposes into three pieces,  where (suppressing orientations for ease of notation):
\[
N_{D_\gamma} =W \cup h^2  \cup N_\Delta.
\]
Here, we initially view the union as being formed from left-to-right, in which case the handle~$h^2$ is attached to~$W$ along the curve~$c \subset S^3 \times \{1\}$ and $W \cup h^2$ is attached to $N_\Delta$ using the homeomorphism~$h \colon \partial (W \cup h^2) \supset M_{K'} \xrightarrow{\cong} M_J =\partial N_\Delta$.
We may also consider the union as being formed from right-to-left, in which case~$h^2$ is attached to~$N_\Delta$ via the cocore curve~$\mu_c$. Note in the latter situation, we view~$W$ as then being attached along the subset of its boundary given by~$E_{f_1(K)} \subset S^3 \times \{1\}$.

Now fix a basepoint~$z \in E_K$ and consider the infinite cyclic cover~$p \colon N_{D_\gamma}^\infty \to N_{D_\gamma}$ corresponding to the isomorphism~$\smash{\pi_1(N_{D_\gamma}, z) \cong \Z}$.
Denote the restricted covers by
\[
W^\infty:=p^{-1}(W) \quad \text{and} \quad N_\Delta^\infty:=p^{-1}(N_\Delta),
\]
so that the infinite cyclic cover of $N_{D_\gamma}$ decomposes as
\[
N_{D_\gamma}^\infty = W^\infty \cup \bigcup_{j \in \Z} t^j \tilde{h}^2 \cup N_\Delta^\infty.
\]
The inclusion maps~$W \hookrightarrow N_{D_\gamma}$ and~$N_\Delta \hookrightarrow N_{D_\gamma}$ induce isomorphisms on first homology, so the restricted covers~$W^\infty$ and~$N_\Delta^\infty$ are homeomorphic (as manifolds) to the infinite cyclic covers of~$W$ and~$N_{\Delta}$ defined using some (potentially different) choice of basepoint.
As before, we may view the above union as being formed from left-to-right or right-to-left, with the handles~$t^j \tilde{h}^2$ being attached along~$t^j \tilde{c}$ or~$t^j \tilde{\mu}_c$, respectively. Note that in the right-to-left case, the final piece~$W^\infty$ is then attached along the subset
\[
E_{f_1(K)}^\infty := p^{-1}(E_{f_1(K)})
\]
of its boundary.
We argue that~$N_{D_\gamma}^\infty$ deformation retracts onto the subset of $\bigcup_{j \in \Z} t^j \tilde{h}^2 \cup N_\Delta^\infty$ as this will be helpful for later homological calculations.

\begin{remark}
\label{rem:DeformationRetract}
We have a retract from~$W$ onto~$E_{f_1(K)}$ given by the map
\[
r(x, t) = (f_1 \circ f_t^{-1}(x), 1).
\]
This is readily seen to be a deformation retract via the homotopy
\[
H_s(x, t) = (f_{t + s(1-t)} \circ f_t^{-1}(x), t + s(1-t)).
\]
It follows that~$W^\infty$ deformation retracts onto the subset~$E^\infty_{f_1(K)}$, and thus that~$N_{D_\gamma}^\infty$ deformation retracts onto the subset
\[
\bigcup_{j \in \Z} t^j \tilde{h}^2 \cup N_\Delta^\infty.
\]
\end{remark}

\subsection{A preferred basis for $H_2(N_{D_\gamma}^\infty)$.}
\label{sub:Basis}

We saw in Lemma~\ref{lem:DiscsGeneral} that for any $\Z$-disk $D \subset \cpring$ with boundary $K$, the $\Z[t^{\pm 1}]$-module $H_2(N_D^\infty)$ is free of rank one.
In this section, when $D$ is a generalized crossing change $\Z$-disk,  we describe a surface in the infinite cyclic cover of this disk exterior that freely generates this module.


\begin{construction}[Bases for $H_2(N_{D_\gamma}^\infty)$ and $H_2(N_{D_\gamma}^\infty,M_K^\infty)$]
\label{cons:Basis}
Given a curve $\gamma \subset  E_K$, an ambient isotopy $(f_t \colon S^3 \to S^3)_{t \in [0,1]}$ and a $\Z$-disk $\Delta \subset D^4$ as in Notation~\ref{not:Data},  we construct a surface~$\Sigma \subset N_{D_\gamma}^\infty$ that freely generates~$H_2(N_{D_\gamma}^\infty)\cong \Z[t^{\pm 1}]$.
As in Section~\ref{sub:Exterior}, we set 
$$W:=(S^3 \times I) \setminus \nu(C_{K, f})$$
 and write $p \colon N_{D_\gamma}^\infty \to N_{D_\gamma}$ for  the infinite cyclic cover so that $N_{D_\gamma}^\infty$ decomposes as
 $$N_{D_\gamma}^\infty=\left( \unaryminus \left( W^\infty \cup \bigcup_{j \in \Z} t^j \widetilde{h}^2\right)\right) \cup N_\Delta^\infty.$$
Recall from Notation~\ref{not:Data} that $\mu_c$ denotes $\partial \cocore(h^2)$.
Thus the loop $\widetilde{\mu}_c \subset M_{J}^\infty=\partial N_\Delta^\infty$ bounds $\cocore(\widetilde{h}^2)$.
On the other hand,  since $H_1(M_J^\infty)=0$,  this loop bounds a surface $S \subset M_J^\infty$.
One can assume that $S \cap \core(h^2)=\emptyset$.
The union of these two surfaces gives rise to the closed surface
$$\Sigma:=S \cup_{\widetilde{\mu}_c} \cocore(\widetilde{h}^2) \subset N_{D_\gamma}^\infty.$$
Generically, we can assume that the surface $S\subset M_J^\infty$ projects down to an immersed surface in~$M_J$.
As we mentioned Remark~\ref{rem:DeformationRetract},  $W^\infty$ deformation retracts onto $E_{f_1(K)}^\infty$,  and since $H_2(N_\Delta^\infty)=0$ (recall Lemma~\ref{lem:DiscsGeneral}) a Mayer-Vietoris argument therefore shows that
$$H_2(N_{D_\gamma}^\infty)
\cong H_2\left(\bigcup_{j \in \Z} t^j \widetilde{h}^2 \cup N_\Delta^\infty \right)
\cong \Z[t^{\pm 1}]$$
is freely generated by $[\Sigma]$.
As explained for example in~\cite[proof of Lemma 3.2]{ConwayPowell}, the universal coefficient spectral sequence implies that the following composition is an isomorphism
$$\ev \circ \PD^{-1} \colon H_2(N_{D_\gamma},M_K;\Z[t^{\pm 1}]) \xrightarrow{\cong} \overline{\Hom_{\Z[t^{\pm 1}]}(H_2(N_{D_\gamma};\Z[t^{\pm 1}]),\Z[t^{\pm 1}])}=:H_2(N_{D_\gamma};\Z[t^{\pm 1}])^*  $$
\color{black}
It follows that~$H_2(N_{D_\gamma}^\infty,M_K^\infty)\cong \Z[t^{\pm 1}]$ is generated by the homology class of a surface Poincar\'e dual to~$\Sigma$, for example the class of 
$$\widehat{\core}(\widetilde{h}^2):=\core(\widetilde{h}^2) \cup_{\widetilde{c} \times \lbrace 1 \rbrace} \widetilde{C}_{\gamma,f}.$$
Here recall from Notation~\ref{not:Data} that~$\widetilde{C}_{\gamma,f}$ denotes the lift of the trace of the isotopy from $\gamma$ to $c$.
Details on why $[\Sigma]$ and $\widehat{\core}(\widetilde{h}^2)$ are Poincar\'e dual will be given during the proof Proposition~\ref{prop:BelongstoG}.

Summarizing the outcome of this construction, we have
\begin{align*}
&H_2(N_{D_\gamma}^\infty)\cong \Z[t^{\pm 1}] \langle [\Sigma] \rangle, \\ 
& H_2(N_{D_\gamma}^\infty,M_K^\infty)\cong \Z[t^{\pm 1}] \langle [\widehat{\core}(\widetilde{h}^2)]\rangle.
\end{align*}
Note for later use that $\partial (\widehat{\core}(\widetilde{h}^2))=\widetilde{\gamma}$.
\end{construction}

\begin{construction}[A basis for $H_2(\cpring)$]
\label{cons:H2CP2}
We define a generator
$$x_{\C P^2} \in H_2(\cpring)$$
by capping off $\cocore(h^2)$ with any immersed surface $F \looparrowright E_J \subset S^3_{-1}(c)$  with boundary $\mu_c$.
Here,  as previously,  we are using the decomposition~$\cpring= \left(\unaryminus ( (S^3 \times [0,1]) \cup_{c \times \lbrace 1 \rbrace} h^2) \right) \cup D^4$.
\end{construction}

Given a generalized crossing change $\Z$-disk $D_\gamma$ with boundary $K$, the next proposition describes some further properties of the homology classes $[\widehat{\core}(\widetilde{h}^2)] \in H_2(N_{D_\gamma}^\infty,M_K^\infty)$ and~$[\widetilde{\gamma}] \in H_1(M_K^\infty)$.

\begin{proposition}
\label{prop:BelongstoG}
Let $K$ be a knot and let $D_\gamma \subset \cpring$ be a generalized crossing change $\Z$-disk with $\partial D_\gamma =K$.
The basis elements $[\Sigma] \in H_2(N_{D_\gamma}^\infty)$ and $[\widehat{\core}(\widetilde{h}^2)] \in H_2(N_{D_\gamma}^\infty,M_K^\infty)$ from Construction~\ref{cons:Basis} satisfy the following properties:
\begin{enumerate}
\item The connecting homomorphism 
$$\partial_{D_\gamma} \colon H_2(N_{D_\gamma}^\infty,M_K^\infty) \to H_1(M_K^\infty)$$
is entirely determined by the equality $\partial_{D_\gamma}([\widehat{\core}(\widetilde{h}^2)])=[\widetilde{\gamma}]$.
\item The classes $[\widehat{\core}(\widetilde{h}^2)]$ and $[\Sigma]$ are Poincar\'e dual:  
$$\operatorname{ev} \circ \operatorname{PD}^{-1} ([\widehat{\core}(\widetilde{h}^2)])=[\Sigma]^*.$$
\item The homology class $[\widetilde{\gamma}] \in H_1(M_K^\infty)$ belongs to
\begin{equation}
\label{eq:GK}
\mathcal{G}_K:= \left\lbrace x \in H_1(M_K^\infty) \  \Big|  \ x \text{  is a generator and } \Bl_K(x,x)
=\frac{1}{\Delta_K} \right\rbrace.
\end{equation}
\item Under the composition
$$ H_2(N_{D_\gamma}^\infty) \xrightarrow{p_*} H_2(N_{D_\gamma}) \xrightarrow{\phi_{D_\gamma},\cong} H_2(\cpring)$$
of the projection and inclusion induced maps, the class $[\Sigma]$ is mapped to the preferred generator $x_{\C P^2}$ from Construction~\ref{cons:H2CP2}.
\end{enumerate}
\end{proposition}
\begin{proof}
We first show that $[\Sigma]$ and $[\widehat{\core}(\widetilde{h}^2)]$ are Poincar\'e dual as was already alluded to in Construction~\ref{cons:Basis}.
Since $H_2(N_{D_\gamma}^\infty)$ is freely generated by~$[\Sigma]$,  it suffices to show that evaluating~$\ev \circ \PD^{-1}([\widehat{\core}(\widetilde{h}^2)])$ and $[\Sigma]^*$  on $[\Sigma]$ leads to the same outcome.
This follows because $\Sigma=S \cup_{\widetilde{\mu}_c} \cocore(\widetilde{h}^2)$ and $\widehat{\core}(\widetilde{h}^2)=\core(\widetilde{h}^2) \cup_{\widetilde{c}} \widetilde{C}_{\gamma,f}$ are geometrically dual (the intersection point occurs at the intersection of $\cocore(\widetilde{h}^2)$ and $\core(\widetilde{h}^2)$)
and thanks to the definition of the equivariant intersection form:
$$ \langle [\Sigma]^*,[\Sigma] \rangle=1
=\lambda_{N_{D_\gamma}}^\partial([\Sigma],[\widehat{\core}(\widetilde{h}^2)])
= \langle \ev \circ \PD^{-1}([\widehat{\core}(\widetilde{h}^2)]),[\Sigma]\rangle.$$
This proves the second assertion.

The first assertion follows readily:
since~$\widehat{\core}(\widetilde{h}^2)$ freely generates~$H_2(N_{D_\gamma}^\infty,M_K^\infty) \cong \Z[t^{\pm 1}]$ and~$\partial (\widehat{\core}(\widetilde{h}^2))=\widetilde{\gamma}$, we deduce that $\partial_{D_\gamma}([\widehat{\core}(\widetilde{h}^2)])=[\widetilde{\gamma}]$;
this equality determines $\partial_{D_\gamma}$ because its domain is freely generated by $[\widehat{\core}(\widetilde{h}^2)]$.

We move on to the third assertion.
First,  the class~$[\widetilde{\gamma}]$ generates the Alexander module because~$\partial_{D_\gamma}$ is surjective and takes the generator of $H_2(N_{D_\gamma}^\infty,M_K^\infty)$ to $[\widetilde{\gamma}]$. 
In order to show that~$\Bl_K([\widetilde{\gamma}],[\widetilde{\gamma}])=\frac{1}{\Delta_K}$, we first record a more detailed formulation of Proposition~\ref{prop:presents}.
\begin{claim}
\label{claim:UsualTrick}
Let~$W$ be a~$4$-manifold with~$\pi_1(W) \cong \Z$, connected boundary,  and $\pi_1(\partial W ) \to \pi_1(W)$ surjective.
If~$H_1(\partial W^\infty)$ is $\Z[t^{\pm 1}]$-torsion, then for any~$\Z[t^{\pm 1}]$-basis~$e_1,\ldots,e_n$ of~$H_2(W;\Z[t^{\pm 1}])$ with dual basis~$e_1^*,\ldots,e_n^*$ of~$ H_2(W;\Z[t^{\pm 1}])^*$,  we have
$$\Bl_{\partial W}\left(\partial (\PD \circ \ev^{-1} (e_i^*)),\partial (\PD \circ \ev^{-1} (e_j^*))\right)\equiv \unaryminus \lambda_W(e_i,e_j)^{-1} \in \Q(t)/\Z[t^{\pm 1}],$$
 where~$\partial \colon H_2(W,\partial W;\Z[t^{\pm 1}]) \to H_1(\partial W;\Z[t^{\pm 1}])$ denotes the connecting homomorphism in the long exact sequence of the pair $(W,\partial W)$,  and the inverse is taken over $\Q(t)$.
\end{claim}

We saw in the second assertion that the isomorphism~$\PD \circ \ev^{-1}$ takes the dual class~$[\Sigma]^*$ to~$[\widehat{\core}(\widetilde{h}^2)] \in H_2(N_D,M_K;\Z[t^{\pm 1}])$.
Since~$\partial_{D_\gamma}([\widehat{\core}(\widetilde{h}^2)])=[\widetilde{\gamma}]$,  the claim now implies that
$$\Bl_K([\widetilde{\gamma}],[\widetilde{\gamma}])
=\unaryminus \lambda_{N_{D_\gamma}}([\Sigma],[\Sigma])^{-1}
=\unaryminus\left(\frac{1}{\unaryminus\Delta_K}\right)
=\frac{1}{\Delta_K} \in \Q(t)/\Z[t^{\pm 1}].
$$ 
For the penultimate equality, since~$H_2(N_{D_\gamma}^\infty) \cong \Z[t^{\pm 1}]$, Proposition~\ref{prop:barCP2} implies that $\lambda_{N_{D_\gamma}}(x,x)=\unaryminus\Delta_K$ for \emph{any} generator~$x$ of~$H_2(N_{D_\gamma}^\infty)$, e.g.  $x=[\Sigma]$.
This concludes the proof of the third assertion.

The fourth assertion essentially follows from the definitions of $[\Sigma]$ and $x_{\C P^2}$:
the composition 
$$H_2(N_{D_\gamma}^\infty) \xrightarrow{p_*} H_2(N_{D_\gamma}) \xrightarrow{\phi_{D_\gamma},\cong}  H_2(\cpring)$$
maps the surface $\Sigma=S \cup_{\widetilde{\mu}_c} \cocore(\widetilde{h}^2)$ to a surface in~$\cpring$ of the form $\cocore(h^2) \cup_{\mu_c} F$ where~$F \looparrowright E_J \subset S^3$ is an immersed surface with boundary $\mu_c$; as mentioned in Construction~\ref{cons:H2CP2},  such a closed immersed surface represents the generator~$x_{\C P^2} \in H_2(\cpring)$.
\end{proof}

The next proposition uses work of Borodzik and Friedl~\cite{BorodzikFriedlLinking} to show that elements of $\mathcal{G}_K$ can be realized by crossing change $\Z$-disks.

\begin{proposition}
\label{prop:BF}
For every knot $K$ that is $\Z$-slice in $\cpring$ and every $x \in \mathcal{G}_K$, there exists a crossing change $\Z$-disk~$D_\gamma$ with  boundary $K$ such that $[\widetilde{\gamma}]=x$.
\end{proposition}
\begin{proof}
Work of Borodzik and Friedl implies that for every $x \in \mathcal{G}_K$,  there exists an embedded disk~$\mathcal{D} \subset S^3$ that intersects $K$ in two transverse intersections of opposite signs,  such that the lift of $\gamma:=\partial \mathcal{D}$ to the infinite cyclic cover satisfies $[\widetilde{\gamma}]=x$ and such that a positive crossing change using~$\gamma$ leads to an Alexander polynomial one knot~\cite[Lemma 5.5 and Proof of Theorem~5.1]{BorodzikFriedlLinking}.
In particular after choosing an ambient isotopy~$(f_t)_{t \in [0,1]}$ of~$S^3$ with $f_1(\gamma)=c$, we see that every~$x \in \mathcal{G}_K$ leads to a crossing change $\Z$-disk with~$[\widetilde{\gamma}]=x$.

We include some more details on how to apply~\cite{BorodzikFriedlLinking}.
First, note that Borodzik-Friedl work with a representative of the Alexander polynomial that evaluates to $1$ at $t=1$~\cite[Example~2.3]{BorodzikFriedlLinking} whereas for us, $\Delta_K(1)=-1$.
Next, we note that in~\cite[Section 1.3]{BorodzikFriedlLinking}, the authors use the opposite sign convention than ours in their definition of a matrix presentating a linking form: contrarily to~\eqref{eq:presents}, no minus sign appears.

We now give the argument proper. 
Proposition~\ref{prop:BelongstoG} ensures that $\Bl_K(x,x)=\tmfrac{1}{\Delta_K}=\tmfrac{\unaryminus1}{\unaryminus\Delta_K}$.
We apply~\cite[Lemma 5.5]{BorodzikFriedlLinking} (in their notation, we take~$n=1$ and~$p_{11}(t)=\unaryminus1$):
the outcome is a based embedded disk $\mathcal{D} \subset S^3$ that intersects~$K$ in two transverse intersections of opposite signs and whose boundary $\gamma=\partial \mathcal{D}$ lifts to a loop that represents $x$ and has equivariant self-linking~$\frac{\unaryminus1}{\unaryminus\Delta_K}=\tmfrac{1}{\Delta_K} \in \Q(t)$.
The argument in~\cite[Proof of Theorem~5.1]{BorodzikFriedlLinking}  shows that the knot obtained by the resulting positive crossing change has Alexander polynomial one.
In their notation we take $A(t)$ the size one matrix $(\Delta_K)$ (as this presents the Blanchfield pairing with their sign conventions) and $\varepsilon_1=A_{11}(1)=\Delta_K(1)=-1$.
\end{proof}

The next remark describes the relation between the homology of $N_D^\infty$ and the homology of $N_D$.
For this, we write $\Z_{\aug}$ for the $\Z[t^{\pm 1}]$-module whose underlying abelian group is $\Z$ and where the action is by $p(t) \cdot n:=p(1)n$ for $p(t) \in \Z[t^{\pm 1}]$ and $n \in \Z$.
In other words $\Z_{\aug}$ has the~$\Z[t^{\pm 1}]$-module induced by the augmentation map $\aug \colon \Z[t^{\pm 1}] \to \Z,p(t) \mapsto p(1).$

\begin{remark}
\label{rem:BF}
Given an infinite cover $p \colon X^\infty \to X$ associated to an epimorphism $\varphi \colon \pi_1(X) \twoheadrightarrow \Z$,  it is customary to identify $H_*(X;\Z_{\aug}):=H_*(X;\Z[t^{\pm 1}] \otimes_{\Z[t^{\pm 1}]} \Z_{\aug})$ with $H_*(X)$.
This isomorphism is induced by the projection $p$ and, in particular,  the following diagram commutes:
$$
\xymatrix@R0.5cm{
H_k(X^\infty)  \ar[rr]^{p_*} \ar[d]^=&& H_k(X) \\
H_k(X;\Z[t^{\pm 1}]) \ar[rr]^-{[\sigma] \mapsto [\sigma \otimes 1]} \ar[rd]^-{[\sigma] \mapsto [\sigma] \otimes 1} && H_k(X;\Z_{\aug})
 \ar[u]^\cong_{p_*} \\
 &H_k(X;\Z[t^{\pm 1}]) \otimes_{\Z[t^{\pm 1}]} \Z_{\aug}. \ar[ur]_-{\psi}&
}
$$
The map labelled $\psi$ maps $[\sigma] \otimes n$ to $[\sigma \otimes n]$ and its kernel and cokernel can be analysed using the universal coefficient spectral sequence.
When $X=N_D$ is a $\Z$-disk exterior and $k=2$,
this map is an isomorphism~\cite[proof of Lemma 5.10]{ConwayPowell}.

The upshot of this discussion is that will frequently use the identification 
$$p_* \circ \psi \colon  H_2(N_D;\Z[t^{\pm 1}]) \otimes_{\Z[t^{\pm 1}]} \Z_{\aug}  \xrightarrow{\cong} H_2(N_D)$$
and, given $x \in H_*(N_D^\infty)=H_k(N_D;\Z[t^{\pm 1}]) $, we will identify $x \otimes 1 \in H_2(N_D;\Z[t^{\pm 1}]) \otimes_{\Z[t^{\pm 1}]} \Z_{\aug}$ with $p_*(x) \in H_2(N_D)$.

For example,  using these identifications, the fourth item of Proposition~\ref{prop:BelongstoG} states that 
$$\phi_{D_\gamma}([\Sigma] \otimes 1)=x_{\C P^2}.$$
\end{remark}

\section{The classification of $\Z$-disks}
\label{sec:ClassifyCrossingChange}

The goal of this section is to prove Theorem~\ref{thm:DiscsWithoutCPP22Intro} from the introduction.
After introducing some notation, we extract some results from~\cite{ConwayPowell} that provide the main ingredients for the proof. 

\begin{notation}
In what follows,  given a $\Z$-disk $D \subset  \cpring$ with boundary a knot $K$, we use the composition $\ev \circ \PD^{-1}$ of the Poincar\'e duality isomorphism with evaluation to identify~$H_2(N_D,M_K;\Z[t^{\pm 1}])$ with
$$H_2(N_D;\Z[t^{\pm 1}])^*:= \overline{\Hom_{\Z[t^{\pm 1}]}(H_2(N_D;\Z[t^{\pm 1}]),\Z[t^{\pm 1}])}.$$
The fact that $\ev$ is an isomorphism can be seen using the universal coefficient spectral sequence; we refer to~\cite[proof of Lemma 3.2]{ConwayPowell} for the details.

Given a basis~$e$ of~$H_2(N_D;\Z[t^{\pm 1}]) \cong \Z[t^{\pm 1}]$, we write~$e^* \in H_2(N_D;\Z[t^{\pm 1}])^*$ for the dual basis.
From now on, we use $\ev \circ \PD^{-1}$ to view $e^*$ as an element of $H_2(N_D,M_K;\Z[t^{\pm 1}])$.
We also write
$$\phi_D \colon H_2(N_D) \to H_2(\cpring)$$
 for the inclusion induced map.
 Continuing with the notation from Remark~\ref{rem:BF},  we note that any~$\sigma \in H_2(N_D;\Z[t^{\pm 1}])$ determines an element $\phi_D(\sigma \otimes 1) \in H_2(N_D).$
 \end{notation}

The next proposition uses results from~\cite{ConwayPowell} and~\cite{OrsonPowell} to formulate a criterion ensuring that two~$\Z$-disks in $\cpring$ are isotopic rel.\ boundary.

\begin{proposition}
\label{prop:CP20BasisVersion}
Assume that~$D_0,D_1 \subset \cpring$ are~$\Z$-disks with boundary a knot~$K$.
The following assertions are equivalent:
\begin{enumerate}[(i)]
\item \label{item:TFAEIsotopic} 
the disks~$D_0$ and~$D_1$ are isotopic rel.\ boundary;
\item \label{item:TFAEGenerator} 
there are generators~$e_0 \in H_2(N_{D_0}^\infty)$ and~$e_1 \in H_2(N_{D_1}^\infty)$ such that for some $k \in \Z$
\begin{align*}
&\phi_{D_0}(e_0 \otimes 1)=\phi_{D_1}(e_1 \otimes 1) \in H_2(\C P^2), \\
&t^k\partial_{D_0}(e_0^*)=\partial_{D_1}(e_1^*) \in H_1(E_K^\infty).
\end{align*}
\end{enumerate}
\end{proposition}
\begin{proof}
We start with a claim that we will be using again later on.
\begin{claim}
\label{claim:IsotopicToIdOrNot}
Assume that~$\Phi \colon (\cpring,D_0) \to (\cpring,D_1)$ is a homeomorphism rel.\  boundary.
For any generators $x_0 \in H_2(N_{D_0}^\infty)$ and $x_1 \in H_2(N_{D_1}^\infty)$, there is a sign $\varepsilon=\pm 1$ so that 
\begin{align*}
&\phi_{D_0}(x_0 \otimes 1)=\operatorname{sign}(\Phi)  \varepsilon \phi_{D_1}(x_1 \otimes 1), \\
& \partial_{D_0}(x_0^*)=\varepsilon t^k \partial_{D_1}(x_1^*)
\end{align*}
where $\operatorname{sign}(\Phi)=1$ if $\Phi \colon \C P^2 \to \C P^2$ induces the identity on $H_2$ and $\operatorname{sign}(\Phi)=-1$ otherwise.
\end{claim}
\begin{proof}[Proof of Claim~\ref{claim:IsotopicToIdOrNot}]
Restricting~$\Phi$ to the disk exteriors, lifting the outcome to the covers, and taking the induced map~$\widetilde{\Phi}_*$ on second homology, we obtain a~$\Z[t^{\pm 1}]$-isomorphism.
As these modules are free of rank one, for any generators~$x_0 \in H_2(N_{D_0}^\infty)$ and~$x_1 \in H_2(N_{D_1}^\infty)$, there is a~$\varepsilon=\pm 1$ and a~$k \in \Z$ so that
$$ \widetilde{\Phi}_*(x_0)=\varepsilon t^kx_1.$$
Apply the projection induced map~$(p_1)_* \colon H_2(N_{D_1}^\infty) \to H_2(N_{D_1})$ to obtain~$\Phi \circ (p_0)_*(x_0)=p_1(\varepsilon t^k x_1).$
Recalling the notation from Remark~\ref{rem:BF} we rewrite this as $\Phi(x_0 \otimes 1)=\varepsilon t^k x_1 \otimes 1.$ 
By definition of the augmentation map, this can be rewritten as $\Phi(x_0 \otimes 1)=\varepsilon x_1 \otimes 1.$ 
Apply $\phi_{D_1}$ to obtain the first claimed equality:
\begin{equation*}
\varepsilon \phi_{D_1}(x_1 \otimes 1) =\phi_{D_1} \circ \Phi(x_0 \otimes 1)= \operatorname{sign}(\Phi)  \phi_{D_0}(x_0 \otimes 1).
\end{equation*}
For the second equality, since $\Phi$ restricts to the identity on the boundary,  we have
\begin{equation*}
 \partial_{D_0}(x_0^*)
=\partial_{D_1} \circ \widetilde{\Phi}_*(x_0^*)
=\varepsilon t^k \partial_{D_1}(x_1^*).
\end{equation*}
This concludes the proof of Claim~\ref{claim:IsotopicToIdOrNot}.
\end{proof}

We now prove the~$\eqref{item:TFAEIsotopic}\Rightarrow\eqref{item:TFAEGenerator}$ direction of the proposition.
Assume that~$\Phi \colon (\cpring,D_0) \to (\cpring,D_1)$ is an ambient isotopy rel.\  boundary and pick generators~$x_0 \in H_2(N_{D_0}^\infty)$ and $x_1 \in H_2(N_{D_1}^\infty)$.
Since~$\Phi \colon \cpring \to \cpring$ induces the identity on $H_2$,
Claim~\ref{claim:IsotopicToIdOrNot} states that there is a sign $\varepsilon=\pm 1$ so that 
\begin{align*}
&\phi_{D_0}(x_0 \otimes 1)=\varepsilon \phi_{D_1}(x_1 \otimes 1), \\
& \partial_{D_0}(x_0^*)=\varepsilon t^k \partial_{D_1}(x_1^*).
\end{align*}
It follows that the generators~$e_0:=x_0$ and $e_1:=\varepsilon x_1$ satisfy the relations stated in~\eqref{item:TFAEGenerator}.

Next we prove the converse, namely we assume that there are~$\Z[t^{\pm 1}]$-generators~$e_0 \in H_2(N_{D_0}^\infty)$ and~$e_1 \in H_2(N_{D_1}^\infty)$ as in~\eqref{item:TFAEGenerator} and prove that the disks are isotopic rel.\  boundary.
The idea is as follows: the first equation in~\eqref{item:TFAEGenerator} is used to show that the disks are equivalent rel.\ boundary; the second condition is used to upgrade this equivalence to an isotopy.

Define an isomorphism $F \colon H_2(N_{D_0};\Z[t^{\pm 1}])  \to H_2(N_{D_1};\Z[t^{\pm 1}])$ by setting~$F(e_0):=t^{-k}e_1$ and extending $\Z[t^{\pm 1}]$-linearly.
Note that~$F$ is automatically an isometry because
\begin{align*}
& \lambda_{N_{D_0}}(e_0,e_0)=\unaryminus\Delta_K,\\
&  \lambda_{N_{D_1}}(F(e_0),F(e_0)) 
= \lambda_{N_{D_1}}(t^{-k}e_1,t^{-k}e_1)
= \lambda_{N_{D_1}}(e_1,e_1)
=\unaryminus\Delta_K.
\end{align*}
Here we used that~$\lambda_{N_{D_i}}(x,x)=\unaryminus\Delta_K$ for any generator~$x \in H_2(N_{D_i};\Z[t^{\pm 1}]) \cong \Z[t^{\pm 1}]$; this is a consequence of Lemma~\ref{lem:DiscsGeneral}.

Using the identification~$H_2(N_{D_i};\Z[t^{\pm 1}])^* \cong H_2(N_{D_i},M_K;\Z[t^{\pm 1}])$ for~$i=0,1$,  we think of the dual isomorphism as a map~$F^* \colon H_2(N_{D_1},M_K;\Z[t^{\pm 1}]) \to H_2(N_{D_0},M_K;\Z[t^{\pm 1}])$.
Applying the connecting homomorphism $\partial_{D_1}$ to the equality $(F^*)^{-1}(e_0^*)=t^{-k}e_1^*$,  we obtain
$$\partial_{D_1} \circ (F^*)^{-1}(e_0^*)=t^{-k}\partial_{D_1}(e_1^*)
=\partial_{D_0}(e_0^*).$$
In the second equality,  we used our assumption that $t^k\partial_{D_0}(e_0^*)=\partial_{D_1}(e_1^*) $.

Since $F$ is an isometry that satisfies $\partial_{D_1} \circ (F^*)^{-1}=\partial_{D_0}$, we can apply~\cite[Theorem~1.3]{ConwayPowell} to deduce that there is an equivalence rel.\  boundary~$\Phi \colon (\cpring,D_0) \cong (\cpring,D_1)$ realizing~$F$.

It remains to verify that the homeomorphism~$\Phi $ is isotopic to the identity.
As shown by~\cite[Lemma~5.10]{ConwayPowell},  any isometry $G \colon \lambda_{N_{D_0}} \cong  \lambda_{N_{D_1}}$ induces a self-isomorphism~$G_\Z$ on $H_2(\cpring)$ that satisfies $G_\Z \circ \phi_{D_0}=\phi_{D_1}\circ (G \otimes_{\Z[t^{\pm 1}]} \id_{\Z_{\aug}})$.
The same reference states that if $G$ is induced by a homeomorphism~$\Psi \colon (\cpring,D_0) \xrightarrow{\cong} (\cpring,D_1)$, then $\Psi_*=G_\Z$.

Returning to our setting, we endow~$H_2(\cpring)$ with the basis~$\phi_{D_0}(e_0 \otimes 1)$ and note that
$$F_\Z \circ \phi_{D_0}(e_0 \otimes 1)
=\phi_{D_1} \circ (F \otimes_{\Z[t^{\pm 1}]} \id_{\Z_{\aug}})(e_0 \otimes 1)
=\phi_{D_1}(t^{-k}e_1 \otimes 1)
=\phi_{D_1}(e_1 \otimes 1)
=\phi_{D_0}(e_0 \otimes 1).$$
In the last equality we used the assumption from~\eqref{item:TFAEGenerator}. 
This calculation implies that the rel.\ boundary homeomorphism~$\Phi \colon \cpring \to \cpring$ induces the identity on $H_2$.
Work of Orson and Powell~\cite[Corollary C]{OrsonPowell} now implies that $\Phi$ is isotopic rel.\ boundary to the identity.
This concludes the proof that $D_0$ and $D_1$ are isotopic rel.\ boundary.
\end{proof}

The next result applies the criterion from Proposition~\ref{prop:CP20BasisVersion} to generalized crossing change $\Z$-disks.

\begin{proposition}
\label{prop:CharacCrossingChangeDisc}
Assume that~$D_0,D_1 \subset \cpring$ are generalized crossing change~$\Z$-disks for a knot~$K$, with respective surgery curves~$\gamma_0,\gamma_1 \subset E_K$.
Let~$\widetilde{\gamma}_0,\widetilde{\gamma}_1$ be the lifts of~$\gamma_0,\gamma_1$ to the infinite cyclic cover.
The following assertions are equivalent:
\begin{enumerate}[(i)]
\item the disks~$D_0$ and~$D_1$ are isotopic rel.\ boundary;
\item \label{item:TFAELifts} 
 the homology classes~$[\widetilde{\gamma}_0],[\widetilde{\gamma}_1] \in \mathcal{G}_K \subset H_1(E_K^\infty)$  satisfy~$[\widetilde{\gamma}_0]=t^k [\widetilde{\gamma}_1]$ for some $k \in \Z$.
\end{enumerate}
\end{proposition}
\begin{proof}
For $i=0,1,$  consider the generator~$e_i:=[\Sigma_i]$ for $H_2(N_{D_i}^\infty)$ from Construction~\ref{cons:Basis} and recall from Proposition~\ref{prop:BelongstoG} that this generator satisfies~$\partial_{D_i}(e_i^*)=[\widetilde{\gamma}_i]$.
We also noted in Remark~\ref{rem:BF} that $\phi_{D_i}(e_i \otimes 1)=x_{\C P^2}$, the preferred homology class $x_{\C P^2} \in H_2(\cpring)$ from Construction~\ref{cons:H2CP2}.

We can now prove the (ii) $\Rightarrow$ (i) direction: we have $\phi_{D_0}(e_0 \otimes 1)=x_{\C P^2}=\phi_{D_1}(e_1 \otimes 1)$ and~$\partial_{D_0}(e_0^*)=[\widetilde{\gamma}_0]=t^k [\widetilde{\gamma}_1]=t^k \partial_{D_1}(e_1^*)$ and therefore Proposition~\ref{prop:CP20BasisVersion} implies that $D_0$ and $D_1$ are isotopic rel.\ boundary.

Next we consider the (i) $\Rightarrow$ (ii) direction. 
Assume the disks are isotopic rel.\ boundary via a homeomorphism~$\Phi \colon (\cpring,D_0) \xrightarrow{\cong} (\cpring,D_1)$.
Claim~\ref{claim:IsotopicToIdOrNot}
shows that for any generators $x_0 \in H_2(N_{D_0}^\infty)$ and $x_1 \in H_2(N_{D_1}^\infty)$, there is a sign $\varepsilon=\pm 1$ so that 
\begin{align*}
&\varepsilon \phi_{D_1}(x_1 \otimes 1)= \phi_{D_0}(x_0 \otimes 1),  \\
& \partial_{D_0}(x_0^*)=\varepsilon t^k \partial_{D_1}(x_1^*).
\end{align*}
Since $\phi_{D_0}(e_0 \otimes 1)=x_{\C P^2}=\phi_{D_1}(e_1 \otimes 1)$,  if we apply these equalities to $x_0=e_0$ and $x_1=e_1$, we deduce that $\varepsilon=1$.
It follows that 
$
[\widetilde{\gamma}_0] 
=\partial_{D_0}(e_0^*)
=t^k \partial_{D_1}(e_1^*)
= t^k[\widetilde{\gamma}_1],
$
as required.
This concludes the proof of the proposition.
\end{proof}

We now prove the first main result of this section.

\begin{theorem}
\label{thm:GeneralisedCrossingChange}
Let $K$ be a knot.
Every~$\Z$-disk in~$\cpring$ with boundary $K$ is isotopic rel.\ boundary to a crossing change~$\Z$-disk $D_\gamma$.
\end{theorem}
\begin{proof}
Let~$D \subset \cpring$ be a~$\Z$-disk with boundary~$K$.
We are going to prove that $D$ is isotopic rel.\ boundary to a crossing change $\Z$-disk.
Fix a generator $y \in H_2(N_D^\infty)$.
The idea of the proof is that the work of Borodzik-Friedl~\cite{BorodzikFriedlLinking} (as recast in Proposition~\ref{prop:BF}) implies that the projection of a loop representing~$\partial_D([y^*]) \in H_1(M_K^\infty) \cong H_1(E_K^\infty)$ will give the required surgery curve.

Use $p \colon  N_{D^\infty} \to N_D$  to denote the covering map and consider the portion
\begin{equation}
\label{eq:Milnor}
\cdots \to H_2(N_D;\Z[t^{\pm 1}]) \xrightarrow{\cdot (t-1)} H_2(N_D;\Z[t^{\pm 1}])  \xrightarrow{p_*}  H_2(N_D) \to H_1(N_D;\Z[t^{\pm 1}]) \to \cdots
\end{equation}
of the exact sequence induced by the short exact sequence of coefficients
$$0 \to \Z[t^{\pm 1}] \xrightarrow{\cdot (t-1)} \Z[t^{\pm 1 }] \to \Z_{\aug}  \to 0.$$
In the third term of~\eqref{eq:Milnor},  we identified~$H_2(N_D)$ with~$H_2(N_D;\Z_{\aug})$ as explained in Remark~\ref{rem:BF}. 
This sequence is sometimes called the \emph{Milnor exact sequence} because of its use in~\cite{MilnorInfiniteCyclic}.

Since~$H_1(N_D;\Z[t^{\pm 1}]) =H_1(N_D^\infty)=0$,  the projection induced map~$p_* \colon H_2(N_D^\infty) \to H_2(N_D)$ is surjective and therefore~$p_*(y) \in H_2(N_D)\cong \Z$ is a generator.
Since~$\phi_D \colon H_2(N_D) \to H_2(\cpring)$ is an isomorphism, it follows that~$\phi_D(p_*(y))=\pm x_{\C P^2} \in H_2(\cpring) \cong \Z$.
Picking the negative of~$y$ if necessary,  we can therefore assume that~$\phi_D(p_*(y))= x_{\C P^2}$.

Using the connecting homomorphism~$\partial_D$ in the long exact sequence of the pair $(N_D^\infty,M_K^\infty)$,  we see that the generator~$y^* \in H_2(N_D^\infty,M_K^\infty)$ determines a generator~$x:=\partial_D([y^*]) \in H_1(M_K^\infty)$; this uses the fact that $\partial_D$ is surjective.
Since $y$ generates $H_2(N_D^\infty) \cong \Z[t^{\pm 1}]$, Proposition~\ref{prop:barCP2} shows that~$\lambda_{N_D}(y,y)=\unaryminus\Delta_K$.
Claim~\ref{claim:UsualTrick} implies that 
$$\Bl_K(x,x) \equiv \unaryminus\lambda_{N_D}(y,y)^{-1}=\frac{1}{\Delta_K}$$
 and therefore~$x \in \mathcal{G}_K$.
As mentioned in Proposition~\ref{prop:BF},  work of Borodzik-Friedl implies that there is a crossing change $\Z$-disk $D_\gamma$ with boundary $K$ and~$[\widetilde{\gamma}]=x \in \mathcal{G}_K$.

We show that $D_\gamma$ is isotopic rel.\  boundary to the $\Z$-disk $D$ that we started with.
Consider the generator $[\Sigma] \in H_2(N_{D_\gamma};\Z[t^{\pm 1}])$ from Construction~\ref{cons:Basis}.
Recall from Proposition~\ref{prop:BelongstoG} that~$\partial_{D_\gamma}([\Sigma]^*)=[\widetilde{\gamma}]$ and~$\phi_{D_\gamma}([\Sigma])=x_{\C P^2}$.
By construction of $y$,  we also have~$\partial_D(y^*)=[\widetilde{\gamma}]$ and $\phi_D([y])=x_{\C P^2}$.
Proposition~\ref{prop:CP20BasisVersion} now implies that~$D$ and~$D_\gamma$ are isotopic rel.\ boundary.
\end{proof}

Combining these results,  we obtain Theorem~\ref{thm:DiscsWithoutCPP22Intro} from the introduction.

\begin{customthm}{\ref{thm:DiscsWithoutCPP22Intro}}
\label{thm:DiscsWithoutCPP22}
Let $K$ be a knot.
\begin{enumerate}
\item  Every~$\Z$-disk in~$\cpring$ with boundary $K$ is isotopic rel.\ boundary to a crossing change~$\Z$-disk.
\item Two crossing change~$\Z$-disks with boundary $K$ are isotopic rel.\  boundary if and only if the homology classes of the lifts of their surgery curves agree in the Alexander module $H_1(E_K^\infty)$ up to multiplication by $t^k$ for some~$k \in \Z$.
\item Mapping a crossing change $\Z$-disk with boundary $K$ to the homology class of the lift of its surgery curve defines a bijection
$$\mathcal{D}_\Z(K,\cpring) \xrightarrow{\approx}  \mathcal{G}_K/\lbrace t^k \rbrace_{k \in \Z}.$$
\end{enumerate}
\end{customthm}
\begin{proof}
The first item is proved in Theorem~\ref{thm:GeneralisedCrossingChange} and the second in Proposition~\ref{prop:CharacCrossingChangeDisc}.
The third item is also a consequence of these results as we now explain.
The combination of these results shows that the map ~$\mathcal{D}_\Z(K,\cpring) \to \mathcal{G}_K/\lbrace t^k \rbrace_{k \in \Z}$ is well defined and injective: Theorem~\ref{thm:GeneralisedCrossingChange} shows that every $\Z$-disk is isotopic rel.\ boundary to a crossing change $\Z$-disk and Proposition~\ref{prop:CharacCrossingChangeDisc} states that this disk is uniquely determined by the homology class of the lift of the surgery curve.
Surjectivity is a consequence of Proposition~\ref{prop:BF}.
\end{proof}

\begin{remark}
\label{rem:Equivalence}
All the results in this section can be modified to involve equivalence rel.\ boundary instead of isotopy rel.\ boundary: in Propositions~\ref{prop:CP20BasisVersion} and~\ref{prop:CharacCrossingChangeDisc} and Theorem~\ref{thm:DiscsWithoutCPP22}, one substitutes multiplication (and modding out) by $t^k$ with multiplication (and modding out) by $\pm t^k$.
We leave the details to the reader, but note the key point: an equivalence between $\Z$-surfaces in $\cpring$ need not induce the identity on $H_2(\cpring)$; indeed it can also induce multiplication by~$-1$.
\end{remark}

\section{Examples of $\Z$-disks in $\cpring$}
\label{sec:Examples}

The goal of this section is to describe examples of smoothly embedded $\Z$-disks in $\cpring$ with a focus on attempting to find smooth representatives for the enumeration from Theorem~\ref{thm:UniquenessIntro}. Smooth representatives will be constructed using the following observation.

\begin{observation}\label{obsv:smooth}
Let $\gamma$ be a generalized unknotting curve for $K$ such that applying a positive generalized crossing change to $K$ along $\gamma$ yields a \textit{smoothly} $\Z$-slice knot. Then the disk $D_\gamma$ is smoothable.
\end{observation}

\subsection{Smooth realization} We now discuss some explicit examples of $\Z$-disks in $\cpring$. From now on, as in the introduction, we use $K_n$ to denote the twist knot depicted in Figure~\ref{fig:twistknot}.
We recall that the Alexander polynomial of $K_n$ is
$$\Delta_n:=nt-(2n+1)+nt^{-1}.$$

We start with the trefoil $K_{-1}$, which 
bounds a unique $\Z$-disk in $\cpring$. This disk can be realized smoothly.

\begin{example}
\label{ex:Trefoil}
Theorem~\ref{thm:UniquenessIntro} predicts that the trefoil knot~$K_{-1}$ bounds a unique $\Z$-disk in $\cpring$ up to isotopy rel.\ boundary. 
The fact that the disk can be realized smoothly is well-known; we illustrate a crossing change curve $\gamma$ for $K_{-1}$ with $D_\gamma$ smoothly embedded (according to Observation~\ref{obsv:smooth}) in Figure \ref{fig:trefoil}.

\begin{figure}[!htbp]
    \includegraphics[width=25mm]{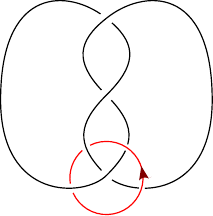}
    \caption{A positive crossing change curve for $K_{-1}$. By Observation \ref{obsv:smooth}, the resulting disk in $\cpring$ can be taken to be
smoothly embedded.}
    \label{fig:trefoil}
\end{figure}
\end{example}

As discussed in Section~\ref{sec:CrossingChange}, in order to define $D_\gamma$ the unknotting curve $\gamma$ must be oriented. 
Reversing orientation on $\gamma$ gives a disk $D_{-\gamma}$ which need not be isotopic rel.\ boundary to~$D_\gamma$.
From this perspective, it may seem counterintuitive that $K_{-1}$ bounds a unique $\Z$-disk. However, in the case of the trefoil, the curves $\gamma$ and $-\gamma$ are in fact isotopic in the complement of $K_{-1}$, as illustrated in Figure \ref{fig:trefoildiscreverse}. Thus, the disks $D_{\gamma}$ and $D_{-\gamma}$ are even smoothly isotopic rel.\ boundary. 

\begin{figure}[!htbp]
    \includegraphics[width=80mm]{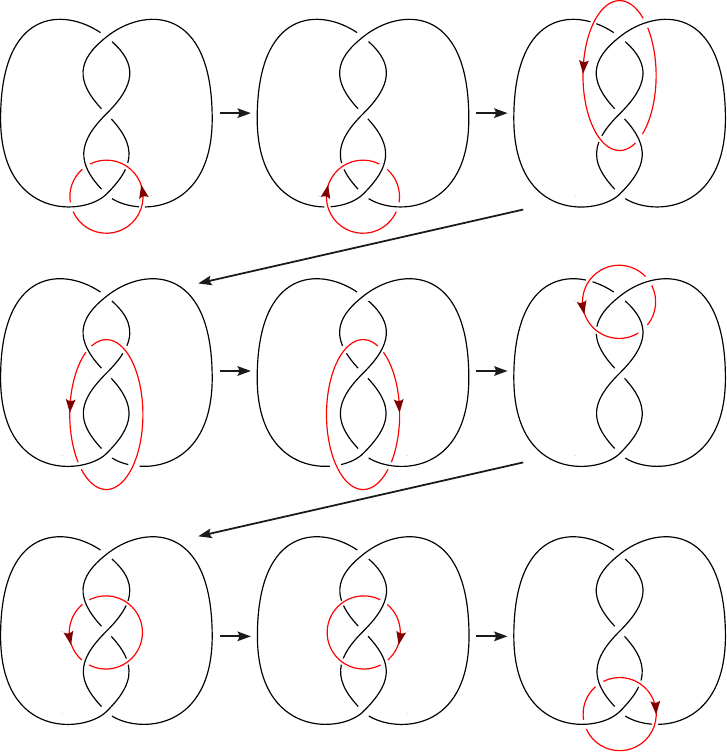}
    \caption{From left to right, top to bottom, we illustrate an isotopy from $\gamma$ to~$-\gamma$ in the complement of $K_{-1}$.}
    \label{fig:trefoildiscreverse}
\end{figure}

\begin{remark}
\label{rem:pmdistinctExampleSection}
In Proposition~\ref{prop:pmdistinct} we show that if $\Delta_K = \Delta_n$, then the two disks $D_{\pm \gamma}$ are always distinct except in the cases $n = 0$ and $n = -1$. In these latter cases, $D_{\pm \gamma}$ are always topologically isotopic rel.\ boundary (for any $\gamma$). 
\end{remark}

\begin{example}
\label{ex:FigureEight}
Theorem~\ref{thm:UniquenessIntro} predicts that the figure eight knot~$K_1$ bounds two~$\Z$-disks in~$\cpring$ up to isotopy rel.\ boundary. 
Figure \ref{fig:fig8} shows an unoriented unknotting curve~$\gamma$ for~$K_1$. 
The two possible orientations of~$\gamma$ yield~$\Z$-disks for~$K_1$ in~$\cpring$ that (by Proposition~\ref{prop:pmdistinct}) are not isotopic rel.\ boundary and hence represent the two possible classes of such disks. 
Since performing a positive crossing change along~$\gamma$ yields the unknot, these two disks are smoothly embedded in~$\cpring$.
\end{example}

\begin{figure}[!htbp]
\centering
\includegraphics[width=30mm]{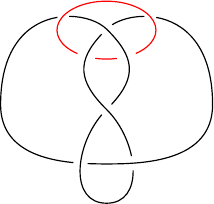}
\caption{An unoriented unknotting curve for $K_1$. The two possible choices of orientation yield the two isotopy rel.\ boundary classes of $\Z$-disk for $K_1$ in $\cpring$.}
\label{fig:fig8}
\end{figure}

\begin{example}\label{ex:knsquared}
We now consider the twist knots~$K_n$ with $n=-k^2$ and $k>1$; the case $k = 1$ has been treated in Example~\ref{ex:Trefoil}.
Theorem~\ref{thm:UniquenessIntro} predicts that, up to isotopy rel.\ boundary, $K_n$ bounds four disks when $k$ is a prime power, and infinitely many $\Z$-disks otherwise.
We realize two of these disks smoothly and, for $k=2$, we realize all four of the disks smoothly. 

The left hand side of Figure \ref{fig:kminus4example} shows two generalized unknotting curves $\gamma_1$ and $\gamma_2$ for the knot~$K_{-k^2}$. 
The fact that $\gamma_1$ is an unknotting curve is evident, but we show explicitly in Figure~\ref{fig:kminus4example} that a positive generalized crossing change along $\gamma_2$ results in an Alexander polynomial one knot.
Specifically, performing a generalized crossing change along $\gamma_2$ yields the negative untwisted Whitehead double of the torus knot $T(k,k-1)$. 
The disks $D_{\pm\gamma_1}$ are smoothly embedded by Observation~\ref{obsv:smooth}. 
However, we can only ascertain that $D_{\pm\gamma_2}$ are smoothly embedded when~$k=2$, since in this case~$T(k,k-1)$ is the unknot.

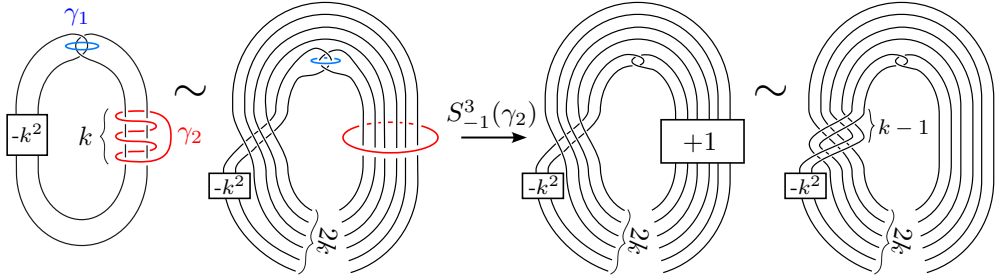
\begin{figure}[!htbp]
\labellist
\pinlabel{\footnotesize{-$k^2$}} at 15 102
\pinlabel{$k$} at 56 100
\pinlabel{\textcolor{blue}{$\gamma_1$}} at 50 185
\pinlabel{\textcolor{red}{$\gamma_2$}} at 130 100
\pinlabel{\footnotesize{-$k^2$}} at 158 65
\pinlabel{$S^3_{-1}(\gamma_2)$} at 343 115
\pinlabel{\huge{$\sim$}} at 130 130
\pinlabel{$+1$} at 490 95
\pinlabel{\footnotesize{-$k^2$}} at 380 65
\pinlabel{\footnotesize{-$k^2$}} at 565 65
\pinlabel{\rotatebox{-100}{$2k$}} at 227 25
\pinlabel{\rotatebox{-100}{$2k$}} at 447 25
\pinlabel{\rotatebox{-100}{$2k$}} at 633 25
\pinlabel{\footnotesize{$k-1$}} at 634 106
\pinlabel{\huge{$\sim$}} at 540 130
\endlabellist
    \includegraphics[width=130mm]{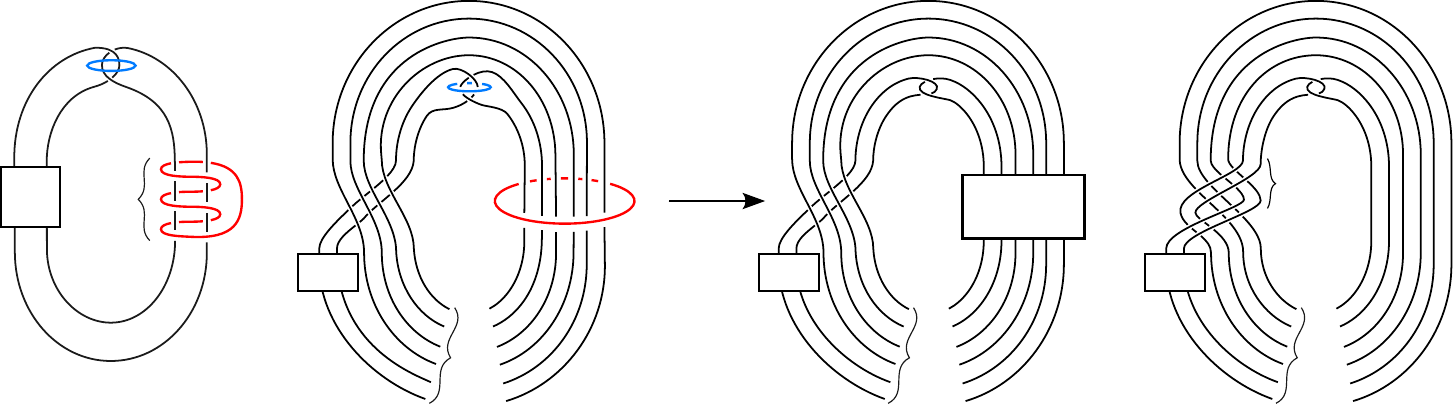}
    \caption{
\textbf{Left:} Two surgery curves $\gamma_1,\gamma_2$ for $K_{n}$ ($n=-k^2, k>1$). 
\textbf{Right:} Adding a positive twist about $\gamma_2$ transforms $K_n$ into the untwisted negative Whitehead double of $T_{k,k-1}$.}\label{fig:knsquaredexample}
\end{figure}

We claim that $D_{\pm \gamma_1}$ and $D_{\pm \gamma_2}$ represent four distinct topological isotopy classes.
For this, after fixing a basepoint, we determine lifts of $\pm \gamma_1$ and $\pm \gamma_2$ to $E_{K_{-k^2}}^\infty$. Figure \ref{fig:kminus4example} illustrates this for~$k=2$; the general case is similar. For an appropriate choice of the generator of the Alexander module~$H_1(E_{K_{-k^2}}^\infty)$, these lifts are given by:
\begin{align*}
[\pm \widetilde{\gamma}_1]&=\pm(nt-2n+nt^{-1}) \equiv \mp 1,  \\
[\pm \widetilde{\gamma}_2]&=\pm(kt-k).
\end{align*}
In Proposition~\ref{prop:fourdistinct}, we show that these are distinct elements of $\mathcal{G}_{K_n}/\{ t^\ell \}_{\ell \in \Z}$. If $k$ is a prime power, then by Theorem~\ref{thm:UniquenessIntro}, this realizes all of the $\Z$-disks for $K_n$ up to isotopy rel.\ boundary. 
\begin{figure}[!htbp]
\labellist
\pinlabel{\textcolor{gray}{$+1$}} at 156 110
\pinlabel{\textcolor{blue}{$\gamma_1$}} at 85 107
\pinlabel{\textcolor{red}{$\gamma_2$}} at 60 10
\pinlabel{\textcolor{gray}{$+1$}} at 270 35
\pinlabel{\rotatebox{120}{$-4$}} at 302 107
\pinlabel{\textcolor{gray}{$-7$}} at 500 15
\pinlabel{\textcolor{gray}{$-7$}} at 500 55
\pinlabel{\textcolor{gray}{$-7$}} at 500 105
\pinlabel{\textcolor{gray}{$-7$}} at 500 145
\pinlabel{\rotatebox{90}{$-4$}} at 420 32
\pinlabel{\rotatebox{90}{$-4$}} at 420 80
\pinlabel{\rotatebox{90}{$-4$}} at 420 128
\endlabellist
    \includegraphics[width=130mm]{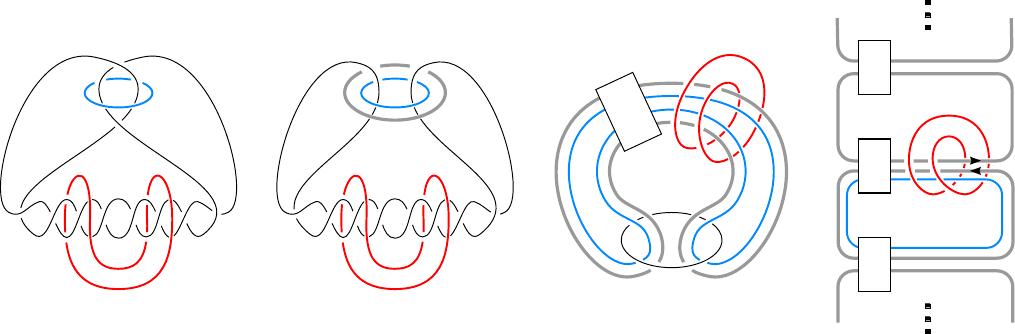}
    \caption{
\textbf{Leftmost}: Two surgery curves $\gamma_1,\gamma_2$ for $K_{-4}$. \textbf{From left to right:} Finding lifts of $\gamma_1, \gamma_2$ in $E_{K_{-4}}^\infty$. 
\textbf{Rightmost}: A surgery presentation of~$E_{K_{-4}}^\infty$; a meridian of a surgery curve represents the homology class $\pm t^k$ for some~$k$ and choosing which meridian represents~$1$ determines a basis for~$H_1(E_{K_{-4}}^\infty)$. 
For one choice of basis,  the pictured lift of $\gamma_1$ represents~$\pm(4t-8+4t^{-1})\equiv \mp 1$ while the pictured lift of~$\gamma_2$ represents~$\pm(2-2t)$.}\label{fig:kminus4example}
\end{figure}
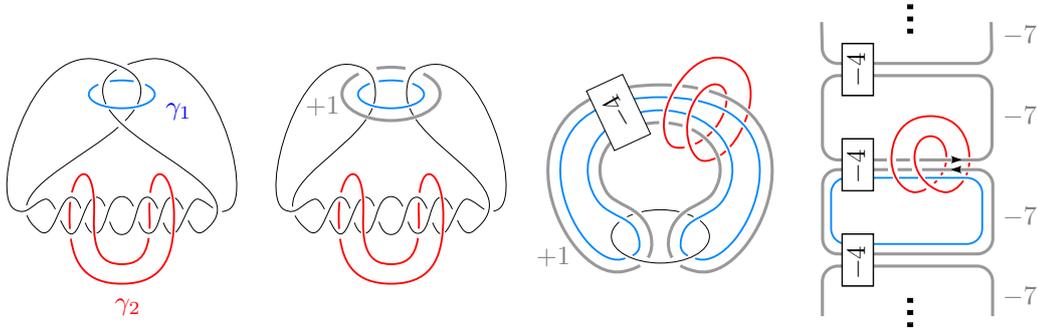
\end{example}

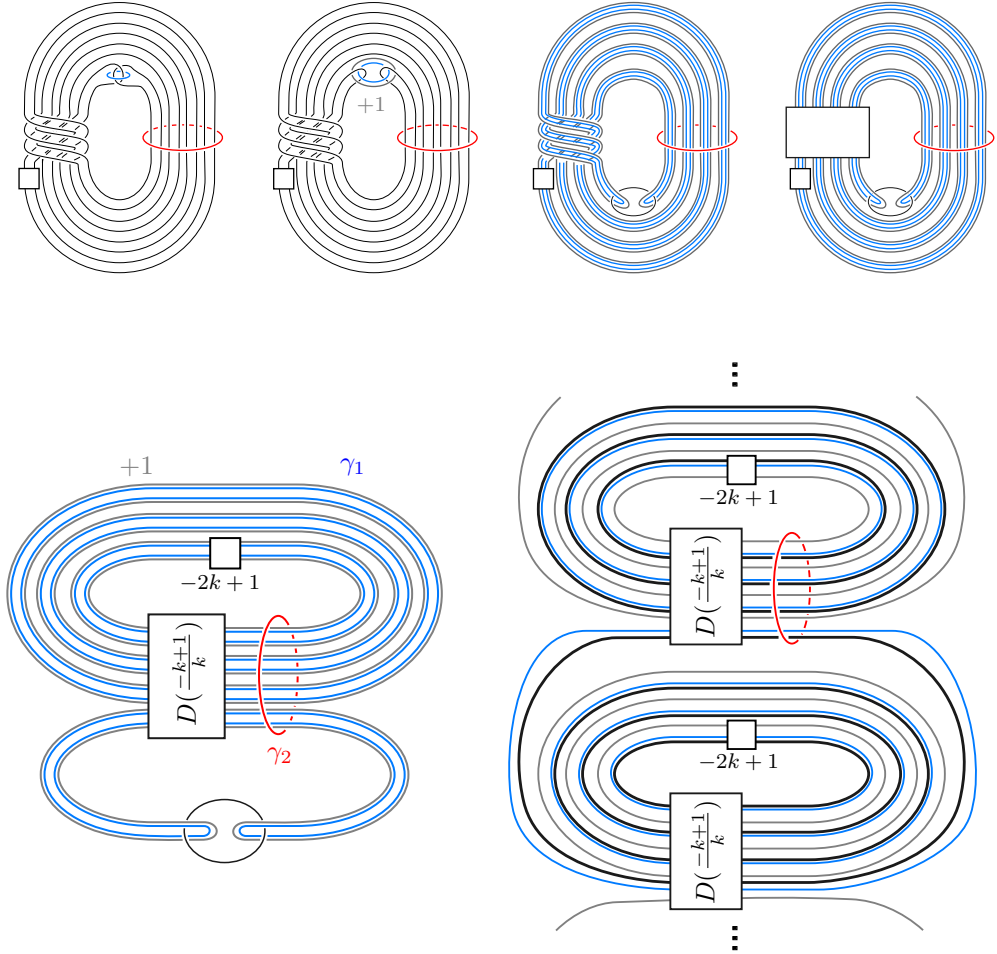
\begin{figure}[!htbp]
\labellist
\pinlabel{\footnotesize{\textcolor{gray}{$+1$}}} at 140 327
\pinlabel{\textcolor{blue}{$\gamma_1$}} at 133 187
\pinlabel{\textcolor{red}{$\gamma_2$}} at 105 75
\pinlabel{\textcolor{gray}{$+1$}} at 50 187
\pinlabel{\rotatebox{90}{$D(\frac{-k+1}{k})$}} at 70 107
\pinlabel{\footnotesize{$-2k+1$}} at 82 142
\pinlabel{\rotatebox{90}{$D(\frac{-k+1}{k})$}} at 269 141
\pinlabel{\footnotesize{$-2k+1$}} at 281 175
\pinlabel{\rotatebox{90}{$D(\frac{-k+1}{k})$}} at 269 38
\pinlabel{\footnotesize{$-2k+1$}} at 281 73
\endlabellist
    \includegraphics[width=130mm]{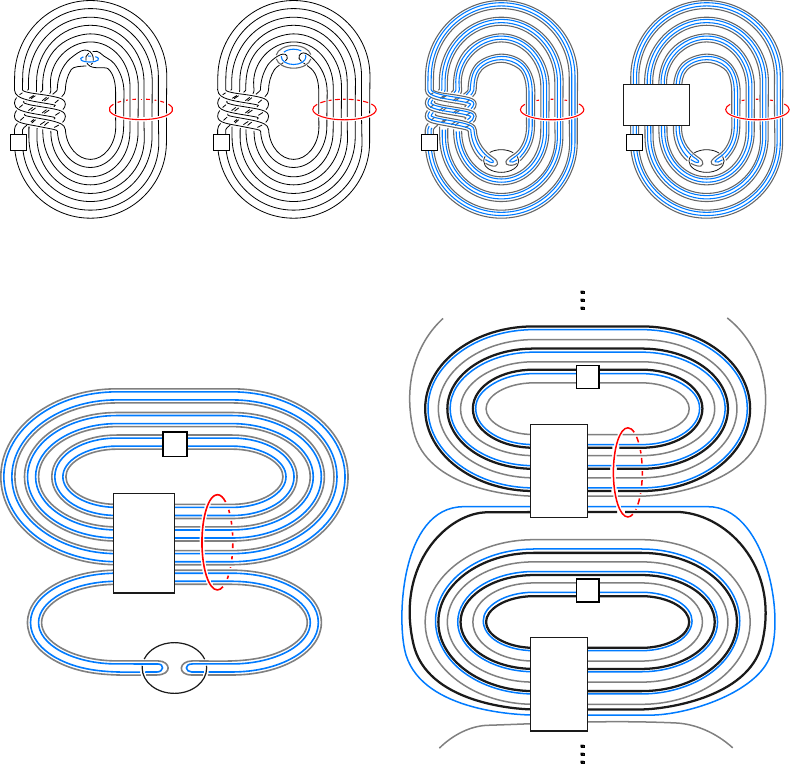}
    \caption{\textbf{Top left:} 
Two surgery curves $\gamma_1,\gamma_2$ for~$J_{-k^2}$, $k>1$, pictured for~$k=~4$.
    The square twist box contains $(k-1)^2-k^2=-2k+1$ full twists, so that the pictured knot is the $-k^2$-twisted Whitehead double $J_{-k^2}$. 
    \textbf{Top row, from left to right:} we manipulate the original figure until obtaining the bottom left configuration. 
    The rectangular box indicates the double of a $-(k-1)/k$ twist (compare the third and fourth figures in the top row). 
    \textbf{Bottom right:}
    We find lifts of $\gamma_1, \gamma_2$ (for some choice of basepoint) to $E_{J_{-k^2}}^\infty$, suppressing surgery framings. In the surgery diagram of $E_{J_{-k^2}}^\infty$, we make one surgery curve bold for clarity.  
    For one choice of basis of~$H_1(E_{J_{-k^2}}^\infty)$ (specifically, taking the meridian of the bold surgery curve to represent~$1$ and the meridian of the surgery curve obtained by translating the bold one upward represents $t$), the lift of $\gamma_1$ represents $\pm (k^2 t-2k^2+k^2 t^{-1})\equiv \mp 1$ while the lift of
    $\gamma_2$ represents $\pm(kt-k)$.
    }\label{fig:toruswhitehead}
\end{figure}

We now turn to a slightly different example which also has Alexander polynomial $\Delta_{-k^2}$.

\begin{example}\label{ex:toruswhitehead}
Let $n=-k^2$ and $k>1$. In Example \ref{ex:knsquared},  we could only ascertain that (in the general case) two of the four~$\Z$-disks bounded by~$K_n$ were smoothly embedded. 
Let~$J_n$ instead denote the~$n$-twisted negative Whitehead double of the torus knot~$T(-k,k-1)$.  
The knots~$J_n$ and~$K_n$ have the same Alexander polynomial because they are $n$-twisted Whitehead doubles.
Since~$J_n$ is a genus one knot with~$\Delta_{J_n} = \Delta_{K_n}$, Theorem~\ref{thm:UniquenessIntro} implies that~$J_n$ bounds four isotopy rel.\ boundary classes of~$\Z$-disks in~$\cpring$ if~$k$ is a prime power and infinitely many classes otherwise. We argue that, this time, we can realize the four classes smoothly. 

Applying a positive twist to either of the curves~$\gamma_1$ or~$\gamma_2$ in Figure \ref{fig:toruswhitehead} (top left) transforms~$J_n$ into the unknot. As illustrated in Figure \ref{fig:toruswhitehead}, with an appropriate choice of the generator of the Alexander module~$H_1(E_{K_{-k^2}}^\infty)$, the lifts of $\pm \gamma_1$ and $\pm \gamma_2$ are again given by 
\begin{align*}
[\pm \widetilde{\gamma}_1]&=\pm(nt-2n+nt^{-1}) \mp \pm 1,  \\
[\pm \widetilde{\gamma}_2]&=\pm(kt-k).
\end{align*}
Since $\Delta_{J_n} = \Delta_{K_n}$, an argument identical to that of Example~\ref{ex:knsquared} shows that these disks are distinct. Once again, if $k$ is a prime power, then by Theorem~\ref{thm:UniquenessIntro}, this realizes all of the $\Z$-disks for $J_n$ up to isotopy rel.\ boundary. 
\end{example}

\subsection{Toward new exotic pairs}\label{sec:exotic}


Several authors have used knot Floer homology to distinguish pairs of surfaces up to smooth isotopy rel.\ boundary \cite{JuhaszMarengon, JuhaszZemke}. These techniques have been used to detect exotic pairs of higher-genus surfaces \cite{JuhaszMillerZemke}, and, subsequently, study exotic pairs of disks \cite{DaiMallickStoffregen}. Below, we describe a prompt
generalization of the results of \cite{DaiMallickStoffregen} to the present situation; in conjunction with Theorem~\ref{thm:DiscsWithoutCPP22Intro}, we explain how this gives a hypothetical program for generating exotic pairs of disks in~$\cpring$.

Let~$\tau \colon S^3 \rightarrow S^3$ be given by~$\pi$-rotation about an unknot in~$S^3$. The involution~$\tau$ admits two extensions~$\tau_+$ and~$\tau_-$ over~$\cpring$, which are constructed as follows. In the model for~$\cpring$ described in Definition~\ref{def:CrossingChangeZdisc}, extend~$\tau$ over~$S^3 \times I$ by~$\tau \times \id$. We then place the~$2$-handle attaching curve~$c \subset S^3 \times \{1\}$ in either of the positions displayed in Figure~\ref{fig:involutions}.

\begin{figure}[!htbp]
\labellist
\pinlabel {\textcolor{darkgreen}{$-1$}} at -5 50
\pinlabel {\textcolor{darkgreen}{$-1$}} at 90 50
\pinlabel{$\tau_+$} at 15 -8
\pinlabel{$\tau_-$} at 72 -8
\endlabellist
    \includegraphics[width=50mm]{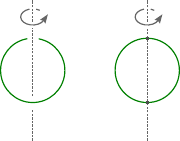}
    \vspace{.2in}
    \caption{Two involutions~$\tau_+$ and~$\tau_-$ of~$\cpring$ that both extend the involution~$\tau$ of the boundary~$S^3$. (Here we give relative handle diagrams of~$-\cpring$; reverse the orientation to obtain actual~$\cpring$.) 
 On~$H_2(\cpring;\mathbb{Z})$,~$\tau_+$ fixes the attaching circle~$c$ setwise and preserves its orientation, thus inducing the identity map. 
On the other hand,~$\tau_-$ fixes~$c$ setwise but with reversed orientation so induces multiplication by~$-1$.}\label{fig:involutions}
\end{figure}
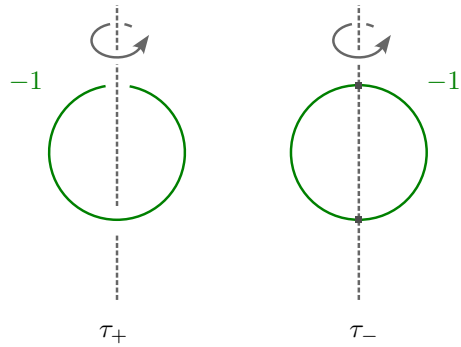

 On the left,~$\tau$ has no fixed points on~$c$ and sends~$c$ to itself as an oriented curve, while on the right,~$\tau$ has two fixed points on~$c$ and reverses orientation on~$c$. In either case, we may extend~$\tau$ over the~$2$-handle attachment and then over the capping~$4$-ball via the Alexander trick; see for example \cite[Section 5.1]{DaiHeddenMallick}. Denote the first extension by~$\tau_+$ and second by~$\tau_-$. It is readily checked that the induced actions of~$\tau_+$ and~$\tau_-$ on the second homology of~$\cpring$ are given by
\[
(\tau_+)_* = \id \quad \text{and} \quad (\tau_-)_* = - \id.
\]

Now let~$K$ be a knot which is fixed setwise by the involution~$\tau$. We allow~$\tau$ to either fix or reverse orientation on~$K$. Associated to such an equivariant knot~$(K, \tau)$, there are four numerical invariants
\[
\Vtu(K) \leq \Vtl(K) \quad \text{and} \quad \Vitu(K) \leq \Vitl(K)
\]
constructed in \cite[Theorem 1.1]{DaiMallickStoffregen} and \cite[Section 8]{DaiMallickStoffregen}. In \cite{DaiMallickStoffregen}, it is shown that these invariants obstruct \textit{isotopy-equivariant sliceness}, in the following sense: let~$\tau_{D^4}$ be any extension of~$\tau$ over~$D^4$. If any of the above numerical invariants are non-zero, then for any slice disk~$D$ with boundary~$K$, the two disks~$D$ and~$\tau_{D^4}(D)$ are not smoothly isotopic rel.\ boundary. \footnote{If~$\Sigma$ is a slice surface such that~$\Sigma$ and~$\tau_{D^4}(\Sigma)$ are isotopic rel.\ boundary, we say that~$\Sigma$ is \textit{isotopy equivariant}. Hence~$\Vtu$,~$\Vtl$,~$\Vitu$, and~$\Vitl$ obstruct the existence of an isotopy-equivariant slice disk for~$K$. See \cite[Section~2.1]{DaiMallickStoffregen}.}

Although \cite{DaiMallickStoffregen} is stated in terms of slice disks in~$D^4$, it is not hard to generalize the results of \cite[Theorem 1.1]{DaiMallickStoffregen} and \cite[Section 8]{DaiMallickStoffregen} to disks in other manifolds by tracing through the ideas of \cite[Section~7.2]{DaiMallickStoffregen}. 
Explicitly, let~$K$ be an equivariant knot in~$S^3$ and~$D$ be any smooth~$H$-slice disk for~$K$ in~$\cpring$ i.e. $D$ represents the zero class in $H_2(\cpring,\partial \cpring)$.
 Then:
\begin{enumerate}
\item If~$\Vtu(K) < 0$, then~$D$ and~$\tau_+(D)$ are not smoothly isotopic rel.\ boundary.
\item If~$\Vitu(K) < 0$, then~$D$ and~$\tau_-(D)$ are not smoothly isotopic rel.\ boundary.
\end{enumerate}
Recalling that $\Z$-disks are examples of~$H$-slice disks,
the relevance of this to the present situation is the following. If~$D_\gamma$ is a generalized crossing change disk associated to an oriented, generalized unknotting curve~$\gamma$, then the construction of~$\tau_{\pm}$ shows that:

\begin{enumerate}
\item $\tau_+(D_\gamma) = D_{\tau(\gamma)}$ and
\item $\tau_-(D_\gamma) = D_{- \tau(\gamma)}$.
\end{enumerate}
Here, we give~$\tau(\gamma)$ the pushforward orientation of~$\gamma$ under~$\tau$. Indeed, clearly both~$\tau_+(D_\gamma)$ and~$\tau_-(D_\gamma)$ are generalized crossing change disks for~$K$ associated to the unknotting curve~$\tau(\gamma)$. The only subtlety is that in the latter case, pushing forward the isotopy from~$\gamma$ to~$c$ under~$\tau_{-}$ gives an isotopy from~$\tau(\gamma)$ to~$-c$. Hence in this situation, we instead have~$\tau_-(D_\gamma) = D_{- \tau(\gamma)}$.

We are thus led to the following concrete question:

\begin{question}\label{question:condition}
Does there exist an equivariant knot $K$ such that either:

\begin{enumerate}
\item $\Vtu(K) < 0$ and $K$ admits a generalized positive crossing change curve $\gamma$ to a smoothly $\Z$-slice knot in $D^4$ such that 
\[
[\tilde{\gamma}] = [\widetilde{\tau(\gamma)}] \in H_1(E^\infty_K).
\]
\item $\Vitu(K) < 0$ and $K$ admits a generalized positive crossing change curve $\gamma$ to smoothly $\Z$-slice knot in $D^4$ such that
\[
[\tilde{\gamma}] = [- \widetilde{\tau(\gamma)}] \in H_1(E^\infty_K).
\]
\end{enumerate}
\end{question}
\noindent
Either alternative produces a pair of exotic slice disks for $K$ in $\cpring.$

\begin{remark}
Unfortunately, the search for such a $\gamma$ is not so straightforward:
if $\gamma$ is an actual (oriented) positive unknotting curve and $\tau(\gamma) = \gamma$, then $\Vtu(K) \geq \Vtu(U) = 0$, while if $\tau(\gamma) = - \gamma$, then $\Vitu(K) \geq \Vitu(U) = 0$. 
These inequalities follow from~\cite[Section~7.2]{DaiHeddenMallick}.
It may thus be fruitful to search for examples among equivariant knots with equivariant crossing change number greater than one. 
\end{remark}


\section{Units in rings of quadratic integers}
\label{sec:NumberTheory}

We have reduced the enumeration of~$\Z$-disks in~$(\C P^2)^\circ$ with boundary a knot~$K$ to understanding (a quotient of) the group~$U(\Z[t^{\pm 1}]/(\Delta_K))$.
Here, recall that 
$$U(\Z[t^{\pm 1}]/(\Delta_K))=\lbrace x \in  \Z[t^{\pm 1}]/(\Delta_K) \mid x\overline{x}=1 \rbrace.$$
Studying this group involves some number theory. In Section~\ref{sec:UnitaryUnits}, we re-interpret $U(\Z[t^{\pm 1}]/(\Delta_K))$ in terms of the localization of a ring of quadratic integers $\Z[\omega]$. This latter object will thus be the focus of the present section. Section~\ref{sub:RingsOfQuadraticIntegers} collects some definitions and facts concerning $\Z[\omega]$, whereas Section~\ref{sub:UnitsZ[omega,1/n]} describes the group of units of~$\Z[1/n, \omega]$ for any nonzero $n \in \Z$. 

\subsection{Rings of quadratic integers}
\label{sub:RingsOfQuadraticIntegers}

We recall some facts about rings of quadratic integers.
Further discussion of these can be found in many texts on algebraic number theory such as~\cite{Neukirch}. 
We note that the following results admit generalizations to rings of (not necessarily quadratic)  integers, but we will not use these here.

\begin{definition}
Given a square free integer $d$, let
\begin{equation}
\label{eq:Omega}
\omega = 
\begin{cases}
\sqrt{d} & \text{ if } d \equiv 2, 3 \bmod 4 \\
\dfrac{1 + \sqrt{d}}{2} & \text{ if } d \equiv 1 \bmod 4,
\end{cases}
\end{equation}
We refer to $\Z[\omega]$ as a \emph{ring of quadratic integers}.
\end{definition}

The reason for this terminology is as follows: the field $\Q(\sqrt{d})$ is typically called a \emph{quadratic number field}, and $\Z[\omega]$ is the integral closure of $\Z$ in $\Q(\sqrt{d})$, see e.g.~\cite[Section 3.1]{Stewart}.
The next result collects some properties of this ring and its ideals. For a proof of the first assertion we refer to~\cite[Section 5.2]{Stewart}; for the second we refer to~\cite[Section 9.3]{Stewart}.

\begin{theorem}[Finiteness of the Class Group]
\label{thm:DedekindClass}
Given a square free integer $d$ and $\omega$ as in~\eqref{eq:Omega},
\begin{enumerate}
\item The ring of integers~$\Z[\omega]$ is a Dedekind domain: every ideal in $\Z[\omega]$ admits a unique prime factorization.
\item The ring of integers~$\Z[\omega]$ has finite class number: there exists some $h \in \N$ such that $\mathfrak{p}^h$ is principal for any ideal $\mathfrak{p}$ in $\Z[\omega]$. 
\end{enumerate}
\end{theorem} 

The following result describes the units of $\Z[\omega]$ and is known as \emph{Dirichlet's Unit Theorem}; see e.g.~\cite[Appendix B]{Stewart} for a proof.

\begin{theorem}[Dirichlet's Unit Theorem]
\label{thm:DirichletUnit}
Given a square free integer $d$ and $\omega$ as in~\eqref{eq:Omega},
\begin{enumerate}
\item If $d > 0$, then $\Z[\omega]^\times \cong \{\pm 1\} \times \Z$. 
In particular, there exists a unit $u \in \Z[\omega]^\times$, called a \emph{fundamental unit}, such that every unit in $\Z[\omega]$ is of the form $\pm u^i$ for some $i \in \Z$.
\item If $d < 0$, then $\Z[\omega]^\times$ is finite. More precisely:
\begin{enumerate}
\item If $d = -1$, then there are four units, given by $\pm 1$ and $\pm i$. 
\item If $d = -3$, then there are six units, given by $\pm 1$ and $(\pm 1 \pm \sqrt{-3})/2$. 
\item If $d < -3$, then the only units are $\pm 1$.
\end{enumerate}
\end{enumerate}
\end{theorem}

In this paper, we will be most concerned with the case where $d \equiv 1$ mod $4$.
In this setting, the next result describes the factorization of ideals $(p)$ in $\Z[\omega]$ for $p$ a prime in $\Z$. See e.g~the discussion surrounding \cite[Proposition 8.5]{Neukirch} for a proof.

\begin{theorem}[Splitting Criterion]
\label{thm:RamifyingSplittingInert}
Assume~$d \equiv 1 \bmod 4$ is square free, and let $\omega$ be as in~\eqref{eq:Omega}. Let~$p$ be an odd prime in~$\Z$. Then~$(p)$ factors into prime ideals in~$\Z[\omega]$ as follows:
\begin{enumerate}
\item~$(p)$ prime factors into the form~$\mathfrak{p}^2$ if and only if~$p$ divides~$d$.
\item~$(p)$ prime factors into the form~$\mathfrak{p}_1 \mathfrak{p_2}$ if and only if~$p$ does not divide~$d$ and~$d$ is a quadratic residue mod~$p$.
\item~$(p)$ remains a prime ideal in~$\Z[\omega]$ if and only if~$p$ does not divide~$d$ and~$d$ is not a quadratic residue mod~$p$.
\end{enumerate}
In the special case~$p = 2$, the quadratic residue conditions for possibilities~$(2)$ and~$(3)$ are replaced by the conditions~$d \equiv 1 \bmod 8$ and~$d \equiv 5 \bmod 8$, respectively.
\end{theorem}

The above possibilities are referred to as~$p$ \emph{ramifying, splitting,} and \emph{remaining inert}.


\subsection{The units of $\Z[1/n, \omega]$}
\label{sub:UnitsZ[omega,1/n]}


While Dirichlet's unit theorem describes the group of units of $\Z[\omega]$, in Section~\ref{sec:UnitaryUnits} we will be interested in the localization of $\Z[\omega]$ at a particular integer $n$. This section is thus devoted to describing the units of $\Z[1/n, \omega]$ for any nonzero $n \in \Z$.
The answer will follow promptly from Dirichlet's Unit theorem once we recall the notion of the saturation of a multiplicatively closed subset.
%
%
%

\begin{definition}
\label{def:Saturation}
The \textit{saturation} of a multiplicatively closed subset $S \subset R$ of an integral domain is defined by
\[
\widetilde{S} = \{ x \in R \text{ such that } x \text{ divides } s \text{ for some } s \in S \}.
\]
\end{definition}



The next lemma is essentially the observation that if we invert~$S$, then every element of~$S$ gives us a new unit in~$S^{-1}R$ and, more generally, every divisor of every element of~$S$ becomes a unit in~$S^{-1}R$. 

\begin{lemma}
\label{lem:Saturation}
Given a multiplicatively closed subset $S \subset R$ of an integral domain, the group of units of~$S^{-1}R$ is multiplicatively generated by the saturation of~$S$.
\end{lemma}

\begin{proof}
We first claim that the elements of~$\smash{\widetilde{S}}$ are units in~$S^{-1}R$. Indeed, let~$x_1 \in \widetilde{S}$, so that~$x_1 x_2 = s$ for some~$s \in S$. Then~$x_1$ and~$x_2$ are units in~$S^{-1} R$, with inverses~$x_2/s$ and~$x_1/s$, respectively. Conversely, suppose that~$r_1/s_1$ is a unit in~$S^{-1}R$. Then we have that~$r_1/s_1 \cdot r_2/s_2 = 1$ for some~$r_2/s_2 \in S^{-1}R$. But then~$r_1 r_2 = s_1 s_2$, whereupon~$r_1$ and~$r_2$ are elements of~$\smash{\widetilde{S}}$. 
\end{proof}

The ring~$\Z[1/n, \omega]$ corresponds to $S^{-1}R$ in the case $R=\Z[\omega]$ and $S=\{n^k\}$. Lemma~\ref{lem:Saturation} thus gives a set of generators for~$\Z[1/n, \omega]^\times$, consisting of the set of units of~$\Z[\omega]$ together with the saturation (in~$\Z[\omega]$) of the set~$\{n^k\}$. Since the former is given by Dirichlet's unit theorem, the principal difficulty is to understand the latter. If $\Z[\omega]$ is a UFD, then the saturation of $\{n^k\}$ may be computed by factoring $n$ in $\Z[\omega]$ and taking all possible products of these prime factors. In general, however, $\Z[\omega]$ is not a UFD and so the saturation of $\{n^k\}$ cannot be computed purely from the factorization of $n$. 

Instead, we have the following criterion based on the factorization of the \textit{ideal} $(n)$:

\begin{lemma}
\label{lem:SaturationCriterion}
Let $R$ be a Dedekind domain and fix $n \in R$. Then $x \in R$ is in the saturation of~$\{n^k\}$ if and only if $(x)$ is a product of prime ideals occurring in the factorization of~$(n)$.
\end{lemma}
\begin{proof}
Suppose that~$x$ is in the saturation of~$\{n^k\}$. Then~$x$ divides some element~$n^k$; this implies that~$(x)$ divides~$(n^k) = (n)^k$ as an ideal in~$R$.
 By unique prime factorization of ideals, this shows that~$(x)$ must be a product of prime ideals occurring in the factorization of~$(n)$. Conversely, suppose that~$(x)$ is a product of prime ideals occurring in the factorization of~$(n)$. Clearly, for~$k$ sufficiently large, the factorization of~$(n)^k$ then contains the factorization of~$(x)$. Hence~$(x)$ divides~$(n^k)$ as an ideal, which implies that~$x$ divides~$n^k$ as an element of~$R$.
\end{proof}

\begin{remark}
\label{rem:Recipe}
Lemma~\ref{lem:SaturationCriterion} gives a recipe for understanding the saturation of~$\{n^k\}$ in $\Z[\omega]$, as follows. 
\begin{enumerate}
\item Factor~$(n)$ into prime ideals in~$\Z[\omega]$. 
\item Consider all possible products of prime ideals occurring in this factorization and determine which of these products are principal. 
\item For each such principal ideal, choose a generating element. 
Up to multiplication by units in~$\Z[\omega]$, this procedure gives all elements in the saturation of~$\{n^k\}$.
\end{enumerate}
\end{remark}

In general, the algorithm of Remark~\ref{rem:Recipe} is difficult to carry out explicitly, as the second step involves computing products of prime ideals and deciding which are principal. Note, however, that due to Theorem~\ref{thm:DedekindClass} we have a particular set of such products which are always principal. Let
$$(n) = \mathfrak{p}_1^{\alpha_1} \mathfrak{p}_2^{\alpha_2}  \cdots \mathfrak{p}_r^{\alpha_r}$$
be the factorization of~$(n)$ into prime ideals in~$\Z[\omega]$. If the class number of~$\Z[\omega]$ is $h$, then by Theorem~\ref{thm:DedekindClass} each $\mathfrak{p}_i^h$ is principal. Write
\[
\mathfrak{p}_i^h = (x_i).
\]
By Lemma~\ref{lem:SaturationCriterion}, $x_i$ is in the saturation of $\{n^k\}$ and is hence a unit in $\Z[1/n, \omega]$. 
We thus obtain a subgroup of $\Z[1/n, \omega]^\times$ generated by the set of $x_i$ together with $\Z[\omega]^\times$. 
The next proposition computes the rank of this subgroup and uses it to calculate the rank of $\Z[1/n, \omega]^\times$.

\begin{proposition}
\label{prop:NotUFD}
Let $d$ be square free and $\omega$ as in~\eqref{eq:Omega}. Let $n$ be any nonzero integer and
$$(n) = \mathfrak{p}_1^{\alpha_1} \mathfrak{p}_2^{\alpha_2}  \cdots \mathfrak{p}_r^{\alpha_r}$$
be the factorization of~$(n)$ into prime ideals in~$\Z[\omega]$. Denote the class number of~$\Z[\omega]$ by $h$, so that~$\mathfrak{p}_i^h = (x_i)$ is principal for each~$i$.
Then the following hold: 
\begin{enumerate}
\item The subgroup of $\Z[1/n, \omega]^\times$ generated by the $x_i$ and $\Z[\omega]^\times$ has presentation
\begin{equation}\label{eq:linearindependence}
\Z[\omega]^\times \times \langle x_1, x_2, \ldots, x_r \rangle.
\end{equation}
\item $\rk(\Z[1/n, \omega]^\times) = \operatorname{rk}(\Z[\omega]^\times \times \langle x_1, x_2, \ldots, x_r \rangle)$. 
\end{enumerate}
In particular, 
\[
\operatorname{rk}(\Z[1/n, \omega]^\times) = \operatorname{rk}(\Z[\omega]^\times) + r 
= 
\begin{cases} 
1 + r & \text{ if } d > 0 \\
r & \text{ if } d < 0.
\end{cases}
\]
\end{proposition}
\begin{proof}
To establish the first claim, we check that the~$x_i$ are multiplicatively linearly independent and that there are no nontrivial relations between the $x_i$ and $\Z[\omega]^\times$. Suppose we had a linear relation amongst the~$x_i$ and $u \in \Z[\omega]^\times$, say
\[
u x_1^{a_1} \cdots x_q^{a_q} = x_{q+1}^{a_{q+1}} \cdots x_r^{a_r}
\]
with all~$a_i \geq 0$. Then 
\[
\mathfrak{p}_1^{h a_1} \cdots \mathfrak{p}_q^{h a_q} = \mathfrak{p}_{q+1}^{h a_{q+1}} \cdots \mathfrak{p}_r^{h a_r},
\]
contradicting unique prime factorization of ideals. This shows that the $x_i$ are linearly independent, even modulo multiplication by any unit in $\Z[\omega]$.

We now prove the second claim by relying on the following fact: if $H \leq G$ is a subgroup of an abelian group $G$ and there is an integer $n$ such that $x^n \in H$ for all $x \in  G$, then $H$ and $G$ have the same rank.
We will therefore show that
that if~$x \in \Z[1/n, \omega]^\times$, then~$x^h$ lies in \eqref{eq:linearindependence}, where $h$ is the class number of~$\Z[\omega]$. 
Without loss of generality, we may assume that~$x$ is an element of the saturation of~$\{n^k\}$ in~$\Z[\omega]$, as these generate $\Z[1/n, \omega]^\times$. Then~$x$ divides some~$n^k$, which implies that~$(x)$ divides~$(n^k)$ as an ideal in~$\Z[\omega]$. But by prime factorization of ideals, this means that~$(x)$ is a product of powers of~$\mathfrak{p}_1, \ldots, \mathfrak{p}_r$. 
Raising both sides to the~$h$-th power, we see that~$(x)^h=(x^h)$ is a product of powers of the~$(x_i)$. 
Hence (up to multiplication by a unit in~$\Z[\omega]$) we have that~$x^h$ is a product of powers of the~$x_i$, as desired.

The final claim of the proposition follows from the observation that
\[
\operatorname{rk}(\Z[\omega]^\times \times \langle x_1, x_2, \ldots, x_r \rangle) = \operatorname{rk}(\Z[\omega]^\times) + r = 
\begin{cases} 
1 + r & \text{ if } d > 0 \\
r & \text{ if } d < 0.
\end{cases}
\]
Here, we have used the fact that the rank of $\Z[\omega]^\times$ is zero or one depending on whether $d < 0$ or $d > 0$; recall Theorem~\ref{thm:DirichletUnit}.
\end{proof}

\section{Enumerating unitary units}
\label{sec:UnitaryUnits}

We are now ready to study the group $U(\Z[t^{\pm 1}]/(\Delta_K))$ when the Alexander polynomial~$\Delta_K$ is quadratic.
We structure our casework according to the possible Alexander polynomials which arise.
In more detail, we set
$$ \Delta_n = nt^2 - (2n + 1) t + n~$$
since, as~$n$ varies over~$\Z$, this encompasses (up to multiplication by~$\pm t^k, k \in \Z$) all possible quadratic Alexander polynomials.
For each of the above Alexander polynomials, the number of~$\Z$-disks for any corresponding knot is deduced by considering the ring 
\[
\Lambda_n = \Z[t^{\pm 1}]/(\Delta_n).
\]
More precisely,  we saw in Theorem~\ref{thm:DiscsWithoutCPP22Intro} that the number of~$\Z$-disks in~$(\C P^2)^\circ$ for a knot with polynomial~$\Delta_n$ is in bijection with $U_n/\{t^k\}_{k \in \Z}$, where $U_n$ is the group of \textit{unitary units} of~$\Lambda_n$:
$$ U(\Lambda_n)=\lbrace p \in \Lambda_n \mid p\overline{p}=1 \rbrace.$$
Here, recall that $\overline{p}$ is the conjugate of $p$ formed by precomposing with the involution $t \mapsto t^{-1}$.

We will primarily be interested in determining whether~$U_n/\{t^k\}_{k \in \Z}$ has zero or positive rank. 
In the case that the group is finite, we will make an explicit count of the number of unitary units; in the case that the group is infinite, we compute the rank.
The answer will involve
$$ \Omega(n):=\# \lbrace \text{positive primes dividing } n \rbrace.~$$
The next theorem is our main number-theoretic result and is Theorem~\ref{thm:NumberTheoryIntro} from the introduction.
\begin{customthm}{\ref{thm:NumberTheoryIntro}}
\label{thm:MainNumberTheory}
For every~$n \in \Z$, the group~$U(\Lambda_n)/\{t^k\}_{k \in \Z}$ can be described as follows. 
\begin{enumerate}
\item The group~$U(\Lambda_n)/\{t^k\}_{k \in \Z}$ is finite precisely when~$n = 2, 1, 0, -1$, or~$-p^k$ for a prime~$p$ and
$$
U(\Lambda_n)/\{t^k\}_{k \in \Z} 
\cong \begin{cases}
\lbrace 1 \rbrace & \quad \text{ if $n=-1, 0$},  \\
\Z/2\Z & \quad \text{ if $n=1,2$ or $n=-p^k$ with $k$ odd},  \\
\Z/4\Z & \quad \text{ if $n=-p^k$ with $k$ even}. \\
\end{cases}
$$
\item In all other cases, ~$U(\Lambda_n)/\{t^k\}_{k \in \Z}$ is infinite and
$$
\rk \left( U(\Lambda_n)/\{t^k\}_{k \in \Z} \right)
=\begin{cases}
\Omega(n) & \quad \text{ if $\Delta_n$ is irreducible and $n>0$},  \\
\Omega(n)-1 & \quad \text{ if $\Delta_n$ is irreducible and $n<0$ or if $\Delta_n$ is reducible.}
\end{cases}
$$
\end{enumerate}
\end{customthm}
\begin{proof}
Proposition~\ref{thm:IrreducibleRank} is concerned the cases where $\Delta_n$ is irreducible and the count is infinite. 
Propositions~\ref{prop:n=-101} and~\ref{prop:UnitCountSquareFree} carry out the count when $\Delta_n$ is irreducible and the group is finite $(n=-1,0,1,-p^k$ with $p$ prime).
Proposition~\ref{prop:ReducibleRank} considers the case where $\Delta_n$ is reducible leaving only the case where~$\Delta_n$ is reducible and the group is finite ($n=2$) which is studied in Proposition~\ref{prop:n=2}.
\end{proof}

\begin{remark}
We collect some remarks on this theorem.
\label{rem:Reducibility}
\begin{itemize}
\item The polynomial $\Delta_n$ is reducible if and only if the discriminant~$4n + 1$ of~$\Delta_n$ is a perfect square. 
We also note that all~$\Delta_n$ with~$n < 0$ are irreducible and that for~$n > 0$, the polynomial~$\Delta_n$ is reducible precisely when~$n = m(m+1)$ for some~$m$. 
\item  The cases in which~$(2)$ give a rank less than or equal to zero correspond precisely to the cases in~$(1)$. 
We use the convention that strictly negative rank should be interpreted as zero rank; we write this for simplicity of the theorem statement.
\end{itemize}
\end{remark}

Before embarking on the proof of Theorem~\ref{thm:MainNumberTheory}, we give some basic examples of distinct elements in $U(\Lambda_n)/\{t^k\}_{k \in \Z}$ for various values of $n$, both for concreteness and because of their relevance to topological applications.
We first show that $U(\Lambda_n)/\{t^k\}_{k \in \Z}$ is at least of cardinality two except in the cases $n = 0$ and $n = -1$.

\begin{proposition}\label{prop:pmdistinct}
The classes of $+1$ and $-1$ in $U(\Lambda_n)/\{t^k\}_{k \in \Z}$ are distinct except in the cases $n = 0$ and $n = -1$.
\end{proposition}
\begin{proof}
Suppose that the classes of $+1$ and $-1$ in $U(\Lambda_n)/\{t^k\}_{k \in \Z}$ coincide. Multiplying through by powers of $t$ as necessary, we then have the congruence of polynomials
\[
t^a \cdot (+1) \equiv t^b \cdot (- 1) \bmod \Delta_n.
\]
Clearly, this forces the leading coefficient of $\Delta_n$ to be $\pm 1$. Moreover, the roots of $\Delta_n$ must lie on the unit circle since $\Delta_n$ divides $t^a + t^b$; this rules out the case $n = 1$. On the other hand, when $n = -1$ it is indeed the case that
\[
t^3 \equiv -1 \bmod \Delta_{-1},
\]
since $\Delta_{-1} = - t^2 + t - 1$ divides $t^3 + 1$. In the case $n = 0$, the ring $\Lambda_n = \Z[t^{\pm 1}]/(1)$ is trivial and the claim is evident.
\end{proof}

\begin{remark}\label{rem:pmextension}
Using Proposition~\ref{prop:pmdistinct} and Theorem~\ref{thm:MainNumberTheory}, it is straightforward to determine the group 
\[
U(\Lambda_n)/\{\pm t^k\}_{k \in \Z}
\]
discussed in Remark~\ref{rem:NumberTheoryEquivalence}. Explicitly, consider the subgroup $\{\pm 1\}$ of $U(\Lambda_n)/\{t^k\}_{k \in \Z}$. If $n = -1$, then $+1$ and $-1$ actually constitute the same class and hence this subgroup is trivial; in this case $U(\Lambda_n)/\{\pm t^k\}_{k \in \Z} = U(\Lambda_n)/\{t^k\}_{k \in \Z}$. Otherwise, we have that $U(\Lambda_n)/\{\pm t^k\}_{k \in \Z}$ is a quotient of $U(\Lambda_n)/\{t^k\}_{k \in \Z}$ by a subgroup of order two. The groups enumerated in Theorem~\ref{thm:MainNumberTheory} then necessarily give those listed in Remark~\ref{rem:NumberTheoryEquivalence}. In the case of positive rank, it is clear that the rank remains unchanged by this further quotient. 
\end{remark}

\begin{remark}\label{rem:generalization}
It is not difficult to generalize Proposition~\ref{prop:pmdistinct} to other Alexander polynomials. Indeed, let $\Delta$ be any (not-necessarily-quadratic) Alexander polynomial and $\Lambda := \Z[t^{\pm 1}]/(\Delta)$. Then the classes of $+1$ and $-1$ in $U(\Lambda)/\{t^k\}_{k \in \Z}$ coincide if and only if $\Delta$ (up to multiplication by~$t^k$) divides $t^\ell + 1$ for some $\ell$.
\end{remark}

The following presents a slightly less trivial set of distinct elements in the case that $n = -k^2$ for some $k > 1$.

\begin{proposition}\label{prop:fourdistinct}
Let $n = - k^2$ for $k > 1$. Then the classes of
\[
\{\pm 1, \pm k(t - 1)\}
\]
are distinct elements of $U(\Lambda_n)/\{t^k\}_{k \in \Z}$. 
\end{proposition}
\begin{proof}
We have already seen in Proposition~\ref{prop:pmdistinct} that (the classes of) $\pm 1$ represent distinct unitary units in $U(\Lambda_n)/\{t^k\}_{k \in \Z}$. It is clear that (the class of) $k(t - 1)$ is a unitary unit, since
\[
k(t-1) \cdot k(t^{-1} - 1) = k^2(-t +2 - t^{-1}) = nt - 2n + nt^{-1}.
\]
Up to multiplication by powers of $t$, this is congruent to $1$ modulo $\Delta_n = nt^2 - (2n + 1)t + n$. We furthermore claim that $k(t - 1)$ is not in the same class as $\pm 1$. Indeed, suppose this were the case. Multiplying through by powers of $t$ as necessary, we then have the congruence of polynomials
\[
t^a \cdot k(t-1) \equiv t^b \cdot (\pm 1) \bmod \Delta_n.
\]
Since the leading coefficient of $\Delta_n$ is greater in absolute value than the leading coefficient of $t^a \cdot k(t - 1) - t^b \cdot (\pm 1)$, it is clear that this is impossible. The case of $- k(t - 1)$ is analogous. Note that $k(t - 1)$ and $- k(t-1)$ are not in the same class, since this would imply $-1$ lies in the trivial class.
\end{proof}

\begin{remark}
According to Theorem~\ref{thm:MainNumberTheory}, the four elements presented in Proposition~\ref{prop:fourdistinct} constitute the entirety of $U(\Lambda_n)/\{t^k\}_{k \in \Z}$ if $n$ is $- p^k$ for $k$ even. In the general case of Proposition~\ref{prop:fourdistinct}, these four elements are still distinct, but do not necessarily constitute all of $U(\Lambda_n)/\{t^k\}_{k \in \Z}$.
\end{remark}
\color{black}

\subsection{Strategy of the proof in the irreducible case}

The proof when $\Delta_n$ irreducible is significantly longer than in the reducible case.
We therefore begin with a brief outline of this former case.
Fix $n \in \Z$. 
Throughout this section, let 
\[
a = \dfrac{(2n+1) + \sqrt{4n+1}}{2} \quad \text{and} \quad \xi = \dfrac{1 + \sqrt{4n+1}}{2}.
\]
We assume throughout the section that $\Delta_n$ is irreducible. 
The first step of the proof will be to construct a ring isomorphism
\[
f \colon \Lambda_n = \Z[t^{\pm 1}]/(\Delta_n) \rightarrow \Z[1/n, \xi].
\]

Both the domain and the codomain of $f$ are equipped with conjugation involutions giving rise to norm maps:

\begin{definition}\label{def:norms}
The conjugation map $\overline{\cdot}$ on $\Z[t^{\pm 1}]$ is defined by sending $p(t)$ to $p(t^{-1})$; this descends to a conjugation map on $\Z[t^{\pm 1}]/(\Delta_n)$. We thus obtain a group homomorphism
\begin{align*}
N_{\Lambda} \colon \Lambda_n^\times &\to \Lambda_n^\times \\
p &\mapsto p\overline{p}.
\end{align*}
The group of unitary units in $\Lambda_n$ is defined to be $U(\Lambda_n)=\ker(N_{\Lambda})$. We likewise have a conjugation map $\overline{\cdot}$ on $\Z[1/n, \xi]$ which is trivial on the subring $\Z[1/n]$ and satisfies
$$\overline{\xi}=\frac{1-\sqrt{4n+1}}{2}.$$
This defines a group homomorphism:
\begin{align*}
N \colon  \Z[1/n, \xi]^\times&\to  \Z[1/n]^\times  \\
x &\mapsto x\overline{x}.
\end{align*}
The group of unitary units in $\Z[1/n, \xi]$ is defined to be $U(\Z[1/n, \xi]) = \ker(N)$. 
\end{definition}

The following proposition constructs $f$ and shows that it intertwines the conjugation involution on $\Lambda_n$ with the conjugation involution on $\Z[1/n, \xi]$. This will show that $f$ gives an isomorphism between $U(\Lambda_n)$ and $U(\Z[1/n, \xi])$. 

\begin{proposition}
\label{prop:RingIrred}
Assume $\Delta_n$ is irreducible. Extending the assignment $t \mapsto a/n$ to a ring homomorphism defines a surjection
$$f \colon \Z[t^{\pm 1}] \twoheadrightarrow \Z[1/n, \xi]$$
with kernel $(\Delta_n)$.
In particular, this assignment defines an isomorphism 
$$f \colon  \Lambda_n  \xrightarrow{\cong} \Z[1/n, \xi].$$
Moreover, $f$ intertwines conjugation on $\Lambda_n$ with conjugation on $\Z[1/n, \xi]$. 
\end{proposition}
\begin{proof}
First note that since~$a$ and~$\xi$ differ by an integer, it is clear that~$\Z[1/n, a] = \Z[1/n, \xi]$ and therefore $f$ does indeed take values in $\Z[1/n, \xi]$.
Next, observe that since~$a/n \cdot \bar{a}/n = 1$, extending~$f$ as a ring homomorphism implies that~$f(t^{-1})=\overline{a}/n$; here note that~$\bar{a}/n = 2 + 1/n - a/n$ is indeed an element of~$\Z[1/n, a]$. 
This also shows that $f$ intertwines conjugation on $\Z[t^{\pm 1}]$ with conjugation on $\Z[1/n, \xi]$.
The fact that $\ker(f)=(\Delta_n)$ follows from the claim that~$\Delta_n$ is the minimal integral polynomial relation satisfied by~$a/n$; this holds since~$\Delta_n$ is irreducible. To see that~$f$ is surjective, note that~$f(nt) = a$ and~$f(t + t^{-1} - 2) = 1/n$, so~$a$ and~$1/n$ are both in~$\im(f)$. 
\end{proof}

The next result is a direct consequence of Proposition~\ref{prop:RingIrred} and reduces the calculation of the rank~$\rk(U(\Lambda_n))$ to the calculation of $\operatorname{rk}(\Z[1/n, \xi]^\times)$ and~$\rk(\im(N))$.

\begin{proposition}
\label{prop:RankNullity}
The isomorphism $f$ intertwines the norms on $\Lambda_n$ and $\Z[1/n, \xi]$ and therefore fits into the following commutative diagram of groups: 
$$
\xymatrix@C+2pc{
0 \ar[r]& U(\Lambda_n) \ar[r]\ar[d]^\cong_{f|_{U(\Lambda_n)}}& \Lambda_n^\times \ar[r]^{N_{\Lambda}}\ar[d]^\cong_f&  \Lambda_n^\times \ar[d]_f \\
0 \ar[r]& U(\Z[1/n, \xi]) \ar[r]& \Z[1/n, \xi]^\times \ar[r]^{N} & \Z[1/n]^\times.
}
$$
In particular, the following equality holds:

$$ \rk (U(\Lambda_n))=\rk(\Z[1/n,\xi]^\times)-\rk(\im(N)).$$
\end{proposition}
\begin{proof}
The fact that $f$ intertwines the norms follows because $f$ is a homomorphism of rings with involutions:
$$N(f(p))=f(p)\overline{f(p)}=f(p\overline{p})=f(N_{\Lambda}(p)).$$
The remaining assertions follow from Proposition~\ref{prop:RingIrred}.
\end{proof}

In the next sections, we calculate the rank of the groups involved in Proposition~\ref{prop:RankNullity}.
Once these ranks have been calculated, leading to a calculation of $\rk(U(\Lambda_n))$,  it is then not difficult to deduce the rank of $U(\Lambda_n)/\{t^k\}_{k \in \Z}$ by studying the image of $\{ t^k \}_{k \in \Z}$ under $f$.

\subsection{Rank calculations when $\Delta_n$ is irreducible}

Following the discussion of the previous subsection, we now calculate $\operatorname{rk}(\Z[1/n, \xi]^\times)$ and~$\rk(\im(N))$.
We begin with the former. Proposition~\ref{prop:NotUFD} shows how to calculate $\operatorname{rk}(\Z[1/n, \omega]^\times)$ when $\Z[\omega]$ is a ring of quadratic integers. However, Proposition~\ref{prop:NotUFD} is not directly applicable to our situation: since $4n + 1$ is not necessarily square free, $\xi$ may not be of the form described in \eqref{eq:Omega}. To rectify this, let
\[
4n + 1 = c^2 d
\]
for $c > 0$ and $d$ a square free integer and define
\[
\omega = \dfrac{1 + \sqrt{d}}{2}.
\]
Note that $\omega$ is of the form described in \eqref{eq:Omega}, whereas
\[
\xi = \dfrac{1 + \sqrt{4n + 1}}{2} = \dfrac{1 + c\sqrt{d}}{2}.
\]
It is readily checked that $\Z[\xi]$ is a subring of $\Z[\omega]$, and hence that $\Z[1/n, \xi]^\times$ is a subgroup of~$\Z[1/n, \omega]^\times$. Our strategy will thus be to prove that the ranks of~$\Z[1/n, \xi]^\times$ and $\Z[1/n, \omega]^\times$ coincide, and then compute the latter using Proposition~\ref{prop:NotUFD}.

Recall that $\Omega(n)$ denotes the number of (positive) distinct primes dividing~$n$ in~$\Z$. 

\begin{proposition}
\label{prop:RankUnits}
If $\Delta_n$ is irreducible, then
\[
\operatorname{rk}(\Z[1/n, \xi]^\times) = \operatorname{rk}(\Z[1/n, \omega]^\times) = 
\begin{cases} 
1 + 2\Omega(n)& \text{ if } n > 0 \\
2\Omega(n) & \text{ if } n < 0.
\end{cases}
\]
\end{proposition}
\begin{proof}
We establish the claimed equalities in succession. 
For the first equality, it suffices to prove there exists a natural number~$g$ such that if~$x_1 \in \Z[1/n, \omega]^\times$, then $x_1^g \in \Z[1/n, \xi]^\times$.
 Without loss of generality we may assume that $x_1$ is an element of the saturation of~$\{n^k\}$ in $\Z[\omega]$, since these generate $\Z[1/n, \omega]^\times$; recall Lemma~\ref{lem:Saturation}.
Then $x_1 x_2 = n^k$ for some~$x_2 \in \Z[\omega]$ and $k \in \N$. 

Consider the quotient ring~$\Z[\omega]/(c)$. 
Since~$\Z[\omega]$ is a free~$\Z$-module of rank two (generated by~$1$ and~$\omega$), it is clear that~$\Z[\omega]/(c) \cong (\Z/c\Z) \oplus (\Z/c\Z)$. 
In particular, the group of units in~$\Z[\omega]/(c)$ is finite, being a subset of a finite ring.  We claim that $g$ may be taken to be the order of this group.
Rearranging the equation~$4n + 1 = c^2d$ and exponentiating gives~$(4n)^k = (c^2d - 1)^k$. Expanding this latter equality and using the fact that $x_1 x_2 = n^k$ gives a linear combination of $x_1$ and $c$ which is equal to $1$. 
This means that the class of~$x_1$ is a unit in~$\Z[\omega]/(c)$. 
Thus~$x_1^g = 1$ in $\Z[\omega]/(c)$; that is,~$x_1^g \equiv 1 \bmod c \Z[\omega]$. A similar statement holds for~$x_2^g$. 

We have thus shown that~$x_1^g \in 1 + c \Z[\omega]$. 
However, by direct inspection, we clearly have~$1 + c\Z[\omega] \subseteq \Z + c\Z[\omega] \subseteq \Z[\xi]$, so~$x_1^g \in \Z[\xi]$. 
Since the same holds for~$x_2^g$, we conclude that~$x_1^g$ is an element of the saturation of~$\{n^k\}$ in~$\Z[\xi]$, and hence a unit in~$\Z[1/n, \xi]$ by Lemma~\ref{lem:Saturation}. This completes the proof of the first equality.

We now apply Proposition~\ref{prop:NotUFD} to establish the second equality. Let $\smash{n=p_1^{\alpha_1} \cdots p_{\Omega(n)}^{\alpha_{\Omega(n)}}}$ be the prime factorization of $n$ in $\Z$.
Note that since $4n + 1 = c^2 d$, we have that $c$ is odd and hence $c^2 \equiv 1 \bmod 4$. 
It follows that $d \equiv 1 \bmod 4$ and we may apply the splitting criterion (Theorem~\ref{thm:RamifyingSplittingInert}) to determine the factorization of each $(p_i) \subset \Z[\omega]$.
Clearly, none of the $p_i$ divide $4n+1$. The equality $4n + 1 = c^2 d$ shows that~$d \equiv 1/c^2 \bmod p_i$, so $d$ is a quadratic residue mod $p_i$. 
Hence by Theorem~\ref{thm:RamifyingSplittingInert}, each $(p_i) \subset \Z[\omega]$ decomposes as a product of two distinct prime ideals in $\Z[\omega]$. (In the special case $p_i = 2$, we have $d \equiv 1 \bmod 8$ since necessarily $c^2 \equiv 1 \bmod 8$.) 
Proposition~\ref{prop:NotUFD} applied with~$r=2 \Omega(n)$ then shows
\[
\operatorname{rk}(\Z[1/n, \omega]^\times) = 
\begin{cases} 
1 + 2\Omega(n)& \text{ if } n > 0 \\
2\Omega(n) & \text{ if } n < 0,
\end{cases}
\]
as desired.
\end{proof}

We now carry out the last step in our calculation of~$\rk(U(\Lambda_n))$ by determining~$\rk(\im(N))$.

\begin{proposition}
\label{prop:RankIm(N)}
If $\Delta_n$ is irreducible, then
$$\rk(\im(N)) = \rk(\Z[1/n]^\times)=\Omega(n).$$
\end{proposition}
\begin{proof}
We establish the claimed equalities in succession. 
Note that by Lemma~\ref{lem:Saturation}, $\Z[1/n]^\times$ is generated by the primes in $\Z$ dividing $n$. (Since $\Z$ is a UFD, it is clear that the saturation of~$\{n^k\}$ in $\Z$ is generated by the prime factors of $n$.) If $p$ is such a prime, then $p^2$ lies in the image of $N$, since by definition
\[
N(p) = p^2.
\]
It follows that $\rk(\im(N))=\rk(\Z[1/n]^\times)$.

To prove the second equality, observe that
$$ \Z[1/n]^\times \cong \lbrace \pm 1 \rbrace \times \langle p_1,\ldots,p_{\Omega(n)}\rangle$$
where $\smash{n = p_1^{\alpha_1} \cdots p_{\Omega(n)}^{\alpha_{\Omega(n)}}}$ is the prime factorization of $n$ in $\Z$. Indeed, the linear independence of~$p_1, \ldots, p_{\Omega(n)}$ follows immediately from uniqueness of prime factorization. Hence $\rk(\Z[1/n]^\times)=\Omega(n)$, as desired.
\end{proof}

We now prove the main result of this section.

\begin{theorem}
\label{thm:IrreducibleRank}
If $\Delta_n$ is irreducible, then
\[
\operatorname{rk}(U(\Lambda_n)/\{t^k\}_{k \in \Z}) = 
\begin{cases} 
\Omega(n) & \text{ if } n > 0 \\
\Omega(n) - 1& \text{ if } n < 0.
\end{cases}
\]
In particular $U(\Lambda_n)/\{t^k\}_{k \in \Z}$ is finite if and only if either $n=-1,0,1$ or $-p^k$ with $p$ prime.
\end{theorem}
\begin{proof}
Proposition~\ref{prop:RankNullity} shows that 
$$ \rk (U(\Lambda_n))=\rk(\Z[1/n,\xi]^\times)-\rk(\im(N)).$$
Combining Proposition~\ref{prop:RankUnits} with Proposition~\ref{prop:RankIm(N)}  we deduce that 
$$ \operatorname{rk}(U(\Lambda_n))=
\begin{cases} 
\Omega(n)+1 & \text{ if } n > 0 \\
\Omega(n)& \text{ if } n < 0.
\end{cases}
$$
It remains to mod out by $\{t^k\rbrace$.
We claim that if $n\neq -1$, then $\{t^k\}_{k \in \Z}$ is a rank-one subgroup of~$U(\Lambda_n)$. 
Suppose not. Then there is some $k$ such that $t^k \equiv 1 \bmod \Delta_n$ in $\Z[t^{\pm 1}]$. This means that $t^k - 1 = \Delta_n \cdot p(t)$ for some Laurent polynomial~$p(t)$. This is only possible if $\Delta_n$ is monic, which shows $n = \pm 1$. In the case that $n = 1$, an explicit calculation of $U(\Lambda_n)$ shows that $t$ is of infinite order; see Proposition~\ref{prop:n=-101} below. 
This concludes the proof of the claim.

The claim shows that as long as $n \neq -1$, quotienting out $U(\Lambda_n)$ by $\{t^k\}_{k \in \Z}$ decreases the rank by one, which immediately gives the theorem. 
On the other hand, if $n = -1$, then  $\Omega(n)=0$, so~$U(\Lambda_n)$ is already finite and the desired equality is trivial.
\end{proof}

\subsection{The cases where $\Delta_n$ is irreducible and the count is finite} 
\label{sub:FiniteIrreducible}

Throughout this section, we continue assuming that~$\Delta_n$ is irreducible.
In Theorem~\ref{thm:IrreducibleRank} we determined the cases in which the number of disks is infinite; we now focus on the cases in which the count is finite.
We first treat the cases $n=-1,0,1$ where the computation of $U(\Lambda_n)$ is immediate, before moving on to the case~$n=-p^k$ which requires more involved arguments.
It will be helpful to keep in mind the norm map on the set of ideals of $\Z[\omega]:$

\begin{definition}
If $\mathfrak{p}$ is an ideal in $\Z[\omega]$, we define the \emph{ideal norm} of $\mathfrak{p}$ by
$$N(\mathfrak{p}):= |\Z[\omega]/\mathfrak{p}|.$$
We refer to~\cite[Section 5.3]{Stewart} for further details. Note that this is a multiplicative map from the set of ideals in $\Z[\omega]$ to $\N$ and that if $\mathfrak{p} = (x)$ is principal, then $N(\mathfrak{p}) = |N(x)|$. (Here, by $N(x)$ we mean the norm of $x$ constructed as in Definition~\ref{def:norms}.) 
\end{definition}

We now begin our analysis of $U(\Lambda_n)/\{t^k\}_{k \in \Z}$ in the finite case.

\begin{proposition}
\label{prop:n=-101}
For $n=-1,0,1$,  the group $U(\Lambda_n)/\{t^k\}_{k \in \Z}$ can be described as follows: 
$$
U(\Lambda_n)/\{t^k\}_{k \in \Z} \cong 
\begin{cases} 
1  & \text{ if } n=\unaryminus 1,0 \\
\Z/2\Z& \text{ if } n=1.
\end{cases}
$$
\end{proposition} 
\begin{proof}
First let $n = -1$, in which case~$\xi = (1 + \sqrt{-3})/2$. The ring $\Z[\xi]$ in this setting is known as the \textit{ring of Eisenstein integers}; we have
\[
\Z[\xi]^\times = \{\pm 1\} \times \{1, \xi, \xi^2\}
\]
as discussed in the case $d = -3$ of Theorem~\ref{thm:DirichletUnit}. Since $\xi \overline{\xi} = 1$, each of these units is unitary. 
Using the homomorphism $f$ from Proposition~\ref{prop:RingIrred},
we compute that
$$f(t)=\dfrac{a}{n} = \dfrac{1 - \sqrt{-3}}{2} = - \xi^2.$$
It readily follows that $U(\Lambda_n)/\{t^k\}_{k \in \Z}$ is trivial. 
Next we assume that~$n = 0$.
In this case,  the ring~$\Lambda_n =\Z[t^{\pm 1}]/(1)$ is trivial and therefore so is the group of units.

We conclude with the case $n=1$, in which case~$\xi = (1 + \sqrt{5})/2$.
Then $\xi$ itself is a fundamental unit for $\Z[\xi]$ and~$\Z[\xi]^\times = \{ \pm \xi^k \}$; see \cite[Section 7 (Exercise 2)]{Neukirch}.  
Since~$\xi \overline{\xi} = -1$, the unitary units in~$\Z[\xi]^\times$ are given by $\{\pm \xi^{2k}\}$. 
Recalling from Proposition~\ref{prop:RingIrred} that $f(t)=a/n$, the fact that~$U(\Lambda_n)/\{t^k\}_{k \in \Z} \cong \Z/2\Z$ now follows from the following calculation:
$$f(t)=\dfrac{a}{n} = \dfrac{3 + \sqrt{5}}{2} = \xi^2. $$
This concludes the proof of the proposition.
\end{proof}

Now let $n=-p^k$. Our main goal will be to explicitly determine $\Z[1/n, \xi]^\times$. However, it will be easier to first analyse the supergroup $\Z[1/n, \omega]^\times$, since Remark~\ref{rem:Recipe} can be utilized in this setting. In Lemmas~\ref{lem:idealrelations} and \ref{lem:omegagroup}, we carry out the procedure of Remark~\ref{rem:Recipe} to determine the structure of $\Z[1/n, \omega]^\times$. In Lemmas~\ref{lem:subgrouppresentation} and \ref{lem:xigroup}, we use our computation of $\Z[1/n, \omega]^\times$ to subsequently calculate $\Z[1/n, \xi]^\times$. It is then straightforward to determine the unitary units and quotient out by $\{t^k\}_{k \in \Z}$; this is done in Proposition~\ref{prop:UnitCountSquareFree}.

The first step of Remark~\ref{rem:Recipe} is to factor $(n)$ into prime ideals in $\Z[\omega]$. Lemma~\ref{lem:idealrelations} carries this out and records some identities involving this factorization which will be useful later:

\begin{lemma} \label{lem:idealrelations}
The ideals
\[
\mathfrak{p}_1 = (p, \xi) \quad \text{and} \quad \mathfrak{p}_2 = (p, \overline{\xi})
\]
of $\Z[\omega]$ are prime and
satisfy the relations
\begin{equation}\label{eq:idealrelations}
\mathfrak{p}_1^k = (\xi), \quad \mathfrak{p}_2^k = (\overline{\xi}), \quad \text{and} \quad \mathfrak{p}_1 \mathfrak{p}_2 = (p).
\end{equation}
In particular, $(n) = (-p^k) = \mathfrak{p}_1^k \mathfrak{p}_2^k$.
\end{lemma}
\begin{proof}
It will be helpful to keep in mind that $\xi = (1 + \sqrt{4n + 1})/2$ satisfies
\[
\xi^2 = n + \xi = -p^k + \xi, \quad \xi + \overline{\xi} = 1, \quad \text{and} \quad \xi \overline{\xi} = -n = p^k.
\]
We first show $\mathfrak{p}_1 \mathfrak{p}_2 = (p)$. We compute
\begin{equation}\label{eq:productp1p2}
\mathfrak{p}_1 \mathfrak{p}_2 = (p^2, p \overline{\xi}, p \xi, \xi \cdot \overline{\xi}) = (p^2, p \overline{\xi}, p \xi, p^k).
\end{equation}
This is equal to the ideal $(p)$. Indeed, all four generators of the right-hand side of \eqref{eq:productp1p2} are multiples of $p$; conversely, $p$ is the sum of the second and third generators of \eqref{eq:productp1p2}. 
It additionally follows that~$N(\mathfrak{p}_1) N(\mathfrak{p}_2) = N(p) = p^2$; hence $N(\mathfrak{p}_1) = N(\mathfrak{p}_2) = p$ and so~$\mathfrak{p}_1$ and $\mathfrak{p}_2$ are prime. Moreover,~$\mathfrak{p}_1$ and $\mathfrak{p}_2$ are distinct since (for example) the fact that $\xi + \overline{\xi} = 1$ shows $\mathfrak{p}_1 + \mathfrak{p}_2 = \Z[\omega]$. 
Note that the factorization $(p) = \mathfrak{p}_1 \mathfrak{p}_2$ is precisely the outcome of the splitting criterion applied to $(p)$; recall Theorem~\ref{thm:RamifyingSplittingInert}.

We now prove $\mathfrak{p}_1^k = (\xi)$. We begin by inductively showing that~$\mathfrak{p}_1^i = (p^i, \xi)$ for all~$i \leq k$. Indeed, assuming the inductive hypothesis, observe
\begin{equation}\label{eq:omegainduction}
\mathfrak{p}_1^i = \mathfrak{p}_1^{i - 1} \mathfrak{p}_1= (p^{i-1}, \xi)(p, \xi) = (p^i, p\xi, p^{i-1}\xi, \xi^2) =  (p^i, p\xi, p^{i-1}\xi, - p^k + \xi).
\end{equation}
We claim that the right-hand side of \eqref{eq:omegainduction} is precisely $(p^i, \xi)$. Indeed, to see \eqref{eq:omegainduction} is contained in~$(p^i, \xi)$, note that the first generator of \eqref{eq:omegainduction} is $p^i$ itself, the second and third are multiples of~$\xi$, and the last is a linear combination of $p^i$ and $\xi$. Conversely, to see $(p^i, \xi)$ is contained in \eqref{eq:omegainduction}, note that~$p^i$ is the first generator of \eqref{eq:omegainduction}, while $\xi$ is a linear combination of the first and fourth generators. (In both directions, we use the fact that $i \leq k$.) In the special case $i = k$, we moreover have
\[
\mathfrak{p}_1^k = (p^k, \xi) = (\xi \overline{\xi}, \xi) = (\xi)
\]
as desired. The claim $\mathfrak{p}_2^k = (\overline{\xi})$ is similar.
\end{proof}

According to the algorithm of Remark~\ref{rem:Recipe}, we now need to determine which ideals of the form~$\mathfrak{p}_1^{a_1} \mathfrak{p}_2^{a_2}$ are principal. The result is recorded in Lemma~\ref{lem:omegagroup} and establishes the structure of~$\Z[1/n, \omega]^\times$:

\begin{lemma}\label{lem:omegagroup}
Let $s$ be the least natural number such that $\mathfrak{p}_1^s$ and $\mathfrak{p}_2^s$ are principal\footnote{Note that since $\mathfrak{p}_1$ and $\mathfrak{p}_2$ are conjugate, this value of $s$ is the same for both.}
and write
\[
\mathfrak{p}_1^s = (y) \quad \text{and} \quad \mathfrak{p}_2^s = (\overline{y}).
\]
Then the group $\Z[1/n, \omega]^\times$ is isomorphic to
\begin{equation}\label{eq:grouppresentation}
\Z[\omega]^\times \times \langle p, y, \overline{y} \ | \ p^s = y \overline{y} \rangle.
\end{equation}
\end{lemma}
\begin{proof}
We begin by determining which ideals of the form~$\mathfrak{p}_1^{a_1} \mathfrak{p}_2^{a_2}$ are principal. Suppose $(x) = \mathfrak{p}_1^{a_1} \mathfrak{p}_2^{a_2}$ for $a_1, a_2 \geq 0$ and $x \in \Z[\omega]$. We assert that up to multiplication by a unit in $\Z[\omega]$, $x$ is a product of $p$, $y$, and $\overline{y}$. Indeed, assume $a_1 \geq a_2$ and write $a_1 - a_2 = gs + h$ for $0 \leq h < s$. Then 
\[
(x) = \mathfrak{p}_1^{a_1 - a_2} (\mathfrak{p}_1 \mathfrak{p}_2)^{a_2} = \mathfrak{p}_1^{gs + h} (\mathfrak{p}_1 \mathfrak{p}_2)^{a_2} = (y)^g \mathfrak{p}_1^h (p)^{a_2},
\]
where we have used $\mathfrak{p}_1^s = (y)$ and $\mathfrak{p}_1 \mathfrak{p}_2 = (p)$. It is straightforward to check that the product of a principal ideal with a non-principal ideal is non-principal.\footnote{In general, let $\mathfrak{a}$, $\mathfrak{b}$, and $\mathfrak{c}$ be ideals in an integral domain $R$. Assume $\mathfrak{a} = (a)$ and $\mathfrak{b} = (b)$ are principal and suppose $\mathfrak{a} = \mathfrak{b} \mathfrak{c}$. Then $a \in \mathfrak{b}\mathfrak{c}$, so $a = bc$ for some $c \in \mathfrak{c}$. We claim that $\mathfrak{c} = (c)$. To see this, let $c' \in \mathfrak{c}$ be arbitrary. Then $bc' \in \mathfrak{a}$, so $bc' = ra$ for some $r \in R$. But then $bc' = r(bc)$, which implies $c' = rc$ since $R$ is an integral domain.} Hence $\mathfrak{p}_1^h$ is principal. But then the minimality of $s$ implies $h = 0$, which shows that $x = uy^gp^{a_2}$ for some unit $u \in \Z[\omega]^\times$, as desired. The case $a_2 \geq a_1$ is similar.

The assertion, together with Lemmas~\ref{lem:Saturation} and~\ref{lem:SaturationCriterion},  shows that $\Z[1/n, \omega]^\times$ is generated by $p$, $y$, and $\overline{y}$, together with the units of $\Z[\omega]$. The relation $p^s = y \overline{y}$ holds, since
\[
y \overline{y} = N(y) = N(\mathfrak{p}_1)^s = p^s.
\]
It remains to show that there are no other relations. This follows from the same linear independence argument as in Proposition~\ref{prop:NotUFD}. Indeed, suppose that we had a positive-exponent linear relation between $p$, $y$, and $\overline{y}$. Passing to ideals, we obtain a positive-exponent linear relation between $(p) = \mathfrak{p}_1 \mathfrak{p}_2$, $(y) = \mathfrak{p}_1^s$, and $(\overline{y}) = \mathfrak{p}_2^s$. By unique factorization of ideals in $\Z[\omega]$, it is clear that every linear relation between $(p)$, $(y)$, and $(\overline{y})$ is a multiple of the fundamental relation $(p)^s = (y)(\overline{y})$. This completes the proof.
\end{proof}

We have now computed $\Z[1/n, \omega]^\times$. The next step is to determine the subgroup $\Z[1/n, \xi]^\times$. 
 Clearly, $\Z[1/n, \xi]^\times$ contains the elements $\pm 1$, $p$, $\xi$, and $\overline{\xi}$, since these are all in the saturation of $\{n^k\}$ in $\Z[\xi]$. (Note that $\xi \overline{\xi} = -n = p^k$.) Hence $\Z[1/n, \xi]^\times$ certainly contains the subgroup of $\Z[1/n, \omega]^\times$ generated by $\pm 1$, $p$, $\xi$, and $\overline{\xi}$. The structure of this subgroup is given by the following:

\begin{lemma}\label{lem:subgrouppresentation}
The subgroup of $\Z[1/n, \omega]^\times$ generated by $\pm 1$, $p$, $\xi$, and $\overline{\xi}$ has presentation
\begin{equation}\label{eq:subgrouppresentation}
\{\pm 1\} \times \langle p, \xi, \overline{\xi} \ | \ \xi \overline{\xi} = p^k\rangle.
\end{equation}
\end{lemma}
\begin{proof}
The claim is almost a tautology, but crucially we must verify that $\xi \overline{\xi} = p^k$ is the only relation. To see this, we appeal to the structure of the ambient group $\Z[1/n, \omega]^\times$ established in Lemma~\ref{lem:omegagroup}. Observe that the natural number $s$ from Lemma~\ref{lem:omegagroup} must divide $k$. Indeed, write $k = gs + h$ with $0 \leq h < s$; then $(\xi) = \mathfrak{p}_1^k = \mathfrak{p}_1^{gs}\mathfrak{p}_1^h = (y)^g \mathfrak{p}_1^h$. Since the product of a principal ideal with a non-principal ideal is non-principal, we have that $\mathfrak{p}_1^h$ is principal. The minimality of $s$ then forces $h = 0$. We thus see that
\[
\xi = u y^{k/s}
\]
for some unit $u \in \Z[\omega]^\times$. The fact that $\xi \overline{\xi} = p^k$ is the only relation then follows from the fact that all relations between $p$, $y$, $\overline{y}$, and $\Z[\omega]^\times$ are given in \eqref{eq:grouppresentation}.
\end{proof}

We now come to the computation of $\Z[1/n, \xi]^\times$:

\begin{lemma}\label{lem:xigroup}
The unit group $\Z[1/n, \xi]^\times$ is the subgroup of $\Z[1/n, \omega]^\times$ 
given by
$$\Z[1/n, \xi]^\times = \{\pm 1\} \times \langle p, \xi, \overline{\xi} \ | \ \xi \overline{\xi} = p^k\rangle.$$
\end{lemma}
\begin{proof}
It suffices to prove that $\Z[1/n, \xi]^\times$ is generated by $\pm 1$, $p$, $\xi$, and $\overline{\xi}$ as the conclusion will then follow from Lemma~\ref{lem:subgrouppresentation}.
As noted previously, it is clear that $\Z[1/n, \xi]^\times$ contains \eqref{eq:subgrouppresentation}, so it suffices to prove the reverse containment. Let $x \in \Z[1/n, \xi]^\times$ be arbitrary. Since $\Z[1/n, \xi]^\times \subset \Z[1/n, \omega]^\times$, Lemma~\ref{lem:omegagroup} shows that we may write
\[
x = u p^g y^h
\]
for some $u \in \Z[\omega]^\times$ and $g, h \in \Z$. Recall from the proof of Lemma~\ref{lem:subgrouppresentation} that the exponent $s$ of Lemma~\ref{lem:omegagroup} divides $k$ and that
\[
u' \xi = y^{k/s}
\]
for some $u' \in \Z[\omega]^\times$. Write $h = (k/s) t + r$ for $0 \leq r < k/s$. Then
\[
x = u p^g y^{(k/s)t + r} = u'' p^g \xi^t y^r = (u'' y^r) (p^g \xi^t)
\]
for $u'' = u (u')^t \in \Z[\omega]^\times$. Note that $p^g \xi^t$ already lies in \eqref{eq:subgrouppresentation}. It thus suffices to prove that $u'' y^r$ must be $\pm 1$. For this, we will utilize the following technical claim:

\begin{claim}
\label{claim:NotPrincipal}
Let $z$ be an element of $\Z[\omega]$ with 
\[
N(z) = p^i
\]
for $0 \leq i < k$. Suppose moreover that $p^\ell z \in \Z[\xi]$ for some $\ell \geq 0$. Then $i$ is even and $z = \pm p^{i/2}$.
\end{claim}
\begin{proof}
Write $z = a + b \omega$ for $a, b \in \Z$. 
First, we analyze the consequences of the condition $p^\ell z \in \Z[\xi]$, before moving on to the condition $N(z) = p^i$.

The requirement that $p^\ell z \in \Z[\xi]$ gives
\[
p^\ell a + p^\ell b \omega = e + f \xi
\]
for $e, f \in \Z$. Using the definitions of $\omega$ and $\xi$, we compute
\[ 
\xi = \dfrac{1 + c\sqrt{d}}{2} = c \left(\dfrac{1 + \sqrt{d}}{2} \right) - \dfrac{c-1}{2} = c \omega - \dfrac{c-1}{2}.
\]
Substituting this expression into the right-hand side above and matching coefficients of $\omega$ shows that $p^\ell b = cf$. Hence $c$ divides $p^\ell b$. Since $-4 p^k + 1 = c^2 d$, we have that $c$ and $p$ are coprime; thus~$c$ divides~$b$. 
We deduce that, for some $g \in \Z$, we have
\[
b = gc.
\]
We now consider that for $z=a+b\omega$,  our assumption that $N(z)=p^i$ is equivalent to
\begin{equation}\label{eq:Diophantine}
a^2 + ab + b^2 \left(\dfrac{1-d}{4}\right) = p^i
\end{equation}
over~$\Z$. The discriminant of \eqref{eq:Diophantine} (viewed as a quadratic equation in $a$) is given by
\[
b^2 - 4\left(b^2 \left(\dfrac{1-d}{4}\right) - p^i\right) = b^2 d + 4p^i = g^2(-4p^k + 1) + 4p^i
\]
where in the last equality we have substituted $b = gc$ and used the fact that $-4 p^k + 1 = c^2 d$. 
In order for $N(a + b \omega) = p^i$, the discriminant of \eqref{eq:Diophantine} must certainly be non-negative. 
However, if $0 \leq i < k$, then this is readily seen to be impossible unless $g = 0$. 
It follows that $(a, b) = (\pm p^{i/2}, 0)$,  and so $z=a+b\omega=\pm p^{i/2} $,  concluding the proof of the claim.
\end{proof}
We now verify that $z = u'' y^r$ satisfies the hypotheses of the claim. Indeed,
\[
N(z) = N(y)^r = p^{rs}.
\]
Since $0 \leq r < k/s$, we have $0 \leq rs < k$. Moreover, recall that $x = z(p^g \xi^t)$. Since $x$ and $p^g \xi^t$ both lie in $\Z[1/n, \xi]^\times$, it follows that $z$ does as well. Hence $z$ is equal to an element of $\Z[\xi]$ divided by some power of $n = -p^k$. The claim then shows that $z = \pm p^{(rs)/2}$, which implies
\[
u'' y^r = \pm p^{(rs)/2}.
\] 
However, the fact that there are no nontrivial relations between $y$, $p$, and $\Z[\omega]^\times$ in $\Z[1/n, \omega]^\times$ then forces $z = \pm 1$. 
Thus $x=(u'' y^r) (p^g \xi^t)=\pm p^g \xi^t$ and
we conclude that $x$ lies in the subgroup~\eqref{eq:subgrouppresentation}, completing the proof.
\end{proof}

Putting everything together gives the desired count of units:

\begin{proposition}\label{prop:UnitCountSquareFree}
For $n=-p^k$,  the group $U(\Lambda_n)/\{t^k\}_{k \in \Z}$ can be described as follows: 
$$
U(\Lambda_n)/\{t^k\}_{k \in \Z} \cong 
\begin{cases} 
\Z/2\Z  & \text{ if $k$ is odd } \\
\Z/4\Z &   \text{ if $k$ is even.}
\end{cases}
$$
\end{proposition}
\begin{proof}
Lemma~\ref{lem:xigroup} states that
\[
\Z[1/n, \xi]^\times = \{\pm 1\} \times \langle p, \xi, \overline{\xi} \ | \ \xi \overline{\xi} = p^k\rangle.
\]
We now determine $U(\Lambda_n)$ by understanding the kernel of~$N$: indeed, recall from Propositions~\ref{prop:RingIrred} and~\ref{prop:RankNullity} that the isomorphism $f$ induces a group isomorphism 
$$U(\Lambda_n) \cong U(\Z[1/n, \xi]) = \ker(N).$$
Note that~$N(p) = p^2$, while~$N(\xi) = N(\overline{\xi}) = p^k$. 
It follows that there are two possible cases. If~$k$ is odd, then~$U(\Lambda_n)$ is generated (up to multiplication by~$\pm 1$) by~$\xi/\overline{\xi}$. 
If~$k$ is even, then~$U(\Lambda_n)$ has a pair of natural potential generators, given by~$\xi/\overline{\xi}$ and~$\xi/p^{k/2}$. 
However, the square of the second generator is given by $\xi^2/p^k = \xi^2/(\xi \overline{\xi})$, and is thus equal to the first. 
Hence in both cases, we have~$U(\Lambda_n) = \{\pm 1\} \times \Z$, but there is a slight difference in the generating element: when $k$ is odd $\xi/\overline{\xi}$ generates, whereas when $k$ is even $\xi/p^{k/2}$ generates.
An explicit calculation shows that the isomorphism~$f$ satisfies 
\[
f(t)= \dfrac{2n +1 + \sqrt{4n +1}}{2n} = \dfrac{\xi^2}{n} = - \dfrac{\xi}{\overline{\xi}}.
\]
Hence in the odd case~$U(\Lambda_n)/\{t^k\}_{k \in \Z}$ is readily seen to be isomorphic to $\Z/2\Z$, while in the even case it is isomorphic to~$\Z/4\Z$. 
\end{proof}

\subsection{The case where $\Delta_n$ is reducible}
\label{sub:ReducibleUnitary}

We now turn to the case when $\Delta_n$ is reducible. It is straightforward to check that in this setting,
\[
\Delta_n = ((m+1)t - m)(mt - (m+1))
\]
where $n = m(m+1)$ and $m \in \N$; recall Remark~\ref{rem:Reducibility}.

Our overall strategy will be similar to the irreducible case, except that a different ring will take the place of $\Z[1/n, \xi]$.
The first goal will be to describe the ring~$\Lambda_n$ when~$\Delta_n$ is reducible.
The outcome,  which is stated in Proposition~\ref{prop:ImageSubring}, is that~$\Lambda_n$ can be identified with a certain subring of~$\Z[1/n] \oplus \Z[1/n]$.
Thanks to this identification we are able to identify $U(\Lambda_n)$ with a concrete subset of~$\Z[1/n] \oplus \Z[1/n]$ (Proposition~\ref{prop:UnitaryUnitsReducible}) and calculate its rank (Proposition~\ref{prop:ReducibleRank}).
The outcome is that $U(\Lambda_n)/\{ t^k \}$ is finite if and only if $n=2$ in which case $U(\Lambda_n)/\{ t^k \} \cong \Z/2\Z$ (Proposition~\ref{prop:n=2}).

\begin{construction}
\label{cons:IsoReducible}
We construct a ring homomorphism
$$f=f_1 \oplus f_2\colon \Z[t^{\pm 1}] \rightarrow \Z[1/n] \oplus \Z[1/n]$$
by extending the following assignments
\[
f_1(t) = \dfrac{m}{m+1} \quad \text{and} \quad f_2(t) = \dfrac{m+1}{m}.
\]
Note that both~$f_1$ and~$f_2$ are individually surjective since~$f_i(t + t^{-1} - 2)=1/n$ for $i=1,2$.
\end{construction}

\begin{definition} 
Endow the ring
$\Z[1/n] \oplus \Z[1/n]$ with the conjugation involution~$\smash{\overline{(a,b)} = (b, a)}$. This defines a norm map
\begin{align*}
N \colon  (\Z[1/n] \oplus \Z[1/n])^\times&\to (\Z[1/n] \oplus \Z[1/n])^\times  \\
(a,b) &\mapsto (a,b)\overline{(a,b)} = (ab, ab)
\end{align*}
which is readily seen to be a group homomorphism.
\end{definition}
The next lemma uses the homomorphism $f$ to describe $\Lambda_n$ as a subring of $\Z[1/n] \oplus \Z[1/n]$.

\begin{lemma}
\label{lem:SubringZ1n}
The homomorphism $f$ from Construction~\ref{cons:IsoReducible} has kernel $(\Delta_n)$ and thus induces an injective homomorphism
$$ f\colon \Lambda_n \hookrightarrow \Z[1/n] \oplus \Z[1/n].$$
Moreover, $f$ intertwines the norms on $\Lambda_n$ and $\Z[1/n] \oplus \Z[1/n]$ and therefore fits into the following commutative diagram of groups: 
$$
\centering
\begin{tikzcd}
0 \arrow{r} 
  & U(\Lambda_n)  \arrow{r}\arrow[hookrightarrow]{d}{f} 
  & \Lambda_n^\times \arrow{r}{N_\Lambda}\arrow[hookrightarrow]{d}{f} 
  & \Lambda_m^\times \arrow{r}\arrow[hookrightarrow]{d}{f} 
  & 0 \\
0 \arrow{r} 
  &  U(\Z[1/n] \oplus \Z[1/n]) \arrow{r}
  & \Z[1/n]^\times \oplus \Z[1/n]^\times\arrow{r}{N} 
  &  \Z[1/n]^\times \oplus \Z[1/n]^\times\arrow{r}
  & 0 
\end{tikzcd}
$$
\end{lemma}

\begin{proof}
A short verification shows that
\[
\ker (f_1) = ((m+1)t - m) \quad \text{and} \quad \ker (f_2) = (mt - (m+1)).
\]
Since~$\Z[t^{\pm 1}]$ is a UFD and the polynomials~$(m+1)t - m$ and~$mt - (m+1)$ have no common factors,  we deduce that 
\begin{align*}
\ker(f) &= \ker(f_1) \cap \ker(f_2) = \left( ((m+1)t - m))(mt - (m+1)) \right ) = (\Delta_n).
\end{align*}
It is straightforward to that $f$ intertwines conjugation on $\Z[t, t^{-1}]$ with conjugation on $\Z[1/n] \oplus \Z[1/n]$. The claim then follows as in the proof of Proposition~\ref{prop:RankNullity}.
\end{proof}

The following proposition determines the image of~$f$, thus concluding our description of~$\Lambda_n$ as a subring of~$\Z[1/n] \oplus \Z[1/n]$.

\begin{proposition}
\label{prop:ImageSubring}
The image of~$f$ agrees with the following subring of~$\Z[1/n] \oplus \Z[1/n]$:
\[
S = \{(a, b) \in \Z[1/n] \oplus \Z[1/n] \ | \ b - a \in (2m+1) \cdot \Z[1/n] \text{ for some } m \in \Z\}.
\]
\end{proposition}
\begin{proof}
Let $D$ be the diagonal $\{(a,a)\}$ of $\Z[1/n] \oplus \Z[1/n]$. Note that $D \subset \im(f)$: this follows by observing that $(1/n, 1/n)=f(t + t^{-1} - 2)$ and taking the image of linear combinations of powers of~$t + t^{-1} - 2$.

We claim that for every $b \in \Z[1/n]$,  we have $(0,b) \in \im(f)$ if and only if $(0,b) \in S$.
Suppose~$(0, b) = f(p)$ for some~$p \in \Lambda$. Then~$p \in \ker(f_1)$, so~$p$ is a multiple of~$(m+1)t - m$. Now,
\[
f((m+1)t - m) = \left(0, (m+1)\dfrac{m+1}{m} - m \right) = \left(0, \dfrac{2m+1}{m}\right).
\]
Thus,~$b$ is~$(2m+1)/m$ times an element of~$\Z[1/n]$. As~$1/m = (m+1)/n$, we see that~$b$ is~$2m+1$ times an element of~$\Z[1/n]$, as desired. Conversely, suppose that~$b = (2m+1) \cdot c$ for some~$c \in \Z[1/n]$. Let~$p$ be the linear combination of powers of~$t + t^{-1} - 2$ mapping to~$(c, c) \in D$. Then~$f$ maps~$((m+1)t - m) \cdot mp$ to~$(0, b)$, as desired.
This concludes the proof of the claim.

We now conclude that $\im(f)=S$.
Since $D \subset \im(f)$, subtracting off $(a, a) \in D$ shows that~$(a, b) \in \im(f)$ if and only if~$(0, b - a) \in \im(f)$.
Thanks to the claim, this is equivalent to~$(0,b-a) \in S$.
 Since the condition of being in~$S$ is clearly invariant under adding elements of~$D$,  this is in turn equivalent to $(a,b) \in S$.
\end{proof}

Now that we have described~$\Lambda_n$ as a subring of~$\Z[1/n] \oplus \Z[1/n]$, we are able to understand its group of unitary units using the injection $f \colon \Lambda \hookrightarrow \Z[1/n]\oplus \Z[1/n]$ and the projection $\operatorname{pr}_1 \colon \Z[1/n] \oplus \Z[1/n] \to \Z[1/n]$ onto the first component.

\begin{proposition}
\label{prop:UnitaryUnitsReducible}
If~$\Delta_n$ is reducible,  then the composition $\operatorname{pr}_1 \circ f$ induces an isomorphism
\[
U(\Lambda_n) \cong T := \{x \in \Z[1/n] \ | \ x \text{ is a unit in } \Z[1/n] \text{ and } x - x^{-1} \in (2m+1) \cdot \Z[1/n]\}.
\]
\end{proposition}
\begin{proof}
Since the norm on $(\Z[1/n] \oplus \Z[1/n])^\times$ is given by $N(a,b)=(ab,ab)$,
Proposition~\ref{prop:ImageSubring} implies that $f$ induces a group isomorphism 
\begin{align*}
U(\Lambda) &\cong  \{(a,b) \in S^\times \ | \ N(a, b) = (ab, ab) = (1, 1)\} \\
&= \{(x,x^{-1}) \in \Z[1/n] \oplus \Z[1/n] \ | \ x \text{ is a unit in } \Z[1/n] \text{ and } x - x^{-1} \in (2m+1) \cdot \Z[1/n]\}. 
\end{align*}
The proposition now follows because $\operatorname{pr}_1 $ maps this group isomorphically onto $T$.
\end{proof}

We now complete the proof of Theorem~\ref{thm:MainNumberTheory} for reducible~$\Delta_n$ and $n \neq 2.$

\begin{proposition}\label{prop:ReducibleRank}
If~$\Delta_n$ is reducible,  then $\operatorname{rk}(U(\Lambda_n)/\{t^k\}_{k \in \Z}) = \Omega(n) - 1$. 
In particular $\operatorname{rk}(U(\Lambda_n)/\{t^k\}_{k \in \Z})$ is finite if and only if $n=2$.
\end{proposition}
\begin{proof}
We first show that
\[
\operatorname{rk}(T) = \operatorname{rk}(\Z[1/n]^\times) = \Omega(n).
\]
The second equality was established in Proposition~\ref{prop:RankIm(N)}. Since we proved in Proposition~\ref{prop:RankIm(N)} that
\[
\Z[1/n]^\times \cong \lbrace \pm 1 \rbrace \times \langle p_1,\ldots,p_{\Omega(n)} \rangle, 
\]
where $\smash{n = p_1^{\alpha_1},\ldots,p_{\Omega(n)}^{\alpha_{\Omega(n)}}}$, the first equality immediately follows from the claim:
\begin{claim}
There exists a $g \in \N$ such that for each prime divisor~$p_i$ of $n$, we have~$\smash{p_i^{g} \in T}$.
\end{claim}
\begin{proof}
We need to show that there exists a $g \in \N$ such that $p_i^g \in T$ for every $i=1,\ldots r$.
By definition of $T$, this is equivalent to proving the existence of a~$g \in \N$ such that~$p_i^g-(1/p_i^g) \in (2m+1)\Z[1/n]$ for each $i$.
Observe that
\[
p_i^{g} - \frac{1}{p_i^{g}} = \frac{p_i^{2g} - 1}{p_i^{g}}.
\]
Note that since $p_i^g$ is always in the saturation of $\{n^k\}$, we have $1/p_i^g \in \Z[1/n]$. It thus suffices to show that there exists a~$g \in \N$ such that $p_i^{2g} - 1 \in (2m+1)\Z[1/n]$ for each $i$; i.e.~$p_i^{2g} \equiv 1$ mod $2m+1$.
Now, $p_i$ is a divisor of either $m$ or $m+1$, both of which are coprime to $2m + 1$. Hence $p_i$ is coprime to $2m+1$,  we can take $g$ to be the cardinality of~$(\Z/(2m+1)\Z)^\times$.
This concludes the proof of the claim.
\end{proof}

Now that we have calculated $\rk(U(\Lambda_n))$, we turn to $U(\Lambda_n)/\{t^k\}_{k \in \Z}$. Tracing through the definition of $f$ and the isomorphism of Proposition~\ref{prop:UnitaryUnitsReducible}, we see that $\{t^k\}_{k \in \Z}$ corresponds to the subgroup
\[
\left\langle \dfrac{m}{m+1}\right\rangle \leq T.
\]
Note that since $T \subset \Z[1/n]$, the image of~$t$ has infinite order and so $\rk (\{t^k\}_{k \in \Z})=1$.
It follows that~$U(\Lambda_n)/\{t^k\}_{k \in \Z}$ has rank~$\Omega(n) - 1$, as desired.
The last sentence of the proposition follows because if~$m > 1$, then $n = m(m+1)$ has at least two distinct prime factors.
\end{proof}

Thanks to Proposition~\ref{prop:ReducibleRank}, it only remains to determine the count when $n=2$.

\begin{proposition}
\label{prop:n=2}
For $n=2$, the group~$U(\Lambda_n)/\{t^k\}_{k \in \Z}$ is isomorphic to $\Z/2\Z$.
\end{proposition}
\begin{proof}
We calculate this quotient explicitly under the isomorphism of Proposition~\ref{prop:ReducibleRank}, starting with $T$ and then quotienting out by the image of $\{ t^k \}$.
For $n=2$, we have that~$\Z[1/n]^\times$ is generated (up to multiplication by~$\pm 1$) by~$2$. 
Since~$2 - 1/2 = 3/2 \in 3 \cdot \Z[1/n]$,  we deduce that $U \cong T = \Z[1/n]^\times$. 
Since the isomorphism $\operatorname{pr}_1 \circ f$ from Proposition~\ref{prop:ReducibleRank} satisfies~$\operatorname{pr}_1 \circ f(t)=1/2$, the quotient~$U/\{t^k\}_{k \in \Z}$ is isomorphic to $\Z/2\Z$.
\end{proof}

\def\MR#1{}
\bibliography{BiblioCP2}
\end{document}